\documentclass[10pt,a4paper]{article}

\usepackage{amsthm}
\usepackage{amsmath}
\usepackage{amssymb}
\usepackage{empheq}
\usepackage{stackrel}
\usepackage{color}
\usepackage{framed}
\usepackage{multirow}
\usepackage{float}
\usepackage{caption}
\usepackage{hyperref}

\newcommand{\ds}{\displaystyle}

\newcommand{\lt}{\left}
\newcommand{\rt}{\right}

\def\Vol{\operatorname{Vol}}
\def \CC{\mathbb C}
\def \Re{\operatorname{Re}}

\newtheorem{conjecture}{Conjecture}
\newtheorem{theorem}{Theorem}
\newtheorem{lemma}{Lemma}

\newtheorem{corollary}{Corollary}

\newtheorem{question}{Question}

\def\SS{{\mathbb S}}

\def\CC{{\mathbb C}}
\def\RR{{\mathbb R}}
\def\ZZ{{\mathbb Z}}
\def\NN{{\mathbb N}}
\def\DD{{\mathbb D}}

\def\Vol{\operatorname{Vol}}
\def\CS{\operatorname{CS}}
\def \Re{\operatorname{Re}}
\def \im{\operatorname{Im}}

\def \Li{\operatorname{Li_2}}
\def \im{\operatorname{Im}}

\def \det{\operatorname{det}}
\def \Hess{\operatorname{Hess}}

\title{Volume conjecture, geometric decomposition and \\deformation of hyperbolic structures}

\author{Ka Ho WONG}

\date{}

\begin{document}

\maketitle

\begin{abstract}
In this paper, we study the generalized volume conjecture for the colored Jones polynomials of links with complements containing more than one hyperbolic piece. First of all, we construct an infinite family of prime links by considering the cabling on the figure eight knot by the Whitehead chains. The complement of these links consist of two hyperbolic pieces in their JSJ decompositions. We show that at the $(N+\frac{1}{2})$-th root of unity, the exponential growth rates for the $\vec{N}$-th colored Jones polynomials for these links capture the simplicial volume of the link complements. As an application, we prove the volume conjecture for the Turaev-Viro invariants for these links complements. 

We also generalize the volume conjecture for links whose complement have more than one hyperbolic piece in another direction. By considering the iterated Whitehead double on the figure eight knot and the Hopf link, we construct two infinite families of prime links. Furthermore, we assign certain ``natural'' incomplete hyperbolic structures on the hyperbolic pieces of the complements of these links and prove that the sum of their volume coincides with the exponential growth rate of certain sequences of values of colored Jones polynomials of the links. 

\end{abstract}

\section{Introduction}
This paper aims to study the asymptotics of the $\vec{M}$-th colored Jones polynomials at the $(N+\frac{1}{2})$-th root of unity for links with more than one hyperbolic pieces in their JSJ decompositions, where $\vec{M}$ is a sequence of multi-color in $N$. Our results combine the volume conjecture for non-hyperbolic links and the generalized volume conjecture for hyperbolic links. First, by considering the cabling of the figure eight knot by the Whitehead chains, we obtain an infinite family of prime links with $2$ hyperbolic pieces in their complement in $\SS^3$. Furthermore, we prove that this infinite family of links satisfies the volume conjecture at $(N+\frac{1}{2})$-th root of unity. Besides, by studying the iterated Whitehead double on the figure eight knot and Hopf link, we establish the relationship between the limiting ratio of the multi-color $\ds \vec{s} = \lim_{N\to \infty}\frac{\vec{M}}{N+\frac{1}{2}}$ and the exponential growth rates of the corresponding colored Jones polynomials. Based on these results, we propose the generalized volume conjecture for prime links whose complement has only hyperbolic pieces under the JSJ decomposition. 

\subsection{Brief review of volume conjectures}

The volume conjecture relates quantum invariants of knots and 3-manifolds with their topological and geometrical structures. For hyperbolic knots, the volume conjecture suggests that the exponential growth rate of the $N$-th normalized colored Jones polynomials of a hyperbolic knot evaluated at the $N$-th root of unity captures the hyperbolic volume of the knot complement. 

\begin{conjecture}\label{cvcK}\cite{K97,MM01} Let $K$ be a hyperbolic knot and $J'_{N}(K;t)$ be the $N$-th normalized colored Jones polynomials of $K$ evaluated at $t$. We have
\begin{align*}
\lim_{N \to \infty} \frac{2\pi}{N} \log|J'_{N}(K;e^{\frac{2\pi i}{N}})| = \Vol(\SS^{3} \backslash K),
\end{align*}
where $\operatorname{Vol}(\SS^{3} \backslash K)$ is the hyperbolic volume of the knot complement, and the normalization is choosen so that for the unknot $U$, $J'_N(U,t) = 1$ for any $N\in \NN$.
\end{conjecture}

Later, J. Murakami and H. Murakami extended the above volume conjecture to non-hyperbolic links and suggested that the exponential growth rate captures the simplicial volume of the link complement.

\begin{conjecture}\label{cvcMM}\cite{MM01} Let $L$ be a link and $J'_{\vec{N}}(L;t)$ be the $\vec{N}$-th normalized colored Jones polynomials of $L$ evaluated at $t$, where $\vec{N} = (N,\dots, N)$. We have
\begin{align*}
\lim_{N \to \infty} \frac{2\pi}{N} \log|J'_{\vec{N}}(L;e^{\frac{2\pi i}{N}})| = v_3||\SS^3 \backslash L||,
\end{align*}
where $||\SS^3 \backslash L||$ is the simplicial volume of the link complement.
\end{conjecture}

Besides, for 3-manifolds, Q. Chen and T. Yang propose a version of volume conjecture for the Turaev-Viro invariants of finite volume hyperbolic 3-manifolds with or without boundary.

\begin{conjecture}\label{vctv}\cite{CY15}
For every hyperbolic 3-manifold $M$ with finite volume, we have
\begin{align*}
\lim_{\substack{ r\to \infty \\r \text{ odd}}}\frac{2\pi}{r}\log \left(TV_{r}(M,e^{\frac{2\pi i}{r}})\right) = \Vol(M)
\end{align*}
\end{conjecture}

Let $J_{\vec{M}}(L,t)$ be the unnormalized colored Jones polynomials of the link $L$, i.e.
$$J_N(U,t) = \frac{t^{N/2} - t^{-N/2}}{t^{\frac{1}{2}}-t^{-\frac{1}{2}}}$$
for the unknot $U$ and any $N \in \NN$. It turns out that when the 3-manifold is the complement of a link $L$ in $\SS^3$, the Turaev-Viro invariants of the manifold $\SS^3\backslash L$ is closely related to the colored Jones polynomials $J_{\vec{M}}(L,t)$ of the link $L$ as follows. 

\begin{theorem}\label{relationship} \cite{DKY17}
Let L be a link in $\SS^3$ with n components. Then given an odd integer $r=2N+1 \geq 3$, we have
\begin{align*}
TV_{r}\left(\SS^3 \backslash L, e^{\frac{2\pi i}{r}}\right) = 2^{n-1}\lt( \frac{2\sin(\frac{2\pi}{r})}{\sqrt{r}}\rt)^{2} \sum_{1\leq \vec{M} \leq \frac{r-1}{2}}\left|J_{\vec{M}}\left(L,e^{\frac{2\pi i }{N+\frac{1}{2}}}\right)\right|^2
\end{align*}
\end{theorem}

Theorem~\ref{relationship} provides a bridge between the volume conjecture of the colored Jones polynomials of links and the volume conjecture of the Turaev-Viro invariants of link complements. In particular, in \cite{DKY17}, the volume conjecture of the Turaev-Viro invariants is extended to non-hyperbolic link complement and this generalization is proved for all knots with zero simplicial volume. 

\begin{conjecture}\label{tvGM}\cite{DKY17}
For every link $L$ in $\SS^3$, we have
$$  \lim_{\substack{ r\to \infty \\r \text{ odd}}} \frac{2\pi }{r} \log(TV_r(\SS^3 \backslash L, e^{\frac{2\pi i}{r}})) = v_3 ||\SS^3 \backslash L||$$
\end{conjecture}

Due to Theorem~\ref{relationship}, it is natural to study the analogue of Conjecture~\ref{cvcK} and~\ref{cvcMM} at the new root $t=e^{\frac{2\pi i}{N+\frac{1}{2}}}$. In \cite{DKY17}, R. Detcherry, E. Kalfagianni and T. Yang asked the following question,  
\begin{conjecture}\label{NVC}(Question 1.7 in \cite{DKY17}) Is it true that for any hyperbolic link $L$ in $\SS^3$, we have
$$ \lim_{N\to \infty} \frac{2\pi}{N+\frac{1}{2}} \log|J_{\vec{N}}(L,e^{\frac{2\pi i}{N+\frac{1}{2}}}) | = \Vol(\SS^3 \backslash L) \quad ?$$
\end{conjecture}

Similar to Conjecture~\ref{cvcMM}, we extend the above conjecture to non-hyperbolic links as follows.

\begin{conjecture}\label{NVCMM} Let $L$ be a link and $J_{\vec{N}}(L,t)$ be the $\vec{N}$-th colored Jones polynomials of $L$ evaluated at $t$. We have
\begin{align*}
\lim_{N \to \infty} \frac{2\pi}{N+\frac{1}{2}} \log|J_{\vec{N}}(L;e^{\frac{2\pi i}{N+\frac{1}{2}}})| = v_3||\SS^3 \backslash L||,
\end{align*}
where $||\SS^3 \backslash L||$ is the simplicial volume of the link complement.
\end{conjecture}

Another direction of generalizing the volume conjecture is to study the asymptotic behavior of the values of the $N$-th colored Jones polynomials near the $N$-th root of unity. It is expected that the exponential growth rate has to do with the volume of the link complement with certain incomplete hyperbolic metric \cite{G05, GM08, HM13}. Unfortunately, at $t=e^{\frac{2\pi i}{N}}$, this generalization is known to be false in general for the colored Jones polynomials of links in $\SS^3$ (see for example \cite{W19}). For knots, due to the choice of the root of unity, the values of the $N$-th colored Jones polynomials of the figure eight knot grow only polynomially at $t=e^{\frac{2\pi i r}{N}}$ for rational number $r \in (\frac{5}{6}, \frac{7}{6})$ \cite{MY07}. Nevertheless, if we allow the ambient manifold to be the connected sum of $\SS^2 \times \SS^1$, this generalization is known to be true for the fundamental shadow links \cite{C07}. 

Recently, in \cite{W19}, the author proposes the following version of generalized volume conjecture at the new root of unity $t=e^{\frac{2\pi i}{N+\frac{1}{2}}}$. In particular, Conjecture~\ref{GVC} implies Conjecture~\ref{NVC}. 

\begin{conjecture}\label{GVC}(Generalized volume conjecture for link) 
Let $L$ be a hyperbolic link with $n$ component. Let $M_1,M_2,\dots, M_n$ be sequence of positive integers in $N$. Let $s_i =\lim_{N\to \infty} \frac{M_i}{N+\frac{1}{2}}$. Then there exists $\delta_L > 0$ such that whenever the limiting ratio $s_i \in (1-\delta_L, 1]$, we have
$$ \lim_{N \to \infty} \frac{2\pi }{N+\frac{1}{2}} \log|J_{M_1,M_2,\dots, M_n}(L, e^{\frac{2\pi i}{N+\frac{1}{2}}}) | = \Vol\lt(\SS^3 \backslash L, u_i = 2\pi i (1 - s_i)\rt), $$
where $ \Vol\lt(\SS^3 \backslash L, u = 2\pi i (1 - s)\rt)$ is the volume of $\SS^3 \backslash L$ equipped with the hyperbolic structure such that the logarithm of the holonomy of the meridian around the $i$-th component is given by $u_i = 2\pi i (1-s_i)$. 
\end{conjecture}

Conjecture~\ref{GVC} has been proved by Thomas Au and the author in the case of the figure eight knot \cite{WA17} and proved by the author for some special cases for the Whitehead chains \cite{W19}. So far there are limited results about Conjecture~\ref{cvcMM} and~\ref{NVCMM} for prime links whose complement have more than one hyperbolic piece. To the best of the author's knowledge, the only known result for Conjecture~\ref{cvcMM} is the cabling of the figure eight knot by the Borromean rings \cite{YY10}. 

\subsection{Construction of the links}

Since the generalized volume conjecture is only well-established for the figure eight knot $4_1$ \cite{WA17} and some special cases for Whitehead chains $W_{a,b,c,d}$ \cite{W19}, in order to provide new example of links that may satisfy (some sort of) generalized volume conjecture, it is natural to do some operations on these links to construct new families of links and study the asymptotics of the corresponding colored Jones polynomials. In this paper, the operations that we are going to use are the cabling by the Whitehead chains, the iterated Whitehead double and the Hopf union.

          \begin{figure}[H]
          \centering
              \includegraphics[width=0.75\linewidth]{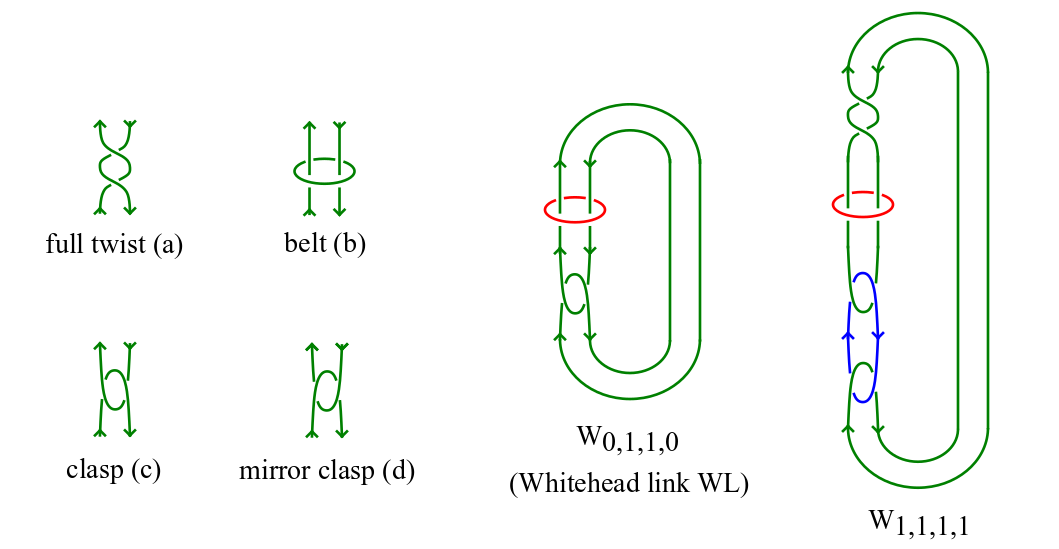}
              \caption{Left: the positive full twist, belt, clasp and mirror clasp tangles, when $a<0$, we twist the tangles in the opposite direction. Right: the links $W_{0,1,1,0}$ (also denoted by $WL$) and $W_{1,1,1,1}$}\label{fullpicture}
          \end{figure}

For $a\in \ZZ, b \in \NN, c,d\in \NN\cup\{0\}$ with $c+d \geq 1$, recall that the Whitehead chains $W_{a,b,c,d}$ (defined in \cite{V08}) is obtained by stacking $a$ full twists, $b$ belts, $c$ clasps and $d$ mirror clasps together and then taking the closure (Figure~\ref{fullpicture}).

The first family of links $W_{a,1,c,d}(4_1)$ is obtained as follows. Recall that $\SS^3$ has a standard genus one Heegaard splitting consisting of two solid tori glued along the red torus boundaries (Figure~\ref{d1} left). From this, we can see that, for the link complement $\SS^3\backslash W_{0,1,1,0}$, if we remove the tubular neighborhood of the belt, we will obtain a red solid torus with a green clasp removed (Figure~\ref{d1} right).

          \begin{figure}[H]
              \centering
              \includegraphics[width=\linewidth]{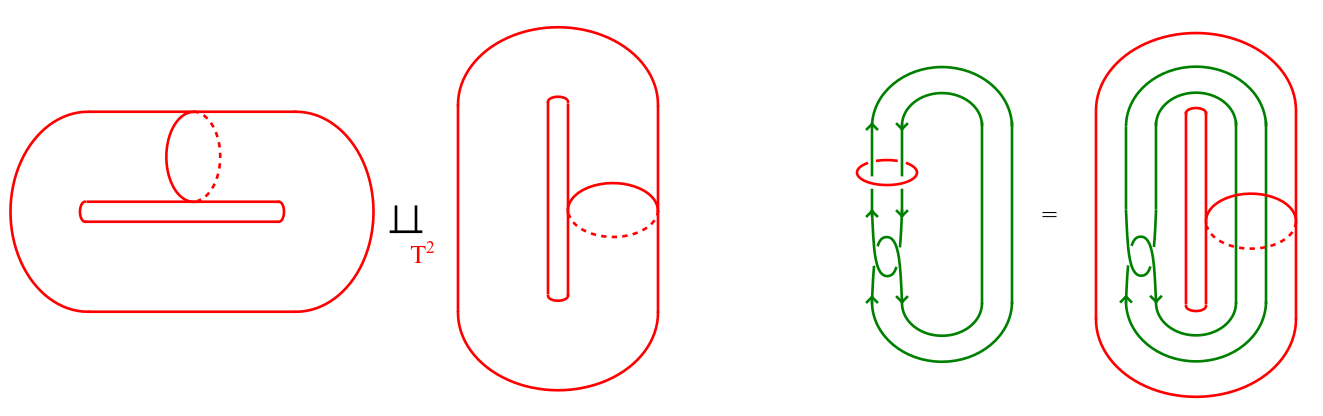}
              \caption{The standard genus one Heegaard splitting of $\SS^3$ (left) and the Whitehead link complement (right)}\label{d1}
          \end{figure}
          
Besides, we consider the complement of the tubular neighborhood of the figure eight knot $4_1$. Then we can glue the manifolds $\SS^3\backslash W_{0,1,1,0}$ and $\SS^3 \backslash N(4_1)$ to obtain the Whitehead double of the figure eight knot (Figure~\ref{wd41}).

          \begin{figure}[H]
              \centering
              \includegraphics[width=0.9\linewidth]{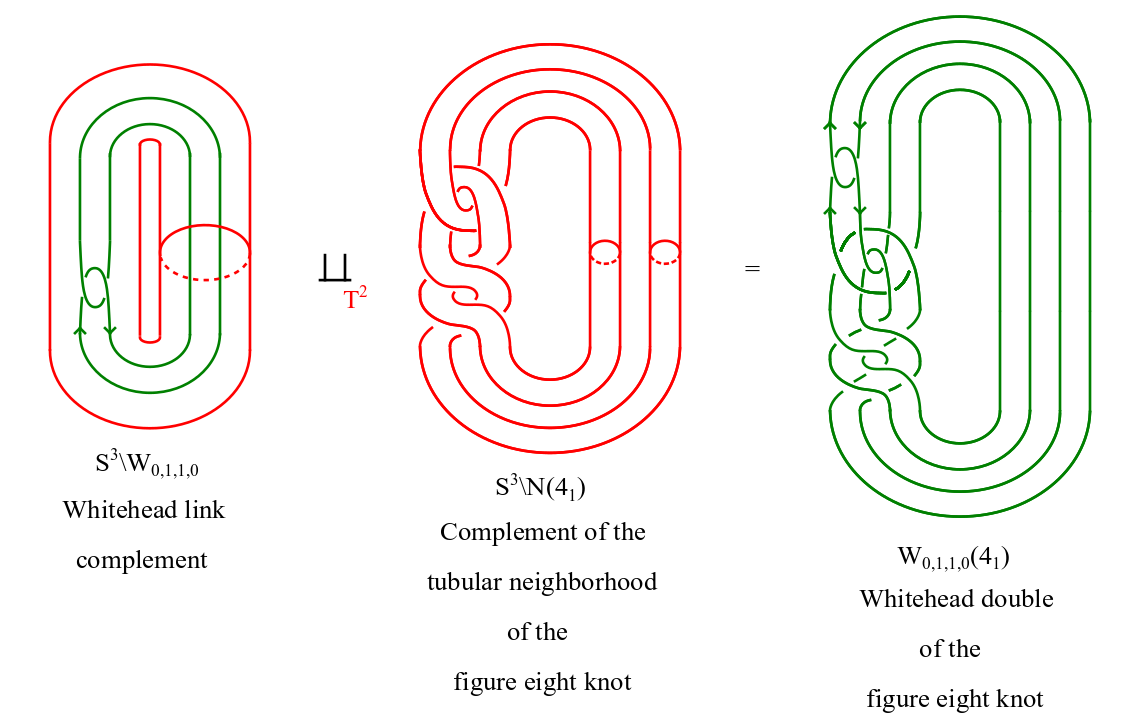}
              \caption{The Whitehead double of the figure eight knot $W_{0,1,1,0}(4_1)$ is obtained by gluing $\SS^3\backslash W_{0,1,1,0}$ and $\SS^3\backslash N(4_1)$ together along the red torus boundary}\label{wd41}
          \end{figure}

Moreover generally, let $W_{a,1,c,d}$ be the Whitehead chains with $a$ twists, $1$ belt, $c$ clasps and $d$ mirror clasps \cite{V08}. For the same reason, we can cable the Whitehead chains $W_{a,1,c,d}$ onto the tubular neighborhood of $4_1$. The resulting link is denoted by $W_{a,1,c,d}(4_1)$ (see for example Figure~\ref{W111141}).

          \begin{figure}[H]
          \centering
              \includegraphics[width=0.9\linewidth]{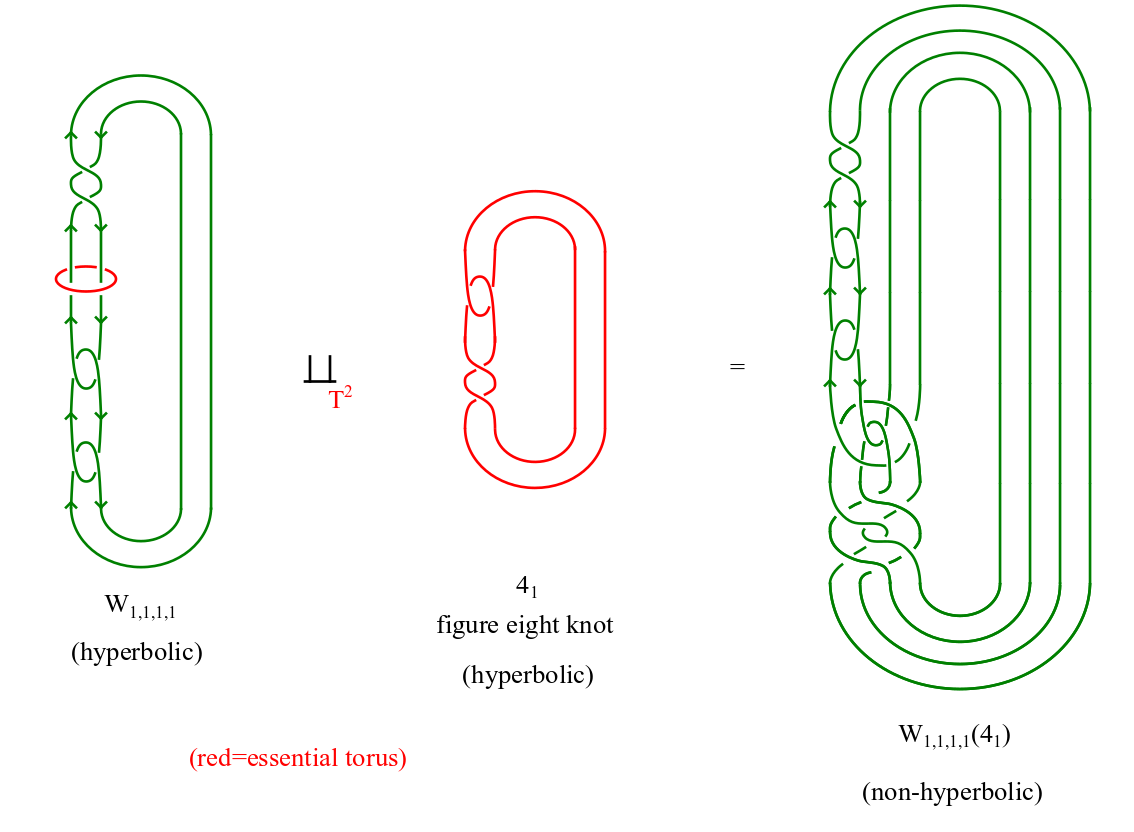}
              \caption{The link $W_{1,2,1,1}(4_1)$. For simplicity, we use the color red to highlight the essential torus.}\label{W111141}
          \end{figure}

Besides, for $p\geq 0$, we consider the $(p+1)$-th iterated Whitehead double of the figure eight knot, denoted by $W_{0,1,1,0}^{p+1}(4_1)$. 

Note that the JSJ decomposition of the manifolds $\SS^3 \backslash W_{a,1,c,d}(4_1)$ and $\SS^3 \backslash W_{0,1,1,0}^{p+1}(4_1)$ are given by
\begin{align*}
\SS^3 \backslash W_{a,1,c,d}(4_1) &\cong (\SS^3 \backslash 4_1) \coprod_{T^2} (\SS^3 \backslash W_{a,1,c,d})  \\
\SS^3 \backslash W_{0,1,1,0}^{p+1}(4_1) &\cong (\SS^3 \backslash 4_1) \coprod_{T^2} \underbrace{\SS^3 \backslash W_{0,1,1,0} \coprod_{T^2} \dots \coprod_{T^2} \SS^3 \backslash W_{0,1,1,0}}_{\text{$p+1$ times}} 
\end{align*}
with simplicial volumes
\begin{align*}
||\SS^3 \backslash W_{a,1,c,d}(4_1)|| &=  ||\SS^3 \backslash 4_1|| + ||\SS^3 \backslash W_{a,1,c,d}|| \\
||\SS^3 \backslash W_{0,1,1,0}^{p+1}(4_1)|| &= ||\SS^3 \backslash 4_1||  + (p+1)||\SS^3 \backslash W_{0,1,1,0}||
\end{align*}
respectively. In particular, $\SS^3 \backslash W_{0,1,1,0}^{p+1}(4_1)$ contains $p+2$ hyperbolic pieces (see for example Figure~\ref{wd41}).

The third family of links $W^\alpha_\beta$ is constructed by doing the iterated Whitehead double on the Hopf link as follows. Consider the Hopf link denoted by $W^0_0$ as shown in Figure~\ref{walphabeta}. The link $W^\alpha_\beta$ is obtained by doing the Whitehead double $\alpha$ times along the first component (colored by red) and $\beta$ times along the second component (colored by blue). Under this notation, the Whitehead link is denoted by $W^1_0$ or $W^0_1$. Figure~\ref{walphabeta} shows the diagrams of the links $W^0_0, W^1_0, W^1_1$ and $W^2_1$.

          \begin{figure}[H]
          \centering
              \includegraphics[width=0.8\linewidth]{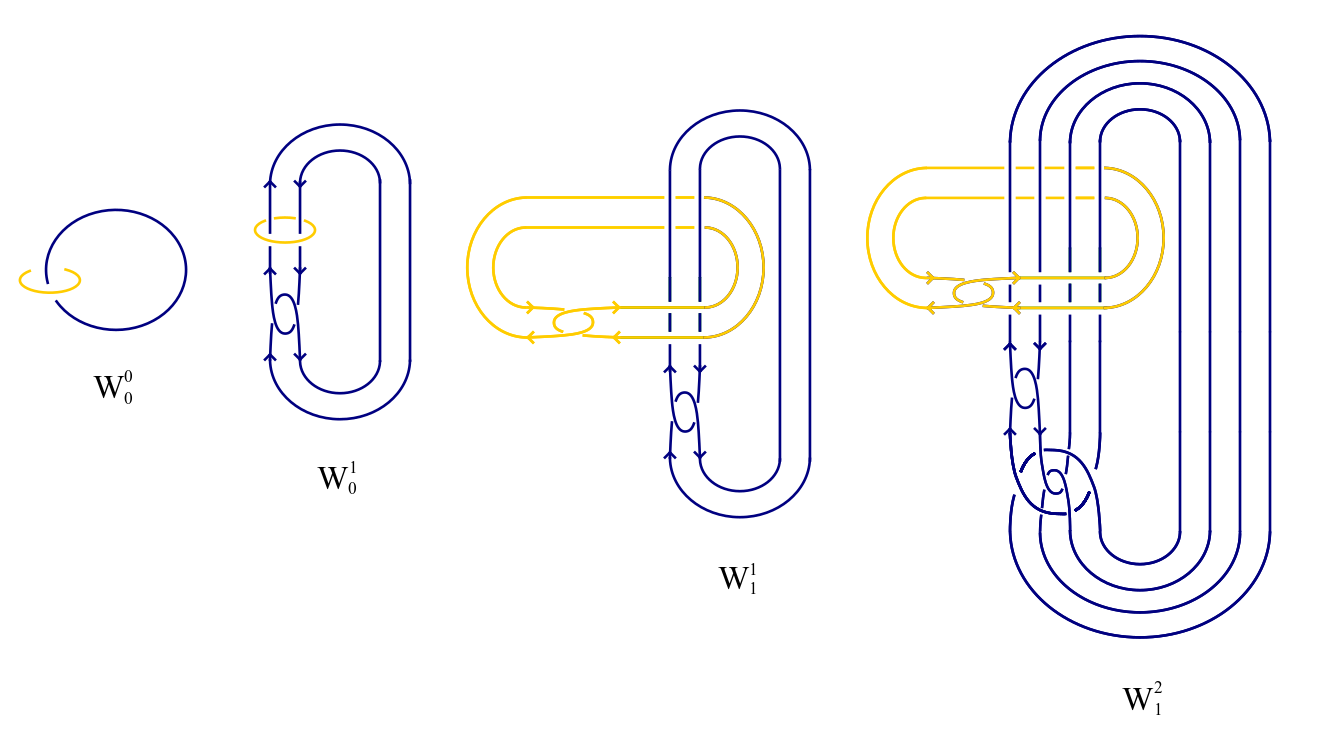}
              \caption{The link $W^\alpha_\beta$ is obtained by doing the Whitehead double $\alpha$ times along the first component (colored by yellow) and $\beta$ times along the second component (colored by blue).}\label{walphabeta}
          \end{figure}

Similar to previous discussion, the JSJ decomposition of the manifold $\SS^3 \backslash W^\alpha_\beta$ is given by
$$ 
\SS^3 \backslash W^\alpha_\beta
\cong
\underbrace{\SS^3 \backslash W_{0,1,1,0}  \coprod_{T^2} \dots  \coprod_{T^2} \SS^3 \backslash W_{0,1,1,0}  }_{\text{$\alpha + \beta$ times}}
$$
with simplicial volume
$$ ||\SS^3 \backslash W^\alpha_\beta|| = (\alpha + \beta)||\SS^3 \backslash W_{0,1,1,0}|| $$

In particular, $\SS^3 \backslash W^\alpha_\beta$ contains $(\alpha + \beta)$ hyperbolic pieces.

Finally, motivated by the property of the colored Jones polynomials, we consider another type of construction of link called Hopf union. Let $K_1$ and $K_2$ be two knots. Consider the connected sum of $K_1$ with the yellow component of the Hopf link and the connected sum of $K_2$ with the blue component of the Hopf link. We call the resulting link by the Hopf union of $K_1$ and $K_2$ and denote it by $K_1 \# W^0_0 \# K_2$. Figure~\ref{41H41} below shows the example of $4_1 \# W^0_0 \# 4_1$. To obtain the JSJ decomposition of $\SS^3\backslash (K_1 \# W^0_0 \# K_2)$, we embed two tori as shown in Figure~\ref{K1HK2}.

\begin{minipage}{\linewidth}
      \centering
            \begin{minipage}[b]{0.3\linewidth}
          \begin{figure}[H]
              \includegraphics[width=\linewidth]{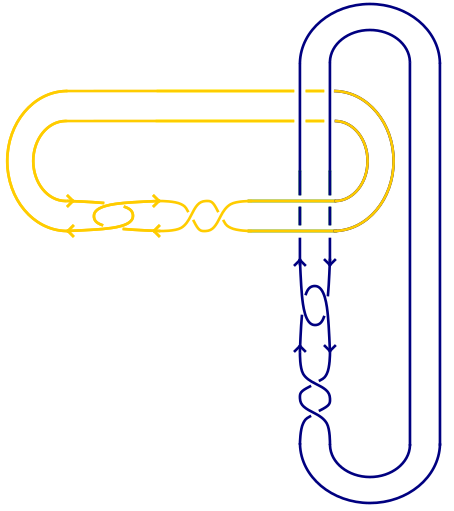}\vspace{-12pt}
          \end{figure}
      \end{minipage}
                  \hspace{0.15\linewidth}
            \begin{minipage}[b]{0.4\linewidth}
          \begin{figure}[H]
          \centering
              \includegraphics[width=\linewidth]{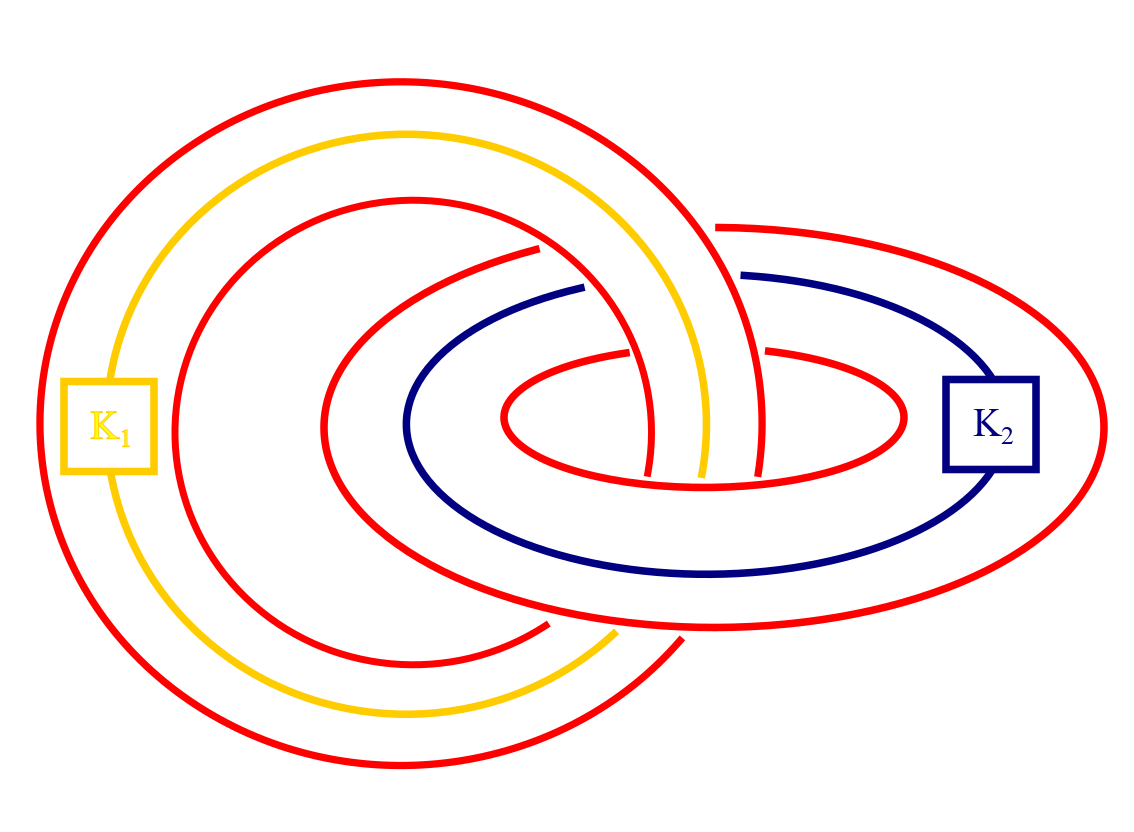}\vspace{-12pt}
          \end{figure}
      \end{minipage}
                 \hspace{0.05\linewidth}
                 \vspace{12pt}
  \end{minipage}
  \begin{minipage}{\linewidth}
      \centering
            \begin{minipage}[t]{0.45\linewidth}
          \begin{figure}[H]
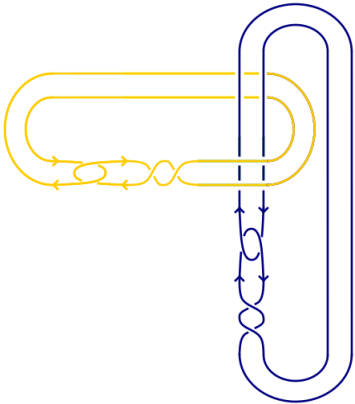

              \caption{The link $4_1 \# W^0_0 \# 4_1$}\label{41H41}
          \end{figure}
      \end{minipage}
                  \hspace{0.05\linewidth}
            \begin{minipage}[t]{0.4\linewidth}
          \begin{figure}[H]
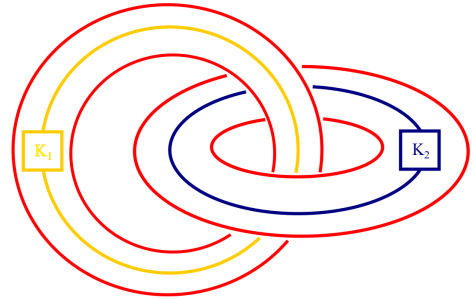

          \centering
              \caption{Embed the red tori to the link complement}\label{K1HK2}
          \end{figure}
      \end{minipage}
                 \hspace{0.05\linewidth}
  \end{minipage}

Let $P=(\DD^2 \backslash \{(-\frac{1}{2},0),(\frac{1}{2},0)\}) \times \SS_1$. Then one can show that the JSJ decomposition of $\SS^3\backslash (K_1 \# W^0_0 \# K_2)$ is given by
$$ \SS^3 \backslash (K_1 \# W^0_0 \# K_2) \cong (\SS^3 \backslash W_0^0)  \coprod_{T^2} (\SS^3 \backslash K_1)  \coprod_{T^2} P  \coprod_{T^2} (\SS^3 \backslash K_2)  \coprod_{T^2} P  $$
with simplicial volume
$$ || \SS^3 \backslash (K_1 \# W^0_0 \# K_2) || = ||\SS^3 \backslash K_1|| + ||\SS^3 \backslash K_2|| $$

\subsection{Main results}

\subsubsection{Volume conjectures for the quantum invariants of $W_{a,b,c,d}(4_1)$}
First of all, we study the asymptotics of the $\vec{N}$-th colored Jones polynomials and the $r$-th Turaev-Viro invariants of $W_{a,1,c,d}(4_1)$, where $a\in \ZZ, c, d\in\ZZ_{\geq 0}$ with $c+d\geq 1$. The first result can be stated as follows. 

\begin{theorem}\label{mainthm1}
Conjecture~\ref{NVCMM} is true for all $W_{a,1,c,d}(4_1)$ with $a\in \ZZ$, $c+d \geq 1$. 
\end{theorem} 

As a consequence, we can prove the following result for any $b\in \NN$.
\begin{corollary}\label{cablingtv1}
Conjecture~\ref{tvGM} is ture for all $\SS^3 \backslash W_{a,b,c,d}(4_1)$ with $a\in \ZZ$, $b \in \NN$, $c+d \geq 1$.
\end{corollary}

Nevertheless, for $p\geq 0$ and $\alpha, \beta \geq 0$ with $\alpha+\beta\geq 1$, we have the following two results.
\begin{theorem}\label{mainthm5}
Conjecture~\ref{NVCMM} is true for $W^{p+1}_{0,1,1,0}(4_1)$ and $W^\alpha_\beta$ for $p\geq 0$ and $\alpha, \beta \geq 0$ with $\alpha+\beta\geq 1$.
\end{theorem}

\begin{corollary}\label{cablingtv2}
Conjecture~\ref{tvGM} is true for $W^{p+1}_{0,1,1,0}(4_1)$ and $W^\alpha_\beta$ for $p\geq 0$ and $\alpha, \beta \geq 0$ with $\alpha+\beta\geq 1$.
\end{corollary}

In particular, Theorem~\ref{mainthm1} and Corollary~\ref{cablingtv1} give an infinite family of prime links with $2$ hyperbolic pieces satisfying Conjecture~\ref{NVCMM} and Conjecture~\ref{tvGM}. Also, for any $n\in \NN$, Theorem~\ref{mainthm5} and Corollary~\ref{cablingtv2} gives examples of prime links with $n$ hyperbolic pieces under JSJ decomposition satisfying Conjecture~\ref{NVCMM} and Conjecture~\ref{tvGM}.

\subsubsection{Generalized volume conjecture for the colored Jones polynomials of prime links whose complements consist of only hyperbolic pieces}

Next, we study the asymptotics of the $M$-th colored Jones polynomials of the iterated Whitehead double of the figure eight knot. We have the following results.

\begin{theorem}\label{mainthm3}
Let $\ds s= \lim_{N\to \infty}\frac{M}{N+\frac{1}{2}}$. For any $p\geq 0$, there exists $\delta>0$ such that for any $1-\delta < s \leq 1$, we have
\begin{align}
&\lim_{N\to \infty} \frac{2\pi}{N+\frac{1}{2}} 
\log |J_{M}( W_{0,1,1,0}^{p+1}(4_1), e^{\frac{2\pi i}{N+\frac{1}{2}}}) |  \notag\\
&=  \Vol(\SS^3 \backslash WL; u_1 = 0, u_2=2\pi i(1-s)) + p\Vol(\SS^3 \backslash WL) + \Vol(\SS^3 \backslash 4_1) \label{3},
\end{align}
where $u_1,u_2$ are the logarithm of the holonomy of the meridian of the belt and that of the clasp respectively, and $ \Vol(\SS^3 \backslash WL; u_1 = 0, u_2=2\pi i(1-s))$ is the volume of the Whitehead link complement equipped with the incomplete hyperbolic structure parametrized by $u_1,u_2$.
\end{theorem}

Similarly, we study the asymptotics of the $(M_1,M_2)$-th colored Jones polynomials of the iterated Whitehead double on the Hopf link and have the following results.

\begin{theorem}\label{mainthm2}
Let $\ds s_i= \lim_{N\to \infty}\frac{M_i}{N+\frac{1}{2}}$ for $i=1,2$. For any $\alpha, \beta \geq 0$ with $\alpha+\beta \geq 1$, there exists $\delta>0$ such that for any $1-\delta < s_1,s_2 \leq 1$, we have
\begin{align}
&\lim_{N\to \infty} \frac{2\pi}{N+\frac{1}{2}} 
\log |J_{M_1,M_2}( W^\alpha_\beta, e^{\frac{2\pi i}{N+\frac{1}{2}}}) |  \notag\\
&=  \Vol(\SS^3 \backslash WL; u_1 = 0, u_2=2\pi i(1-s_1)) 
+ (\alpha+\beta-2)\Vol(\SS^3 \backslash WL) \notag \\
&\qquad + \Vol(\SS^3 \backslash WL; u_1 = 0, u_2=2\pi i(1-s_2))  \label{4}
\end{align}
where $u_1,u_2$ are the logarithm of the holonomy of the meridian of the belt and that of the clasp respectively, and $ \Vol(\SS^3 \backslash WL; u_1 = 0, u_2=2\pi i(1-s))$ is the volume of the Whitehead link complement with incomplete hyperbolic structure parametrized by $u_1,u_2$.
\end{theorem}

Recall that the JSJ decompositions of the manifolds $\SS^3\backslash W_{0,1,1,0}^{p+1}(4_1)$ and $\SS^3\backslash W^\alpha_\beta$ consist of only hyperbolic pieces. From Equations (\ref{3}) and (\ref{4}) in Theorem~\ref{mainthm3} and~\ref{mainthm2}, we can see that the relationship between the color of a link component and the corresponding choice of hyperbolic structure is in the same spirit as that in Conjecture~\ref{GVC}. Moreover, along the incompressible torus, the metric is always chosen to be complete. We call this assignment of hyperbolic structures to each hyperbolic piece the \emph{natural hyperbolic structure associated to the coloring} (or \emph{natural hyperbolic structure} in short). Furthermore, recall that the colored Jones polynomials of the connected sum $L_1\# L_2$ along a component colored by $i$ is given by
$$ [i]J_{\vec{M}}(L_1 \# L_2, t) = J_{\vec{M_1}}(L_1,t) J_{\vec{M_2}}(L_2,t) ,$$
where $\vec{M_1}$ and $\vec{M_2}$ are the restrictions of the color $\vec{M}$ to $L_1$ and $L_2$ respectively and 
$$[i]=\frac{t^{\frac{i}{2}}-t^{-\frac{i}{2}}}{t^{\frac{1}{2}}-t^{-\frac{1}{2}}}.$$ Therefore, we restrict our attention to prime links. The above observation leads to the following conjecture. 

\begin{conjecture}\label{GVCS} 
Let $L$ be a $n$ components prime link with complement having only hyperbolic pieces. Then there exists $\delta_L>0$ such that for any coloring $\vec{M} = (M_1,\dots, M_n)$ with the limiting ratio $1-\delta< s_i \leq 1$, $i=1,2,\dots,n$, we have
\begin{align*}
\lim_{N\to \infty}\frac{2\pi}{N+\frac{1}{2}} \log| J_{\vec{M}}(L,e^{\frac{2\pi i}{N+\frac{1}{2}}}) |
= \frac{1}{2\pi} \Vol(\SS^3 \backslash L; u_i = 2\pi i(1-s_i)),
\end{align*}
where $\Vol(\SS^3 \backslash L; u_i = 2\pi i (1-s_i))$ is the sum of the volume of the hyperbolic pieces equipped with the natural hyperbolic structures.
\end{conjecture}

By Theorem~\ref{mainthm2}, we have
\begin{corollary}
Conjecture~\ref{GVCS} is true for the links $W_{0,1,1,0}^{p+1}(4_1)$ and $W^\alpha_\beta$ for $p\geq 0$ and $\alpha,\beta \geq 0$ with $\alpha+\beta\geq 1$.
\end{corollary}

\subsubsection{Generalized volume conjecture for the colored Jones polynomials under Hopf union}

So far we discussed the generalized volume conjecture for links whose complement consists of only hyperbolic pieces. It is natural to consider a similar version of generalized volume conjecture for links whose complements consist of hyperbolic pieces together with Seifert fiber pieces.

To obtain examples of this type of links, we consider the Hopf union of knots and introduce the following lemma, which can be proved by using skein theory (see Appendix A).

\begin{lemma}\label{CJHUlemma}
Let $M_1,M_2 \in \NN$ and let $K_1,K_2$ be two knots. For $t=e^{\frac{2\pi i}{N+\frac{1}{2}}}$, we have
\begin{align}
\lt| J_{M_1,M_2}(K_1 \# W^0_0 \# K_2, t) \rt| = \lt| \frac{[M_1M_2]}{[M_1][M_2]} J_{M_1}(K_1, t) J_{M_2}(K_2,t)  \rt|  ,\label{cjhueqn}
\end{align}
where $[n] = \frac{t^\frac{n}{2}-t^\frac{n}{2}}{t^\frac{1}{2}-t^\frac{1}{2}}$.
\end{lemma}

Suppose $\lt| J_{M_1,M_2}(K_1 \# W^0_0 \# K_2, e^{\frac{2\pi i}{N+\frac{1}{2}}}) \rt| \neq 0$. Then we have
\begin{align*}
& \lim_{N \to \infty} \frac{2\pi }{N+\frac{1}{2}} \log| J_{M_1,M_2}(K_1 \# W^0_0 \# K_2, e^{\frac{2\pi i}{N+\frac{1}{2}}}) | \\
&= \lim_{N \to \infty} \frac{2\pi }{N+\frac{1}{2}} \log| J_{M_1}(K_1, e^{\frac{2\pi i}{N+\frac{1}{2}}}) |
+ \lim_{N \to \infty} \frac{2\pi }{N+\frac{1}{2}} \log| J_{M_2}(K_2, e^{\frac{2\pi i}{N+\frac{1}{2}}}) |
\end{align*}
In particular, if the generalized volume conjectures for $K_1$ and $K_2$ are true for the sequences $M_1$ and $M_2$ respectively, i.e.
\begin{align*}
\lim_{N \to \infty} \frac{2\pi }{N+\frac{1}{2}} \log| J_{M_1}(K_1, e^{\frac{2\pi i}{N+\frac{1}{2}}}) | 
&= \Vol(\SS^3 \backslash K_1; u = 2\pi i (1-s)) \text{ and } \\
 \lim_{N \to \infty} \frac{2\pi }{N+\frac{1}{2}} \log| J_{M_2}(K_2, e^{\frac{2\pi i}{N+\frac{1}{2}}}) |
&= \Vol(\SS^3 \backslash K_2; u = 2\pi i (1-s)) 
\end{align*}
then we have
$$\lim_{N \to \infty} \frac{2\pi }{N+\frac{1}{2}} \log| J_{M_1,M_2}(K_1 \# W^0_0 \# K_2, e^{\frac{2\pi i}{N+\frac{1}{2}}}) | 
= \Vol(\SS^3 \backslash K_1, u = 2\pi i (1-s)) 
+ \Vol(\SS^3 \backslash K_2, u = 2\pi i (1-s)) 
$$

In particular, when $t=e^{\frac{2\pi i}{N+\frac{1}{2}}}$, since $[N^2]\neq 0$, we have
\begin{theorem}\label{CJHU}
Conjecture~\ref{NVCMM} for $K_1$ and $K_2$ imply Conjecture~\ref{NVCMM} for $K_1 \# W^0_0 \# K_2$.
\end{theorem}

In particular, we have
\begin{corollary}
Conjectures~\ref{NVCMM} amd \ref{tvGM} are true for $4_1 \# W^0_0 \# 4_1$.
\end{corollary}

We remark that one can consider the Hopf union of links. Using the same argument, we have
\begin{corollary}\label{GVCHU}
Conjectures~\ref{NVCMM} amd \ref{tvGM} are true for $\underbrace{4_1 \# W^0_0 \# 4_1 \# \dots \# W^0_0 \#  4_1}_{\text{$p$ many copies of $4_1$}}$ for any $p>0$.
\end{corollary}

This result gives another examples of links whose complements consist of more than one hyperbolic pieces (together with some Seifert fibered piece) satisfying the Conjectures~\ref{NVCMM} amd \ref{tvGM}. 

In general, it is possible that $[M_1M_2]=0$. For example, let $p>1$ be an odd number and consider the case where $2N+1=p^2$ and $M_1=M_2=p(p-1)$. Note that the limiting ratios are given by
$$\ds s_i = \lim_{p\to \infty}\frac{p(p-1)}{p^2} = 1 \text{ for $i=1,2$}.$$
However, when $t=e^{\frac{2\pi i}{N+\frac{1}{2}}}$, since
$$ [M_1M_2] 
= \frac{t^{\frac{M_1M_2}{2}}-t^{-\frac{M_1M_2}{2}}}{t^{\frac{1}{2}}-t^{-\frac{1}{2}}} 
= \frac{e^{\frac{2\pi ip^2 (p-1)^2}{p^2}}-e^{-\frac{2\pi ip^2 (p-1)^2}{p^2}}}{e^{\frac{2\pi i}{p^2}}-e^{-\frac{2\pi i}{p^2}}}
=0,
$$
we have $|J_{M_1,M_2}(K_1\#W_0^0\#K_2, e^{\frac{2\pi i}{N+\frac{1}{2}}})| = 0$ for any knots $K_1$ and $K_2$.

\subsection{Further discussion}
There are several interesting results about Conjecture~\ref{tvGM} under cabling. In Theorem 1.5 of \cite{D18}, R.Detcherry proved that Conjecture~\ref{tvGM} is stable under $(2n+1,2)$ cabling. In \cite{DK19}, R. Detcherry and E. Kalfagianni showed that if $M$ satisfies Conjecture~\ref{vctv} and $S$ is an invertible cabling space with a distinguished torus boundary component $T$, then the manifold $M'$ obtained by gluing a component of $\partial S \backslash T$ to a component of $\partial M$ also satisfies Conjecture~\ref{vctv}. Combining these results with Corollary~\ref{cablingtv1} and \ref{GVCHU}, we can obtain more examples of manifolds satisfying Conjecture~\ref{tvGM}.

Besides, we know that the colored Jones polynomials of links are special cases of the Reshetikhin-Turaev invariants, which form a topological quantum field theory (TQFT) \cite{BHMV95}. Since the link complement is obtained by gluing along the boundary tori of each of the pieces in the JSJ decomposition, it is natural to ask:
\begin{question}
Is there any TQFT interpretation for Conjecture~\ref{GVCS}?
\end{question}

\subsection{Outline of this paper}
In Section~\ref{Wabcd41}, we study the volume conjectures for the cabling of the figure eight knot by the Whitehead chains. In Section~\ref{W021041} and \ref{W021041tv}, in order to illustrate the technique of proving the volume conjecture, we first prove Theorem~\ref{mainthm1} and Corollary~\ref{cablingtv1} for the link $W_{0,1,1,0}(4_1)$. In Section~\ref{strategy}, we give a summary of the technique we used in the previous subsection. Theorem~\ref{mainthm1} and Corollary~\ref{cablingtv1} will be proved in Section~\ref{Wa1cd41pf}. Theorem~\ref{mainthm3} and \ref{mainthm2} will be proved in Section~\ref{Wp0110W021041} and \ref{Walphabeta} respectively. Finally, the proof of Lemma~\ref{CJHUlemma} will be included in Appendix A. 

\subsection{Acknowledgements}

The author would like to thank his advisor Tian Yang for his guidance. Moreover, the author would like to thank Effie Kalfagianni for comments on this work.

\section{Cabling of the figure eight knot by Whitehead chains}\label{Wabcd41}
\subsection{Volume conjecture for $J_{N}(W_{0,1,1,0}(4_1),e^{\frac{2\pi i}{N+\frac{1}{2}}})$}\label{W021041}

First of all, we compute the $N$-th colored Jones polynomials of $W_{0,1,1,0}(4_1)$. Let 
$$(t)_n = \prod_{k=1}^n(1-t^k)$$
Recall that the $M$-th normalized colored Jones polynomials of the figure eight knot and the $(N,N)$-th colored Jones polynomials of the Whitehead link $W_{0,1,1,0}$ are given by \cite{HZ07}
\begin{align*} 
J'_{M}(4_{1},t) &= \sum_{k=0}^{M-1} \frac{t^{-k M}}{1-t^{M}} \cdot \frac{(t)_{M+k}}{(t)_{M-k-1}} \\
J_{\vec{N}}(W_{0,1,1,0}, t) &=  \sum_{n=0}^{N-1} \frac{t^{(2n+1)/2}-t^{-(2n+1)/2}}{t^{\frac{1}{2}}-t^{-\frac{1}{2}}}  
\cdot \lt(\frac{t^{N(2n+1)/2}-t^{-N(2n+1)/2}}{t^{(2n+1)/2}-t^{-(2n+1)/2}}\rt)
\cdot C(n,t; N)
\end{align*}
respectively, where
$$C(n,t; N) = t^{\frac{N^2-1}{2} + \frac{N(N-1)}{2}} \sum_{l = 0}^{N-1-n} t^{-N(l+n)} \cdot \frac{(t)_{N-l-1}(t)_{l+n}}{(t)_n(t)_{N-l-n-1}(t)_l}$$
By equation (2.9) in \cite{HZ07}, the $N$-th colored Jones polynomial of $W_{0,1,1,0}(4_1)$ is given by
\begin{align*}
&J_{N} (W_{0,1,1,0}(4_1), t) \\
&= t^{\frac{N^2-1}{2} + \frac{N(N-1)}{2}} \sum_{n=0}^{N-1} \frac{t^{(2n+1)/2}-t^{-(2n+1)/2}}{t^{\frac{1}{2}}-t^{-\frac{1}{2}}} \cdot C(n,t; N)  \cdot J'_{2n+1}(4_1,t) \\
&= \frac{t^{\frac{N^2-1}{2} + \frac{N(N-1)}{2}}}{1-t}
\sum_{n=0}^{N-1} \sum_{l = 0}^{N-1-n} \sum_{k=0}^{2n} t^{-n-k}
\lt(t^{-N(l+n) }\frac{(t)_{N-l-1}(t)_{l+n}}{(t)_n(t)_{N-l-n-1}(t)_l}\rt)
\lt(t^{-2nk} \frac{(t)_{2n+1+k}}{(t)_{2n-k}}  \rt)
\end{align*}

When $t=e^{\frac{2\pi i}{N+\frac{1}{2}}}$, we have the following upper bound estimates.

\begin{lemma}\label{uppbd1}(Lemma 1 in \cite{WA17})
For $k,l \in  \{1,2,\dots, M-1\}$, let 
\begin{align*}
g_{M}(k) 
= \left| \frac{(t)_{M+k}}{(t)_{M-k-1}} \right| 
\end{align*}
For each $M$, let $k_{M} \in \{ 1,\dots, M-1 \}$ such that $g_M(k_M)$ achieves the maximum among all $g_M (k)$. Assume that $ \frac{M}{N+\frac{1}{2}} \to s$ and $\frac{k_M}{N+\frac{1}{2}} \to k_{s}$ as $N \to \infty$. Then we have
\[ \lim_{r \to \infty}\frac{1}{N+\frac{1}{2}} \log( g_M (k_M) ) = -\frac{1}{2\pi}\left( \Lambda(\pi(k_s -s)) + \Lambda(\pi(k_s +s)) \right) \leq \frac{\Vol(\SS^3 \backslash 4_1)}{2\pi}.\]
Furthermore, the equality holds if and only if $(s=1$ and $k_s = \frac{5}{6})$ or $(s=\frac{1}{2}$ and $k_s = \frac{1}{3})$.
\end{lemma}

\begin{lemma}\label{uppbdW}(Lemma 5 in \cite{W19}) 
For any $n \in \{1,2,\dots,M-1\}$, $l \in \{1,2,\dots, M-1-n\}$, let 
$$ c_{M}(n,l;t) = \lt| \frac{(t)_{N-l-1}(t)_{l+n}}{(t)_n(t)_{N-l-n-1}(t)_l} \rt| $$
For each $M$, let $n_M\in \{1,2,\dots,M-1\}$ and $l_M \in \{1,2,\dots, M-1-n\}$ such that $c_M(n_M,l_M)$ achieves the maximum among all $c_{M}(n,l)$. Assume that $\frac{M}{N+\frac{1}{2}} \to s \in [0,1]$, $\frac{n_M}{N+\frac{1}{2}} \to n_s$ and $\frac{l_M}{N+\frac{1}{2}} \to l_s$. Then we have
$$ \lim_{N \to \infty} \frac{1}{N+\frac{1}{2}} \log(c_M(n,l;e^{\frac{2\pi i}{N+\frac{1}{2}}})) \leq \frac{\Vol(\SS^3\backslash WL)}{2\pi} $$
Furthermore, the equality holds if and only if 
$ s=1, n_s = \frac{1}{2} \text{ and } l_s = \frac{1}{4}$.
\end{lemma}

From Lemma~\ref{uppbd1} and \ref{uppbdW}, we have
$$\limsup_{N \to \infty} \frac{1}{N+\frac{1}{2}}\log|J_{N} (W_{0,1,1,0}(4_1), e^{\frac{2\pi i}{N+\frac{1}{2}}}) | \leq \frac{\Vol(\SS^3\backslash WL)+\Vol(\SS^3 \backslash 4_1)}{2\pi} = \frac{v_3 ||W_{0,1,1,0}(4_1) || }{2\pi}$$

For $\eta >0,$ we let 
$$ D_\eta = \lt\{(n,l,k) \mid \lt|\frac{n}{N+\frac{1}{2}}-\frac{1}{2}\rt|, \lt|\frac{l}{N+\frac{1}{2}} - \frac{1}{4}\rt| , \lt|\frac{k}{N+\frac{1}{2}} - \frac{5}{6}\rt| < \eta \rt\}$$
At the end we will show that the exponential growth rate is exactly $\dfrac{v_3 ||W_{0,1,1,0}(4_1) || }{2\pi}$. In particular, we can write
\begin{align}
&J_{N} (W_{0,1,1,0}(4_1), e^{\frac{2\pi i}{N+\frac{1}{2}}}) \notag \\
&\stackrel[N \to \infty]{\sim}{} 
\frac{t^{\frac{N^2-1}{2} + \frac{N(N-1)}{2}}}{1-t}
\sum_{D_\delta} t^{- n - k}
\lt(t^{\frac{l+n}{2} }\frac{(t)_{N-l-1}(t)_{l+n}}{(t)_n(t)_{N-l-n-1}(t)_l}\rt)
\lt(t^{-2nk} \frac{(t)_{2n+1+k}}{(t)_{2n-k}}  \rt)
\end{align}

For any odd integer $r=2N+1\geq 3$, we consider the following version of the quantum dilogarithm function $\varphi_r(z)$ (see \cite{F95, FKV01}, or Section 2.3 in \cite{WY20} for a review of its properties) defined by 
\begin{align*}
\varphi_r (z) = \frac{4\pi i}{r}\int_\Omega \frac{e^{(2z - \pi)x} }{4x \sinh(\pi x) \sinh (\frac{2\pi x}{r} )} dx,
\end{align*}
where 
\begin{align*}
\Omega = (-\infty, - \epsilon] \cup \{ z \in \CC \mid |z| = \epsilon, \im z >0\} \cup [ \epsilon, \infty)
\end{align*}
for some $\epsilon \in (0,1)$ and 
\begin{align*}
z \in \lt\{ z \in \CC \lt\vert - \frac{\pi}{2N+1} < \Re z < \pi + \frac{\pi}{2N+1}\rt. \rt\}
\end{align*}
For any $z \in \CC$ with $0 < \Re z < \pi$, the quantum dilogarithm function satisfies the functional equation (Lemma 2.1)1) in \cite{WY20})
\begin{align}
1 - e^{2i z} = \exp \lt(\frac{N+\frac{1}{2}}{2\pi i}\lt( \varphi_r \lt( z - \frac{\pi }{2N+1} \rt) - \varphi_r \lt( z + \frac{\pi }{2N+1} \rt) \rt)\rt) \label{QDfunc}
\end{align}
Also, for $t=e^{\frac{2\pi i}{N+\frac{1}{2}}}$, we have (Lemma 2.2)1) in \cite{WY20})
\begin{align}\label{t!toQD}
(t)_n = e^{\frac{r}{4\pi i}\lt( \varphi_r\lt(\frac{\pi}{r}\rt) - \varphi_r\lt(\frac{2\pi n}{r}+\frac{\pi}{r}\rt)\rt)}
\end{align}
for $0\leq n \leq r-2$.

Furthermore, let $\Li: \CC\backslash (1,\infty) \to \CC$ be the dilogartihm function defined by
$$ \Li(z) = - \int_0^z \frac{\log(1-u)}{u} du $$
For any $z$ with $0<\Re z < \pi$, the quantum dilogarithm function satisfies (Lemma 2.3 in \cite{WY20})
\begin{align}
\varphi_r (z) = \Li(e^{2iz}) + \frac{2\pi^2 e^{2iz}}{3(1-e^{2iz})} \frac{1}{(2N+1)^2} + O\lt(\frac{1}{(N+\frac{1}{2})^3}\rt) \label{QtoC},
\end{align}
and as $N \to \infty$, $\varphi_r(z)$ uniformly converges to $\Li(e^{2iz})$ on any compact subset of $\{z \in \CC \mid 0 < \Re z < \pi \}$.

Using (\ref{t!toQD}), we can write
\begin{align}
&J_{N} (W_{0,1,1,0}(4_1), e^{\frac{2\pi i}{N+\frac{1}{2}}}) \notag\\
&\stackrel[N \to \infty]{\sim}{} 
- \frac{N+\frac{1}{2}}{\pi} \exp\lt(  -\frac{r}{4\pi i} \varphi_r \lt(  \frac{\pi}{2N+1}\rt)\rt)
 \sum_{D_\eta } 
 e^{\frac{-2\pi i}{N+\frac{1}{2}}(-\frac{n}{2}+\frac{l}{2}-k)}
 \notag\\ 
&\qquad \times
\exp\lt( (N+\frac{1}{2}) \Phi_{N,N}\lt(W_{0,1,1,0}(4_1); \frac{n}{N+\frac{1}{2}}, \frac{l}{N+\frac{1}{2}}, \frac{k}{N+\frac{1}{2}}\rt) \rt)
\end{align}
where $\Phi_{N} \lt(W_{0,1,1,0}(4_1); z_1, z_2, z_3\rt) $ is given by
\begin{align}
\Phi_{N}\lt(W_{0,1,1,0}(4_1); z_1, z_2, z_3\rt) 
&= \Phi_{N} (W_{0,1,1,0};z_1,z_2) + \Psi_N(4_1;z_1,z_3)
\end{align}
with
\begin{align}
\Phi_{N} (W_{0,1,1,0};z_1,z_2) 
&= \frac{1}{2\pi i} 
\lt[\varphi_r \lt( \frac{N\pi}{N+\frac{1}{2}}  - \pi z_1 - \pi z_2 -\frac{\pi}{2N+1}\rt) \rt.  \notag\\
&\qquad  - \varphi_r\lt(  \frac{N\pi}{N+\frac{1}{2}} - \pi z_2- \frac{\pi}{2N+1}\rt) \notag \\
&\qquad  + \varphi_r \lt(\pi z_2 + \frac{\pi}{2N+1}\rt)  - \varphi_r\lt( \pi z_1 + \pi z_2 + \frac{\pi}{2N+1}\rt) \notag\\
&\lt.\qquad  + \varphi_r\lt( \pi z_1 + \frac{\pi}{2N+1}\rt) \rt] \\
\Psi_N\lt(4_1; z_1,z_3 \rt) 
&= \frac{1}{2\pi i} \lt[ \varphi_r \lt( - \pi z_3 + 2\pi z_1 + \frac{\pi}{2N+1} \rt)  \rt.\notag\\
&\lt.\qquad - \varphi_r\lt( \pi z_3 + 2\pi z_1 + \frac{3\pi}{2N+1} \rt) \rt] - 4\pi i z_1z_3
\end{align}

Let $N\to \infty$. We have
\begin{align}
\Phi\lt(W_{0,1,1,0}(4_1); z_1, z_2, z_3\rt) 
&= \Phi(W_{0,1,1,0};z_1,z_2) + \Phi(4_1;z_1,z_3)\\
\Phi(W_{0,1,1,0};z_1,z_2) 
&= \frac{1}{2\pi i} \lt[ \Li\lt(e^{- 2\pi i z_1 - 2\pi i z_2}\rt) - \Li\lt( e^{ - 2\pi i z_2 }\rt) + \Li \lt(e^{2\pi i z_2}\rt) \rt. \notag\\
&\qquad \lt. - \Li\lt(e^{2\pi iz_1+ 2\pi i z_2}\rt)
 + \Li\lt(e^{2\pi i z_1}\rt)  \rt] \\
\Psi\lt(4_1; z_1,z_3 \rt) 
&= \frac{1}{2\pi i}\lt[ \Li(e^{-2\pi i z_3 + 4 \pi i z_1}) - \Li(e^{2\pi i z_3 + 4 \pi i z_1}) \rt] - 4\pi i z_1z_3
\end{align}
Note that the functions above are well-defined on
$$ D = \lt\{(z_1,z_2,z_3) \mid \lt|z_1-\frac{1}{2}\rt|, \lt|z_2 - \frac{1}{4}\rt| , \lt|z_3 - \frac{5}{6}\rt| < \eta \rt\}$$
for some sufficiently small $\eta>0$.

It is important to note that the function $\Psi\lt(4_1; z_1,z_3 \rt)$ is the same as the function $\Phi^{(s)}(z) - 2\pi i z$ studied in \cite{WA17} with $s=2z_1$ and $z=z_3$. This suggests that the parameter $2z_1$ can be understood as choosing the hyperbolic structure of the figure eight knot complement. In particular, to get the complete hyperbolic structure of the figure eight knot complement, we should take $z_1=\frac{1}{2}$.

By (\ref{QtoC}) and considering the Talyor series expansion (see Lemma 2 in \cite{W19} for a similar computation), we have
\begin{align*}
&\exp\lt( \lt(N+\frac{1}{2}\rt) \Phi_{N}(W_{0,1,1,0}(4_1), z_1,z_2,z_3) \rt)\\
&= E(W_{0,1,1,0}(4_1); z_1,z_2,z_3)
\exp\lt( \lt(N+\frac{1}{2}\rt) \Phi(W_{0,1,1,0}(4_1), z_1,z_2,z_3) \rt)\lt(1+O\lt(\frac{1}{N+\frac{1}{2}}\rt)\rt),
\end{align*}
where
\begin{align*}
E(W_{0,1,1,0}(4_1); z_1,z_2,z_3)
&= \exp\lt(\log(1-e^{-2\pi i z_1 - 2\pi i z_2}) - \log(1-e^{-2\pi i z_2}) - \frac{1}{2}\log(1-e^{2\pi i z_2}) \rt.\\
&\qquad + \frac{1}{2}\log(1-e^{2\pi i z_1 + 2\pi i z_2}) - \frac{1}{2}\log(1-e^{2\pi i z_1}) - \frac{1}{2}\log(1-e^{-2\pi i z_3+4\pi i z_1}) \\
&\qquad \lt.  + \frac{3}{2}\log(1-e^{2\pi i z_3+4\pi i z_1}) \rt)
\end{align*}

In particular, $E(W_{0,1,1,0}(4_1); z_1,z_2,z_3)$ is always nonzero and does not affect the exponential growth rate.

Next, note that the critical point equations for the potential function $\Phi(W_{0,1,1,0}(4_1);z_1,z_2,z_3)$ are given by 
\begin{empheq}[left = \empheqlbrace]{align}
0 = \frac{\partial \Phi(W_{0,1,1,0}(4_1); z_1,z_2,z_3) }{\partial z_1} &= \frac{\partial \Phi(W_{0,1,1,0};z_1,z_2)}{\partial z_1}+\frac{\partial \Psi(4_1;z_1,z_3)}{\partial z_1} \\
0 = \frac{\partial \Phi(W_{0,1,1,0}(4_1); z_1,z_2,z_3) }{\partial z_2} &= \frac{\partial \Phi(W_{0,1,1,0};z_1,z_2)}{\partial z_2}\\
0 = \frac{\partial \Phi(W_{0,1,1,0}(4_1); z_1,z_2,z_3) }{\partial z_3} &= \frac{\partial \Psi(4_1;z_1,z_3)}{\partial z_3}
\end{empheq}

Recall from Section 2.2 of \cite{W19} that $(z_1,z_2)=(\frac{1}{2},\frac{1}{4})$ is the critical point of the potential function $\Phi(W_{0,1,1,0};z_1,z_2)$ (which is $\Phi(WL,z_1,z_2)$ in \cite{W19}). As a result, if we put $(z_1,z_2)=(\frac{1}{2},\frac{1}{4})$, the system of critical point equations become
\begin{empheq}[left = \empheqlbrace]{align}
0 &= \frac{\partial \Psi(4_1; \frac{1}{2},z_3)}{\partial z_1} 
= -2\log(1-e^{-2\pi i z_3}) + 2\log(1-e^{2\i i z_3}) - 4\pi i z_3\\
0 &= \frac{\partial \Psi(4_1; \frac{1}{2},z_3)}{\partial z_3} 
= \log(1- e^{-2\pi i z_3}) + \log(1-e^{2\i i z_3}) - 2\pi i
\end{empheq}
Note that the second equation is the critical point equation of the potential function $\Psi(4_1;\frac{1}{2},z_3)$ (which is the function $\Phi^{(1)}(z)$ in \cite{WA17} minus to $2\pi i z_3$). Recall from Section 1.4 in \cite{WA17} that $z_3=\frac{5}{6}$ is a critical point of the function $\Phi^{(1)}(z)$. In particular, when $z_3 = \frac{5}{6}$, we have
\begin{align}
\frac{\partial \Psi(4_1; \frac{1}{2}, \frac{5}{6})}{\partial z_3}  = -2\pi i
\end{align}
Finally, by direct computation, we have
\begin{align}
\frac{\partial \Psi(4_1; \frac{1}{2},\frac{5}{6})}{\partial z_1} 
= -2\log(1-e^{-5\pi i/3}) + 2\log(1-e^{5\pi i/3}) - 10\pi i /3
=  -2\pi i
\end{align}
As a result, the point $(z_1,z_2,z_3) = \lt(\frac{1}{2}, \frac{1}{4}, \frac{5}{6}\rt)$ is indeed a critical point of the Fourier coefficient 
$$\Phi(W_{0,1,1,0}(4_1); z_1,z_2,z_3) + 2\pi i z_1 + 2\pi i z_3$$
with critical value
\begin{align}
\Phi\lt(W_{0,1,1,0}(4_1); \frac{1}{2}, \frac{1}{4}, \frac{5}{6}\rt) + 2\pi i \lt(\frac{1}{2}\rt) + 2\pi i \lt(\frac{5}{6}\rt) 
&= \frac{v_3 ||W_{0,1,1,0}(4_1) || + 2 \pi^2 i}{2\pi}
\end{align}

Moreover, the Hessian of the potential function at the critical point is given by
\begin{align}
\Hess\lt(\Phi\lt(W_{0,1,1,0}(4_1); \frac{1}{2},\frac{1}{4},\frac{5}{6}\rt)\rt)
=
2\pi i
\begin{pmatrix}
i + \frac{7}{2} & i & 2+2\sqrt{3} i \\
i & 2i & 0 \\
2+2\sqrt{3}i & 0 & \sqrt{3}i
\end{pmatrix}
\end{align}
with non-zero determinant.

Therefore, by the Poisson summation formula and saddle point approximation (Proposition 4.6 and Proposition 3.5 in \cite{O52}) (see for example Section 2.2 in \cite{W19} about how the apply them given the above coniditions), we have

\begin{align}
&J_{N} (W_{0,1,1,0}(4_1), e^{\frac{2\pi i}{N+\frac{1}{2}}}) \notag\\
&\stackrel[N \to \infty]{\sim}{} 
- \frac{N+\frac{1}{2}}{\pi} \exp\lt(  -\frac{r}{4\pi i} \varphi_r \lt(  \frac{\pi}{2N+1}\rt)\rt)
 \sum_{D_\eta }  e^{\frac{2\pi i}{N+\frac{1}{2}}(-\frac{n}{2}+\frac{l}{2}-k)}
 \notag\\ 
&\qquad \times
\exp\lt( \lt(N+\frac{1}{2}\rt) \Phi_{N}\lt(W_{0,1,1,0}(4_1); \frac{n}{N+\frac{1}{2}}, \frac{l}{N+\frac{1}{2}}, \frac{k}{N+\frac{1}{2}}\rt) \rt)
 \notag\\
&\stackrel[N \to \infty]{\sim}{} 
\frac{-(N+\frac{1}{2})}{\pi}\exp\lt(-\frac{r}{4\pi i}\varphi_r \lt( \frac{\pi}{2N+1} \rt) \rt) 
\lt(N+\frac{1}{2}\rt)^{3}
\int\int_{D} e^{-\pi i z_1 + \pi i z_2 - 2\pi i z_3}
 \notag\\ 
&\qquad \times
\exp\lt( \lt(N+\frac{1}{2}\rt) (\Phi_{N}\lt(W_{0,1,1,0}(4_1); z_1, z_2, z_3\rt) 
+ 2\pi iz_1 + 2\pi i z_3) \rt)  dz_1 dz_2 dz_3 \notag \\
&\stackrel[N \to \infty]{\sim}{} 
\frac{-(N+\frac{1}{2})}{\pi}\exp\lt(-\frac{r}{4\pi i}\varphi_r \lt( \frac{\pi}{2N+1} \rt) \rt) 
\lt(N+\frac{1}{2}\rt)^{3}
\int\int_{D} e^{-\pi i z_1 + \pi i z_2 - 2\pi i z_3}
 \notag\\ 
&\qquad \times
E(W_{0,1,1,0}(4_1); z_1,z_2,z_3)\exp\lt( \lt(N+\frac{1}{2}\rt) (\Phi\lt(W_{0,1,1,0}(4_1); z_1, z_2, z_3\rt) + 2\pi i z_1 + 2\pi i z_3 )\rt)  dz_1 dz_2 dz_3 \notag \\
&\stackrel[N \to \infty]{\sim}{} 
\frac{-(N+\frac{1}{2})}{\pi}\exp\lt(-\frac{r}{4\pi i}\varphi_r \lt( \frac{\pi}{2N+1} \rt) \rt) \lt(N+\frac{1}{2}\rt)^{3/2} e^{-\frac{23\pi i}{12}} \notag\\
&\qquad \times \frac{(2\pi)^{3/2}E\lt(W_{0,1,1,0}(4_1); \frac{1}{2},\frac{1}{4},\frac{5}{6}\rt)}{\sqrt{ -\det \Hess(\Phi(W_{0,1,1,0}(4_1); \frac{1}{2},\frac{1}{4},\frac{5}{6}))}} \exp \lt(  \frac{N+\frac{1}{2}}{2\pi} \lt(v_3 ||W_{0,1,1,0}(4_1) || +  2 \pi^2 i \rt) \rt)
\end{align}

By Lemma A.3 in \cite{O52}, we have
\begin{align}\label{QDFcoe}
\exp\lt(-\frac{r}{4\pi i}\varphi_r \lt( \frac{\pi}{2N+1} \rt) \rt) 
&\stackrel[N \to \infty]{\sim}{ } e^{-\frac{\pi i}{4}}\lt(N+\frac{1}{2}\rt)^{-\frac{1}{2}}\exp\lt( \frac{N+\frac{1}{2}}{2\pi} \frac{\pi^2 i}{6} \rt) 
\end{align}
In particular, this factor decays polynomially. Thus, we have
\begin{align}
\lim_{N\to\infty} \frac{2\pi}{N+\frac{1}{2}} \log |J_{N} (W_{0,1,1,0}(4_1), e^{\frac{2\pi i}{N+\frac{1}{2}}})|  
=v_3 ||W_{0,1,1,0}(4_1) ||
\end{align}

\subsection{Volume conjecture for $TV_r(\SS^3\backslash W_{0,1,1,0}(4_1), e^{\frac{2\pi i}{N+\frac{1}{2}}})$}\label{W021041tv}
Next, by Lemma~\ref{uppbd1} and \ref{uppbdW}, for any $0\leq M \leq N$, we have
$$ \limsup_{N \to \infty} \frac{2\pi}{N+\frac{1}{2}}\lt|J_{M} \lt( W_{0,1,1,0}(4_1) , e^{\frac{2\pi i}{N+\frac{1}{2}}} \rt) \rt| \leq v_3 ||W_{0,1,1,0}(4_1) || $$
As a result, by squeeze theorem, it is easy to show that
\begin{align}
&\lim_{\substack{ r\to \infty \\r \text{ odd}}} \frac{2\pi}{r} \log|TV_r (\SS^3 \backslash W_{0,1,1,0}(4_1))| \notag\\
&= \lim_{\substack{ r\to \infty \\r \text{ odd}}} \frac{2\pi}{r} 
\log\lt|2\lt( \frac{2\sin(\frac{2\pi}{r})}{\sqrt{r}}\rt)^{2} \sum_{1\leq M \leq \frac{r-1}{2}}\left|\tilde{J}_{M}\left(L,e^{\frac{2\pi i }{N+\frac{1}{2}}}\right)\right|^2 \rt|^2 \notag\\
&=
\lim_{N \to \infty} \frac{2\pi}{N+\frac{1}{2}}\log\lt|J_{N} \lt( W_{0,1,1,0}(4_1) , e^{\frac{2\pi i}{N+\frac{1}{2}}} \rt) \rt| \notag\\
&= v_3 ||W_{0,1,1,0}(4_1) || 
\end{align}

\subsection{A summary of the strategy of proving the volume conjectures}\label{strategy}
In this subsection, we summarize the strategy we used to prove the volume conjectures. This strategy will be used repeatedly in the rest of this paper.
\begin{enumerate}
\item
To compute the asymptotic expansion formula for $J_{N}(W_{0,1,1,0}(4_1),e^{\frac{2\pi i}{N+\frac{1}{2}}})$, 
\begin{enumerate}
\item {\bf{(maximum estimation)}} by using Lemma~\ref{uppbd1} and \ref{uppbdW}, we show that in order to study the asymptotic expansion formula of $J_{N}(W_{0,1,1,0}(4_1),e^{\frac{2\pi i}{N+\frac{1}{2}}})$, we only need to consider a small neighborhood around the critical point;
\item {\bf{(Poisson Summation formula and the saddle point approximation)}} by using the Poisson Summation formula and the saddle point approximation (Proposition 4.6 and Proposition 3.5 in \cite{O52}), we obtain the asymptotic expansion formula for $J_{N}(W_{0,1,1,0}(4_1),e^{\frac{2\pi i}{N+\frac{1}{2}}})$.
\end{enumerate}
\item
To prove the volume conjecture for $TV_r (\SS^3 \backslash W_{0,1,1,0}(4_1))$, 
\begin{enumerate}
\item by using Lemma~\ref{uppbd1} and \ref{uppbdW} as well as Theorem~\ref{relationship}, we can show that 
$$ \limsup_{\substack{ r\to \infty \\r \text{ odd}}} \frac{2\pi}{r} \log|TV_r (\SS^3 \backslash W_{0,1,1,0}(4_1))|  
\leq v_3 ||W_{0,1,1,0}(4_1) || $$
\item together with the result that
$$ \lim_{N \to \infty} \frac{2\pi}{N+\frac{1}{2}}\log\lt|J_{N} \lt( W_{0,1,1,0}(4_1) , e^{\frac{2\pi i}{N+\frac{1}{2}}} \rt) \rt| = v_3 ||W_{0,1,1,0}(4_1) || ,$$
by Theorem~\ref{relationship} and the squeeze theorem, we have
$$ \lim_{\substack{ r\to \infty \\r \text{ odd}}} \frac{2\pi}{r} \log|TV_r (\SS^3 \backslash W_{0,1,1,0}(4_1))|  
= v_3 ||W_{0,1,1,0}(4_1) || $$
\end{enumerate}
\end{enumerate}

\subsection{Volume conjectures for $J_{\vec{N}} (W_{a,1,c,d}(4_1), e^{\frac{2\pi i}{N+\frac{1}{2}}})$}\label{Wa1cd41pf}

By the same method discussed in Section~\ref{W021041}, the $\vec{N}$-th colored Jones polynomials for $W_{a,1,c,d}(4_1)$ is given by \cite{HZ07,V08}

\begin{align}
&J_{\vec{N}} (W_{a,1,c,d}(4_1), t) \notag\\
&= \frac{t^{\frac{N^2-1}{2} + \frac{N(N-1)}{2}}}{1-t}
\sum_{n=0}^{N-1} \sum_{l_1=0}^{M_2-1-n}\dots \sum_{l_c=0}^{M_2-1-n} \sum_{l_1'=0}^{M_2-1-n}\dots \sum_{l_d'=0}^{M_2-1-n} \sum_{k=0}^{2n}
t^{an(n+1) - n - k} \notag\\
&\prod_{\gamma=1}^c \lt(t^{-N(l_\gamma+n) }\frac{(t)_{N-l_\gamma-1}(t)_{l_\gamma+n}}{(t)_n(t)_{N-l_\gamma-n-1}(t)_l}\rt) 
\prod_{\gamma=1}^d \lt(t^{N(l'_\gamma+n) }\frac{(t^{-1})_{N-l'_\gamma-1}(t^{-1})_{l'_\gamma+n}}{(t^{-1})_n(t^{-1})_{N-l'_\gamma-n-1}(t^{-1})_l'}\rt) 
\lt(t^{-2nk} \frac{(t)_{2n+1+k}}{(t)_{2n-k}}  \rt)
\end{align}

Let $D_\eta$ to be the set
\begin{align*}
D_{\eta} = \lt\{(n,l_1,\dots, l_{c}, l'_1, \dots, l_{d}', k) \mid \lt|\frac{n}{N+\frac{1}{2}}-\frac{1}{2}\rt|, \lt|\frac{l_1}{N+\frac{1}{2}} - \frac{1}{4}\rt| ,\dots, \lt|\frac{l_{d}'}{N+\frac{1}{2}} - \frac{1}{4}\rt|, \lt|\frac{k}{N+\frac{1}{2}} - \frac{5}{6}\rt| < \eta \rt\}
\end{align*}
Similar to the arguments in Section~\ref{W021041}, we have
\begin{align}
&J_{\vec{N}} (W_{a,1,c,d}(4_1), e^{\frac{2\pi i}{N+\frac{1}{2}}}) \notag\\
&\stackrel[N \to \infty]{\sim}{} \frac{-\lt(N+\frac{1}{2}\rt)(-1)^{a } e^{\frac{2\pi i}{N+\frac{1}{2}}\lt(\frac{N^2-1}{2}+\frac{(c-d)N(N-1)}{2}\rt)}  e^{- \frac{a}{4}(2\pi i)(N+\frac{1}{2})} }{\pi} \notag\\
&\qquad \times \exp\lt(  -\frac{r}{4\pi i} \varphi_r \lt(  \frac{\pi}{2N+1}\rt)\rt)^{c}
\exp\lt(  \frac{r}{4\pi i} \varphi_r \lt(\pi  - \frac{\pi}{2N+1}\rt)\rt)^{d} 
\sum_{D_\eta}  
e^{\frac{2\pi i \lt(-1+a+\frac{c-d}{2}\rt)n+\sum_{\gamma=1}^c \frac{l_\gamma}{2} - \sum_{\gamma=1}^d\frac{l_\gamma}{2}-k}{N+\frac{1}{2}}} \notag \\
&\qquad\times \exp\lt( \lt(N+\frac{1}{2}\rt)\Phi_{N} \lt(W_{a,1,c,d}(4_1);\frac{n}{N+\frac{1}{2}},\frac{l_1}{N+\frac{1}{2}}, \dots, \frac{l_c}{N+\frac{1}{2}}, \frac{l_1'}{N+\frac{1}{2}} ,\dots, \frac{l_d'}{N+\frac{1}{2}} , \frac{k}{N+\frac{1}{2}}\rt) \rt) 
\end{align}
where
\begin{align*}
&\Phi_{N} \lt(W_{a,1,c,d}(4_1); z_1, z_2,\dots,z_{c+1}, z_{c+2},\dots, z_{c+d+1}, z_{c+d+2}\rt) \\
&= \Phi_{N} (W_{a,1,c,d};z_1, z_2,\dots,z_{c+1}, z_{c+2},\dots, z_{c+d+1}) + \Psi_N(4_1;z_1,z_{c+d+2})
\end{align*}
with
\begin{align*}
&\Phi_{N} (W_{a,1,c,d};z_1, z_2,\dots,z_{c+1}, z_{c+2},\dots, z_{c+d+1}) \\
&=  \frac{a}{2\pi i} \lt[ 2\pi i \lt(z_1-\frac{1}{2}  \rt) \rt]^2 
 + \sum_{i=\gamma}^c \psi_{N,N}(z_1,z_{\gamma+1}) + \sum_{\gamma=1}^d \kappa_{N,N}(z_1,z_{c+\gamma+1})
\end{align*}
\begin{align*}
\psi_{N}(z_1,z_{\gamma+1})
&= \frac{1}{2\pi i} \lt[\varphi_r \lt( \frac{N\pi}{N+\frac{1}{2}} - \pi z_1 - \pi z_{\gamma+1} -\frac{\pi}{2N+1}\rt) \rt.  
  \notag \\
&\qquad - \varphi_r\lt( \frac{N\pi}{N+\frac{1}{2}}  - \pi z_{\gamma+1}- \frac{\pi}{2N+1}\rt) + \varphi_r \lt(\pi z_{\gamma+1} + \frac{\pi}{2N+1}\rt) \notag\\
&\lt.\qquad  - \varphi_r\lt( \pi z_1 + \pi z_{\gamma+1} + \frac{\pi}{2N+1}\rt)  + \varphi_r\lt( \pi z_1 + \frac{\pi}{2N+1}\rt) \rt] \qquad \text{for $\gamma=1,\dots, c$}\\
\kappa_{N} (z_1,z_{c+i+1})
&= \frac{-1}{2\pi i} \lt[\varphi_r \lt( -\pi  \lt(\frac{N}{N+\frac{1}{2}} -1\rt) + \pi z_1 + \pi z_{c+i+1} + \frac{\pi}{2N+1}\rt) \rt.  \notag\\
&\qquad  - \varphi_r\lt( -\pi  \lt(\frac{N}{N+\frac{1}{2}} -1\rt) + \pi z_{c+i+1} + \frac{\pi}{2N+1}\rt) \notag \\
&\qquad  + \varphi_r \lt( \pi -\pi z_{c+i+1} - \frac{\pi}{2N+1}\rt)  - \varphi_r\lt( \pi - \pi z_1 - \pi z_{c+i+1} - \frac{\pi}{2N+1}\rt) \notag\\
&\lt.\qquad  + \varphi_r\lt( \pi - \pi z_1 - \frac{\pi}{2N+1}\rt) \rt] 
\qquad\qquad\qquad\qquad\qquad\qquad\qquad \text{for $\gamma=1,\dots, d$}\\
\Psi_N\lt(4_1; z_1,z_{c+d+2} \rt) 
&= \frac{1}{2\pi i} \lt[ \varphi_r \lt( - \pi z_{c+d+2} + 2\pi z_1 + \frac{\pi}{2N+1} \rt) \rt. \notag\\
&\qquad \lt.- \varphi_r\lt( \pi z_{c+d+2} - \pi + 2\pi z_1 + \frac{3\pi}{2N+1} \rt) \rt] - 4\pi i z_1z_{c+d+2}
\end{align*}

Take $N \to \infty$, we have
\begin{align*}
&\Phi \lt(W_{a,1,c,d}(4_1); z_1, z_2,\dots,z_{c+1}, z_{c+2},\dots, z_{c+d+1}, z_{c+d+2}\rt) \\
&= \Phi (W_{a,1,c,d};z_1, z_2,\dots,z_{c+1}, z_{c+2},\dots, z_{c+d+1}) + \Psi(4_1;z_1,z_{c+d+2})
\end{align*}
with
\begin{align*}
&\Phi (W_{a,1,c,d};z_1, z_2,\dots,z_{c+1}, z_{c+2},\dots, z_{c+d+1}) \\
&=  \frac{a}{2\pi i} \lt[ 2\pi i \lt(z_1-\frac{1}{2}  \rt) \rt]^2  
 + \sum_{\gamma=1}^c \psi(z_1,z_{\gamma+1}) + \sum_{\gamma=1}^d \kappa(z_1,z_{c+\gamma+1})
\end{align*}
\begin{align*}
\psi_{i} (z_1, z_{\gamma+1})
&= \frac{1}{2\pi i} \lt[   
\Li\lt(e^{- 2\pi i z_1 - 2\pi i z_{\gamma+1}}\rt) - \Li\lt( e^{- 2\pi i z_{\gamma+1} }\rt) \rt.\\
&\qquad \lt. + \Li \lt(e^{2\pi i z_{\gamma+1}}\rt) - \Li\lt(e^{2\pi iz_1+ 2\pi i z_{\gamma+1}}\rt) + \Li\lt(e^{2\pi i z_1}\rt) \rt] \\
& \qquad\qquad \text{for $\gamma=1,\dots, c,$}\\
\kappa_i(z_1,z_{c+\gamma+1}) 
&= \frac{1}{2\pi i} [ - \Li\lt(e^{2\pi i z_1 + 2\pi i z_{c+\gamma+1}}\rt) + \Li\lt( e^{ 2\pi i z_{c+\gamma+1} }\rt) \\
&\qquad - \Li \lt(e^{-2\pi i z_{c+\gamma+1}}\rt) + \Li\lt(e^{-2\pi iz_1- 2\pi i z_{c+\gamma+1}}\rt) - \Li\lt(e^{-2\pi i z_1}\rt) ]\\
& \qquad\qquad \text{for $\gamma=1,\dots, d,$}\\
\Psi\lt(4_1; z_1,z_{c+d+2} \rt) 
&= \frac{1}{2\pi i}\lt[ \Li(e^{-2\pi i z_{c+d+2} + 4 \pi i z_1}) - \Li(e^{2\pi i z_{c+d+2} + 4 \pi i z_1}) \rt] - 4\pi i z_1z_{c+d+2}
\end{align*}

In the above, the functions are well-defined on
\begin{align*}
D = \lt\{(z_1,z_2,\dots, z_{c+1}, z_{c+2}, \dots, z_{c+d+1}, z_{c+d+2}) \mid \lt|z_1 - \frac{1}{2}\rt|, \lt|z_2 - \frac{1}{4}\rt| ,\dots, \lt|z_{c+d+1} - \frac{1}{4}\rt|, \lt|z_{c+d+2} - \frac{5}{6}\rt| < \eta \rt\}
\end{align*}
for some sufficiently small $\eta>0$.

By (\ref{QtoC}) and considering the Talyor series expansion (see Lemma 2 in \cite{W19} for a similar computation), we have
\begin{align*}
&\exp\lt( \lt(N+\frac{1}{2}\rt) \Phi_{N} (W_{a,1,c,d}(4_1);z_1, z_2,\dots,z_{c+1}, z_{c+2},\dots, z_{c+d+1}) \rt)\\
&= E(W_{a,1,c,d}(4_1);z_1, z_2,\dots,z_{c+1}, z_{c+2},\dots, z_{c+d+1}) \\
&\qquad \exp\lt( \lt(N+\frac{1}{2}\rt) \Phi (W_{a,1,c,d}(4_1) ;z_1, z_2,\dots,z_{c+1}, z_{c+2},\dots, z_{c+d+1}) \rt)\lt(1+O\lt(\frac{1}{N+\frac{1}{2}}\rt)\rt),
\end{align*}
where
\begin{align*}
&E(W_{a,1,c,d}(4_1);z_1, z_2,\dots,z_{c+1}, z_{c+2},\dots, z_{c+d+1}) \\
&= \exp\lt( \sum_{\gamma=1}^c \lt(\log(1-e^{-2\pi i z_1 - 2\pi i z_{\gamma+1}}) - \log(1-e^{-2\pi i z_{\gamma+1}}) - \frac{1}{2}\log(1-e^{2\pi i z_{\gamma+1}}) \rt.\rt.\\
&\qquad\qquad\qquad \lt. + \frac{1}{2}\log(1-e^{2\pi i z_1 + 2\pi i z_{\gamma+1}}) - \frac{1}{2}\log(1-e^{2\pi i z_1}) \rt) \\
&\qquad\quad + \sum_{\gamma=1}^d \lt(\log(1-e^{2\pi i z_1 + 2\pi i z_{c+\gamma+1}}) - \log(1-e^{2\pi i z_{c+\gamma+1}}) - \frac{1}{2}\log(1-e^{-2\pi i z_{c+\gamma+1}}) \rt.\\
&\qquad\qquad\qquad \lt. + \frac{1}{2}\log(1-e^{-2\pi i z_1 - 2\pi i z_{c+\gamma+1}}) - \frac{1}{2}\log(1-e^{-2\pi i z_1}) \rt) \\
&\qquad\qquad \lt. - \frac{1}{2}\log(1-e^{-2\pi i z_3+4\pi i z_1}) 
+ \frac{3}{2}\log(1-e^{2\pi i z_3+4\pi i z_1}) \rt)
\end{align*}

Note that the critical point equations for $\Phi (W_{a,1,c,d}(4_1);z_1, z_2,\dots,z_{c+1}, z_{c+2},\dots, z_{c+d+1})$ are given by

\begin{empheq}[left = \empheqlbrace]{align}
0 &= 2a\lt(z_1-\frac{1}{2}\rt) + \sum_{\gamma=1}^c \frac{\partial \psi(z_1,z_{\gamma+1})}{\partial z_1} + \sum_{\gamma=1}^d \frac{\partial \kappa(z_1,z_{\gamma+1})}{\partial z_1} + \frac{\partial \Psi(4_1;z_1,z_3)}{\partial z_1} \\
0 &=  \frac{\partial \psi(z_1,z_{\gamma+1})}{\partial z_{\gamma+1}} \qquad \hspace{9pt} \text{for $\gamma=1,\dots, c$} \\
0 &=  \frac{\partial \kappa(z_1,z_{c+\gamma+1})}{\partial z_{c+\gamma+1}} \qquad \text{for $\gamma=1,\dots, d$} \\
0 &= \frac{\partial \Psi(4_1;z_1,z_{c+d+2})}{\partial z_{c+d+2}}
\end{empheq}

Using the same arguments in Section~\ref{W021041}, it is straightforward to verify that 
$$(z_1,z_2,\dots,z_c, z_{c+1},\dots, z_{c+d+2}) = \lt(\frac{1}{2},\frac{1}{4},\dots, \frac{1}{4} , \frac{5}{6}\rt)$$
is the solution of critical point of the Fourier coefficient
$$ \Phi \lt(W_{a,1,c,d}(4_1); z_1, z_2,\dots,z_{c+1}, z_{c+2},\dots, z_{c+d+1}, z_{c+d+2}\rt) + 2\pi i z_1 + 2\pi i z_{c+d+2}$$
with critical value \cite{W19}
\begin{align*}
\Phi\lt(W_{a,1,c,d}(4_1);\frac{1}{2},\frac{1}{4},\dots, \frac{1}{4}, \frac{5}{6}\rt)
&= \frac{1}{2\pi i} \sum_{\gamma=1}^c \psi_i \lt(\frac{1}{2},\frac{1}{4}\rt) + \frac{1}{2\pi i} \sum_{\gamma=1}^d \kappa_i \lt(\frac{1}{2},\frac{1}{4} \rt) + \Psi\lt(4_1; \frac{1}{2},\frac{5}{6} \rt) + \frac{11\pi i}{6} \\
&= \frac{1}{2\pi} \lt[\Vol(\SS^3\backslash W_{a,1,c,d}) + \Vol(\SS^3 \backslash 4_1) + (c-d)\frac{\pi^2 i}{12} \rt] + \frac{11\pi i}{6} \\
&= \frac{1}{2\pi} \lt[ v_3 || \SS^3\backslash W_{a,1,c,d}(4_1) || + (c-d)\frac{\pi^2 i}{12} \rt] + \frac{11\pi i}{6}
\end{align*}

Moreover, by direct computation, the Hessian of the potential function evaluated at the critical point is a $(c+d+2)\times (c+d+2)$ matrix given by
$$
\Hess\lt(\Phi\lt(W_{a,1,c,d}(4_1);\frac{1}{2},\frac{1}{4},\dots, \frac{1}{4}, \frac{5}{6}\rt)\rt)
=
2\pi i
\begin{pmatrix}
(c+d)i-\frac{c-d}{2} + 2a & i & i & \dots & i & 2 + 2\sqrt{3}i \\
i& 2i & 0 & \dots & 0 & 0 \\
i& 0 & 2i & \dots & 0 & 0 \\
\vdots & \vdots & \vdots & \vdots &\vdots & \vdots \\
i& 0 & 0 & \dots & 2i & 0  \\
2+ 2\sqrt{3}i & 0 & 0 & \dots & 0 & \sqrt{3}i
\end{pmatrix}
$$

For any $a_1,a_2,a_3,a_4, a_5\in\CC$, one can show that the determinant of the $(c+d+2)\times (c+d+2)$ matrix
$$
\begin{pmatrix}
a_1&a_2&a_2&a_2&\dots&a_2&a_3 \\
a_2&a_4&0&0&\dots&0&0 \\
a_2&0&a_4&0&\dots&0&0 \\
&&&\vdots&& \\
a_2&0&0&0&\dots&a_4&0 \\
a_3&0&0&0&\dots&0&a_5
\end{pmatrix}
$$
is given by
$$
\begin{vmatrix}
a_1&a_2&a_2&a_2&\dots&a_2&a_3 \\
a_2&a_4&0&0&\dots&0&0 \\
a_2&0&a_4&0&\dots&0&0 \\
&&&\vdots&& \\
a_2&0&0&0&\dots&a_4&0 \\
a_3&0&0&0&\dots&0&a_5
\end{vmatrix}
=
a_4^{n-1}(a_1 a_4 a_5 - na_2^2 a_5 - a_3^2 a_4)
$$

As a result,
\begin{align*}
&\det\Hess\lt(\Phi\lt(W_{a,1,c,d}(4_1);\frac{1}{2},\frac{1}{4},\dots, \frac{1}{4}, \frac{5}{6}\rt)\rt) \\
&=
(2\pi i)^{c+d+2}
\begin{vmatrix}
(c+d)i-\frac{c-d}{2}+2a & i & i & \dots & i & 2 + 2\sqrt{3}i \\
i& 2i & 0 & \dots & 0 & 0 \\
i& 0 & 2i & \dots & 0 & 0 \\
\vdots & \vdots & \vdots & \vdots &\vdots & \vdots \\
i& 0 & 0 & \dots & 2i & 0  \\
2+ 2\sqrt{3}i & 0 & 0 & \dots & 0 & \sqrt{3}i
\end{vmatrix}\\
&= (2\pi i )^{c+d+2} (2i)^{c+d-1} 
\lt[
\lt( (c+d)i - \lt(\frac{c-d}{2}\rt) + 2a\rt) (2i)(\sqrt{3}i) - (c+d)(i)^2(\sqrt{3}i)
- (2+\sqrt{3}i)^2(2i)
\rt]
\\
&\neq 0
\end{align*}

By the argument discussed in Section~\ref{strategy}, we have
\begin{align}
&J_{\vec{N}} (W_{a,1,c,d}(4_1), e^{\frac{2\pi i}{N+\frac{1}{2}}}) \notag\\
&\stackrel[N \to \infty]{\sim}{} \frac{-\lt(N+\frac{1}{2}\rt)(-1)^{a } e^{\frac{2\pi i}{N+\frac{1}{2}}\lt(\frac{N^2-1}{2}+\frac{(c-d)N(N-1)}{2}\rt)}  e^{- \frac{a}{4}(2\pi i)(N+\frac{1}{2})} }{\pi} \notag\\
&\qquad \times \exp\lt(  -\frac{r}{4\pi i} \varphi_r \lt(  \frac{\pi}{2N+1}\rt)\rt)^{c}
\exp\lt(  \frac{r}{4\pi i} \varphi_r \lt(\pi  - \frac{\pi}{2N+1}\rt)\rt)^{d} 
\sum_{D_\delta}  e^{\frac{2\pi i \lt(-1+a+\frac{c-d}{2}\rt)n+\sum_{\gamma=1}^c \frac{l_\gamma}{2} - \sum_{\gamma=1}^d\frac{l_\gamma}{2}-k}{N+\frac{1}{2}}}  \notag \\
&\qquad\times \exp\lt( \lt(N+\frac{1}{2}\rt)\Phi_{N} \lt(W_{a,1,c,d}(4_1);\frac{n}{N+\frac{1}{2}},\frac{l_1}{N+\frac{1}{2}}, \dots, \frac{l_c}{N+\frac{1}{2}}, \frac{l_1'}{N+\frac{1}{2}} ,\dots, \frac{l_d'}{N+\frac{1}{2}} , \frac{k}{N+\frac{1}{2}}\rt) \rt) \notag\\
&\stackrel[N \to \infty]{\sim}{} \frac{-\lt(N+\frac{1}{2}\rt)(-1)^{a } e^{\frac{2\pi i}{N+\frac{1}{2}}\lt(\frac{N^2-1}{2}+\frac{(c-d)N(N-1)}{2}\rt)}  e^{- \frac{a}{4}(2\pi i)(N+\frac{1}{2})} e^{\pi i \lt(-\frac{8}{3}+a+\frac{3(c-d)}{4}\rt)}}{\pi} \notag\\
&\qquad \times \exp\lt(  -\frac{r}{4\pi i} \varphi_r \lt(  \frac{\pi}{2N+1}\rt)\rt)^{c}
\exp\lt(  \frac{r}{4\pi i} \varphi_r \lt(\pi  - \frac{\pi}{2N+1}\rt)\rt)^{d} 
\times  \lt(N+\frac{1}{2}\rt)^{\frac{c+d+2}{2}}\notag \\
&\qquad\times \frac{(2\pi)^{\frac{c+d+2}{2}}E(W_{a,1,c,d}(4_1);\frac{1}{2},\frac{1}{4},\dots,\frac{1}{4},\frac{5}{6})}{\sqrt{-\det\Hess(\Phi(W_{a,1,c,d}(4_1);\frac{1}{2},\frac{1}{4},\dots,\frac{1}{4},\frac{5}{6}))}}
\exp \lt(
\frac{1}{2\pi} \lt( v_3 || \SS^3\backslash W_{a,1,c,d}(4_1) || + (c-d)\frac{\pi^2 i}{12} \rt) + \frac{11\pi i}{6}
\rt)
\end{align}

By Lemma A.3 in \cite{O52}, we have
\begin{align}\label{QDFcoe}
\exp\lt(-\frac{r}{4\pi i}\varphi_r \lt( \frac{\pi}{2N+1} \rt) \rt) 
&\stackrel[N \to \infty]{\sim}{ } e^{-\frac{\pi i}{4}}\lt(N+\frac{1}{2}\rt)^{-\frac{1}{2}}\exp\lt( \frac{N+\frac{1}{2}}{2\pi} \frac{\pi^2 i}{6} \rt) \\
\exp\lt(\frac{r}{4\pi i}\varphi_r \lt(\pi - \frac{\pi}{2N+1} \rt) \rt) 
&\stackrel[N \to \infty]{\sim}{ } e^{\frac{\pi i}{4}}\lt(N+\frac{1}{2}\rt)^{-\frac{1}{2}}\exp\lt( \frac{N+\frac{1}{2}}{2\pi} \lt(\frac{-\pi^2 i}{6}\rt) \rt)
\end{align}

In particular, they do not affect the exponential growth rate of the colored Jones polynomials. This completes the proof of Theorem~\ref{mainthm1}. To prove Corollary~\ref{cablingtv1} for $W_{a,b,c,d}(4_1)$, note that by Theorem~\ref{relationship}, we have
\begin{align*}
TV_{r}\left(\SS^3 \backslash W_{a,b,c,d}(4_1), e^{\frac{2\pi i}{r}}\right) 
&= 2^{n-1}\lt( \frac{2\sin(\frac{2\pi}{r})}{\sqrt{r}}\rt)^{2} \sum_{1\leq \vec{M} \leq \frac{r-1}{2}}\left|J_{\vec{M}}\left(W_{a,b,c,d}(4_1),e^{\frac{2\pi i }{N+\frac{1}{2}}}\right)\right|^2 \\
&\geq 2^{n-1}\lt( \frac{2\sin(\frac{2\pi}{r})}{\sqrt{r}}\rt)^{2} 
\left|J_{\vec{N}}\left(W_{a,1,c,d}(4_1),e^{\frac{2\pi i }{N+\frac{1}{2}}}\right)\right|^2
\end{align*}
Corollary~\ref{cablingtv1} then follows from the arguments discussed in Section~\ref{strategy}.

\section{Iterated Whitehead double of the figure eight knot}\label{Wp0110W021041}

In this section, we are going to compute the $N$-th colored Jones polynomials of the family of links $W_{0,1,1,0}^{p+1}(4_1)$, where $p\geq 0$. Again, by equation (2.9) in Section 2 of \cite{HZ07}, by induction we have
\begin{align}
J_{N} (W_{0,1,1,0}^{p+1} (4_1),t) 
&= \sum_{n_1=0}^{N-1} C(n_1,t;N) \cdot J_{N,2n_1+1}(W_{0,1,1,0}^{p}(4_1),t) \notag \\
&= \sum_{n_1=0}^{N-1}\sum_{n_2=0}^{2n_1+1}\dots\sum_{n_{p+1}=0}^{2n_p+1}  
C(n_1,t;N)  \prod_{\gamma=1}^{p} C(n_{\gamma+1},t;2n_\gamma+1) \cdot J_{2n_{p+1}+1}(4_1) \notag\\
&= \frac{t^{\frac{N^2-1}{2} + \frac{N(N-1)}{2}}}{1-t}
\sum_{n_1=0}^{N-1}\sum_{n_2=0}^{2n_1+1}\dots\sum_{n_{p+1}=0}^{2n_p+1} 
\sum_{l_1=0}^{N-1-n_1}\sum_{l_2=0}^{2n_1-n_2}\sum_{l_3=0}^{2n_2-n_3}\dots\sum_{l_{p+1}=0}^{2n_p-n_{p+1}}\sum_{k=0}^{2n_{p+1}} \notag\\
&\qquad t^{- \frac{n_1}{2} + \frac{l_1}{2} - k + \sum_{\gamma=1}^{p} (3n_\gamma - (l_{\gamma+1}+n_{\gamma+1}))}
\lt(t^{-N(l_1+n_1) }\frac{(t)_{N-l_1-1}(t)_{l_1+n_1}}{(t)_{n_1}(t)_{N-l_1-n_1-1}(t)_{l_1}}\rt) \notag\\
&\qquad \prod_{\gamma=1}^{p} 
\lt(t^{(2n_\gamma)^2} \cdot t^{-(2n_{\gamma})(l_{\gamma+1}+n_{\gamma+1}) }\frac{(t)_{2n_{\gamma}-l_{\gamma+1}}(t)_{l_{\gamma+1}+n_{\gamma+1}}}{(t)_{n_{\gamma+1}}(t)_{2n_{\gamma}-l_{\gamma+1}-n_{\gamma+1}}(t)_{l_{\gamma+1}}}\rt)\notag\\
&\qquad\lt(t^{- 2n_{p+1}k} \frac{(t)_{2n_{p+1}+1+k}}{(t)_{n_{p+1}-k}}  \rt)
\end{align}

By Lemma~\ref{uppbd1} and \ref{uppbdW}, when $t=e^{\frac{2\pi i}{N+\frac{1}{2}}}$, we know that
\begin{align*}
\lim_{N \to \infty} \frac{2\pi}{N+\frac{1}{2}} \log| J_{N,} (W_{0,1,1,0}^{p+1}(4_1),t) | \leq v_3 || \SS^3 \backslash W^{p+1}_{0,1,1,0}(4_1) ||
\end{align*}
Furthermore, by considering the term
$$ \frac{(t)_{N-l_1-1}(t)_{l_1+n_1}}{(t)_{n_1}(t)_{N-l_1-n_1-1}(t)_{l_1}}, $$
to compute the asymptotic expansion formula, it suffices to consider those $n_1$ and $l_1$ with
$$ \frac{n_1}{N+\frac{1}{2}} \sim \frac{1}{2} \quad \text{and} \quad \frac{l_1}{N+\frac{1}{2}} \sim \frac{1}{4}$$
In this case, we have $2n_1 \sim N$. By considering the term
$$ \frac{(t)_{2n_{1}-l_{2}}(t)_{l_{2}+n_{2}}}{(t)_{n_{2}}(t)_{2n_{1}-l_{2}-n_{2}}(t)_{l_{2}}} $$
to compute the asymptotic expansion formula, it suffices to consider those $n_2$ and $l_2$ with
$$ \frac{n_2}{N+\frac{1}{2}} \sim \frac{1}{2} \quad \text{and} \quad \frac{l_2}{N+\frac{1}{2}} \sim \frac{1}{4}$$
Inductively, for $\gamma=3,\dots, p+1$, it suffices to consider
$$ \frac{n_\gamma}{N+\frac{1}{2}} \sim \frac{1}{2} \quad \text{and} \quad \frac{l_\gamma}{N+\frac{1}{2}} \sim \frac{1}{4}$$
In particular, for $2n_{p+1} \sim N$, by considering the term 
$$\lt(t^{-k (2n_{p+1}+1)} \prod_{l=1}^{k}\left(1-t^{2n_{p+1}+1-l}\right)\left(1-t^{2n_{p+1}+1+l}\right)\rt),$$
it suffices to consider
$$ \frac{k}{N+\frac{1}{2}} \sim \frac{5}{6}$$
Next, define
\begin{align}
&\Phi_{N}\lt(W_{0,1,1,0}^{p+1}(4_1); z_1, z_2, \dots, z_{2p+2}, z_{2p+3}\rt) \notag\\
&= \Phi_{N} (WL;z_1,z_2) + \sum_{\gamma=1}^p \Psi_{N}(WL;2z_{2\gamma-1}, z_{2\gamma+1},z_{2\gamma+2}) +  \Psi(4_1; 2z_{2p+1},z_{2p+3})
\end{align}
where the functions $\Phi_{N}(WL,z_1,z_2)$ and $\Psi_N(4_1; 2z_{2p+1},z_{2p+3})$ are the functions defined before with formulas
\begin{align}
\Phi_{N} (WL;z_1,z_2) 
&= \frac{1}{2\pi i} 
\lt[\varphi_r \lt( \frac{N\pi}{N+\frac{1}{2}}  - \pi z_1 - \pi z_2 -\frac{\pi}{2N+1}\rt) \rt.  \notag\\
&\qquad  - \varphi_r\lt(  \frac{N\pi}{N+\frac{1}{2}} - \pi z_2- \frac{\pi}{2N+1}\rt) \notag \\
&\qquad  + \varphi_r \lt(\pi z_2 + \frac{\pi}{2N+1}\rt)  - \varphi_r\lt( \pi z_1 + \pi z_2 + \frac{\pi}{2N+1}\rt) \notag\\
&\lt.\qquad  + \varphi_r\lt( \pi z_1 + \frac{\pi}{2N+1}\rt) \rt] \\
\Psi_N\lt(4_1; 2z_{2p+1},z_{2p+3} \rt) 
&= \frac{1}{2\pi i} \lt[ \varphi_r \lt( - \pi z_{2p+3} + 2\pi z_{2p+1} + \frac{\pi}{2N+1} \rt) \rt. \notag\\
&\qquad \lt.- \varphi_r\lt( \pi z_{2p+3} - \pi + 2\pi z_{2p+1} + \frac{3\pi}{2N+1} \rt) \rt]- 4\pi i z_{2p+1} z_{2p+3}
\end{align}
and the function $\Psi_{N}(WL;2z_{2\gamma-1}, z_{2\gamma+1},z_{2\gamma+2})$ is defined by
\begin{align}
\Psi_{N}(WL;2z_{2\gamma-1}, z_{2\gamma+1},z_{2\gamma+2})
&= 8\pi i z_{2\gamma-1}^2 + \frac{1}{2\pi i} \lt\{ -(2\pi i)\lt(2z_{2\gamma-1} - 1\rt)(2\pi i)\lt(z_{2\gamma+1}+z_{2\gamma+2} \rt)\rt. \notag \\
&\qquad + \varphi_r \lt( 2z_{2\gamma-1}\pi  - \pi z_{2\gamma+1} - \pi z_{2\gamma+2} + \frac{\pi}{2N+1}\rt)   \notag\\
&\qquad  - \varphi_r\lt(  2z_{2\gamma-1}\pi - \pi z_{2\gamma+2} + \frac{\pi}{2N+1}\rt) \notag \\
&\qquad  + \varphi_r \lt(\pi z_{2\gamma+2} + \frac{\pi}{2N+1}\rt)  - \varphi_r\lt( \pi z_{2\gamma+1} + \pi z_{2\gamma+2} + \frac{\pi}{2N+1}\rt) \notag\\
&\lt.\qquad  + \varphi_r\lt( \pi z_{2\gamma+1} + \frac{\pi}{2N+1}\rt)  \rt\} 
\end{align}

Take $N\to \infty$, we have
\begin{align}
&\Phi\lt(W_{0,1,1,0}^{p+1}(4_1); z_1, z_2, \dots, z_{2p+2}, z_{2p+3}\rt)  \notag \\
&= \Phi (WL;z_1,z_2) + \sum_{\gamma=1}^p \Psi(WL;2z_{2\gamma-1}, z_{2\gamma+1},z_{2\gamma+2})  +  \Psi(4_1;z_{2p+1},z_{2p+3})
\end{align}
where the functions $\Phi(WL,z_1,z_2)$ and $\Phi(4_1; 2z_{2p+1},z_{2p+3})$ are the functions defined before with formulas
\begin{align}
\Phi(WL;z_1,z_2) 
&= \frac{1}{2\pi i} \lt[ \Li\lt(e^{- 2\pi i z_1 - 2\pi i z_2}\rt) - \Li\lt( e^{ - 2\pi i z_2 }\rt) + \Li \lt(e^{2\pi i z_2}\rt) \rt. \notag\\
&\qquad \lt. - \Li\lt(e^{2\pi iz_1+ 2\pi i z_2}\rt)
 + \Li\lt(e^{2\pi i z_1}\rt)  \rt] \\
\Psi\lt(4_1; 2z_{2p+1},z_{2p+3} \rt) 
&= \frac{1}{2\pi i}\lt[ \Li(e^{-2\pi i z_{2p+3} + 4 \pi i z_{2p+1}}) - \Li(e^{2\pi i z_{2p+3} + 4 \pi i z_{2p+1}}) \rt] \notag\\
&\qquad - 4\pi i z_{2p+1} z_{2p+3}
\end{align}
and the function $\Psi (WL; 2z_{2\gamma - 1},z_{2\gamma + 1},z_{2\gamma + 2})$ is given by
\begin{align}
\Psi (WL; 2z_{2\gamma - 1},z_{2\gamma + 1},z_{2\gamma + 2})
&= 8\pi i z_{2\gamma-1}^2 + \frac{1}{2\pi i} \lt[  - (2\pi i (2z_{2\gamma - 1}-1))(2\pi i z_{2\gamma + 1} + 2\pi i z_{2\gamma + 2})  \rt. \notag\\
&\qquad + \Li\lt(e^{2\pi i (2z_{2\gamma - 1}-1) - 2\pi i z_{2\gamma + 1} - 2\pi i z_{2\gamma + 2}}\rt) - \Li\lt( e^{2\pi i (2z_{2\gamma - 1}-1) - 2\pi i z_{2\gamma + 2} }\rt) \notag\\
&\qquad \lt.+ \Li \lt(e^{2\pi i z_{2\gamma + 2}}\rt) - \Li\lt(e^{2\pi iz_{2\gamma + 1}+ 2\pi i z_{2\gamma + 2}}\rt) + \Li\lt(e^{2\pi i z_{2\gamma + 1}}\rt)  \rt]\notag\\
\end{align}
Here the function $\Phi\lt(W_{0,1,1,0}^{p+1}(4_1); z_1, z_2, \dots, z_{2p+2}, z_{2p+3}\rt) $ is defined on the domain
$$ D = \lt\{(z_1,z_2,\dots,z_{2p+1},z_{2p+2},z_{2p+3}) \mid |z_{2\gamma-1}-\frac{1}{2}|, |z_{2\gamma} - \frac{1}{4}|, |z_{2p+3} - \frac{5}{6}| < \eta  \text{,where } \gamma=1,\dots, p+1\rt\}$$
for some sufficiently small $\eta>0$. 
Also, note that the function $\Psi(WL;2z_{2\gamma-1}, z_{2\gamma+1},z_{2\gamma+2}) - 8\pi i z_{2\gamma-1}^2$ is the potential function $\Phi^{(1,s_2)}(z_1,z_2)$ in \cite{W19} with $s_2 = 2z_{2\gamma-1}$, $z_1 =  z_{2\gamma+1}$ and $z_2 = z_{2\gamma+2}$.

Let $D_\eta$ to be the collection of all $(n_1,\dots, n_{p+1}, l_1, \dots, l_{p+1}, k)$ such that
$$ \lt( \frac{n_1}{N+\frac{1}{2}}, \dots, \frac{n_{p+1}}{N+\frac{1}{2}} , \frac{l_1}{N+\frac{1}{2}},\dots,\frac{l_{p+1}}{N+\frac{1}{2}}, \frac{k}{N+\frac{1}{2}} \rt) \in D $$ 

By direct computation and the maximum point estimation, we have
\begin{align}
&J_{N, N} (W_{0,1,1,0}^{p+1}(4_1),e^{\frac{2\pi i}{N+\frac{1}{2}}}) \notag\\
&\stackrel[N \to \infty]{\sim}{} - \frac{N+\frac{1}{2}}{\pi} 
\exp\lt(  -\varphi_r \lt(  \frac{\pi}{2N+1}\rt)\rt)^{p+1}
\sum_{D_\delta }e^{\frac{2\pi i}{N+\frac{1}{2}}\lt( -\frac{n_1}{2}+\frac{n_2}{2}- k +\sum_{\gamma=1}^p (3n_\gamma-(l_{\gamma+1}+n_{\gamma+1})) \rt)}
\notag \\
&\qquad \exp\lt( \lt(N+\frac{1}{2}\rt)\Phi_{N}\lt(W_{0,1,1,0}^{p+1}(4_1);\frac{n_1}{N+\frac{1}{2}}, \dots, \frac{n_{p+1}}{N+\frac{1}{2}}, \frac{l_1}{N+\frac{1}{2}}\dots, \frac{l_{p+1}}{N+\frac{1}{2}}, \frac{k}{N+\frac{1}{2}}\rt) \rt)
\end{align}

Besides, by considering the Talyor series expansion, we have
\begin{align*}
&\exp\lt( \lt(N+\frac{1}{2}\rt) \Phi_{N,N}\lt(W_{0,1,1,0}^{p+1}(4_1); z_1, z_2, \dots, z_{2p+3}\rt) \rt)\\
&= E(W_{0,1,1,0}^{p+1}(4_1); z_1, z_2, \dots, z_{2p+3}) \\
&\qquad \exp\lt( \lt(N+\frac{1}{2}\rt) \Phi\lt(W_{0,1,1,0}^{p+1}(4_1); z_1, z_2, \dots, z_{2p+3}\rt) \rt)\lt(1+O\lt(\frac{1}{N+\frac{1}{2}}\rt)\rt),
\end{align*}
where
\begin{align*}
&E(W_{0,1,1,0}^{p+1}(4_1); z_1, z_2, \dots, z_{2p+3}) \\
&= \exp\lt(\log(1-e^{-2\pi i z_1 - 2\pi i z_2}) - \log(1-e^{-2\pi i z_2}) - \frac{1}{2}\log(1-e^{2\pi i z_2}) \rt.\\
&\qquad\qquad + \frac{1}{2}\log(1-e^{2\pi i z_1 + 2\pi i z_2}) - \frac{1}{2}\log(1-e^{2\pi i z_1}) \\
&\qquad + \sum_{\gamma=1}^p \lt(
-\frac{1}{2}\log(1-e^{4\pi iz_{2\gamma - 1}-2\pi i z_{2\gamma+1} - 2\pi i z_{2\gamma+2}})
+ \frac{1}{2}\log(1-e^{4\pi i z_{2\gamma -1} - 2\pi i z_{2\gamma+2}})
\rt. \\
&\qquad\qquad\quad \lt. - \frac{1}{2}\log(1-e^{2\pi i z_{2\gamma+2}}) + \frac{1}{2}\log(1-e^{2\pi i z_{2\gamma+1} + 2\pi i z_{2\gamma+2}}) - \frac{1}{2}\log(1-e^{2\pi i z_{2\gamma+1}}) \rt) \\
&\qquad\qquad \lt. - \frac{3}{2}\log(1-e^{-2\pi i z_{2p+3}+4\pi i z_{2p+1}}) 
- \frac{1}{2}\log(1-e^{2\pi i z_{2p+3}+4\pi i z_{2p+1}}) \rt)
\end{align*}

Note that the critical point equations of the potential function
$$\Phi\lt(W_{0,1,1,0}^{p+1}(4_1); z_1, z_2, \dots, z_{2p+2}, z_{2p+3}\rt)$$
are given by
\begin{empheq}[left = \empheqlbrace]{align}
0 &= \frac{\partial\Phi(WL;z_1,z_2)}{\partial z_1} +  \frac{\partial\Psi(WL;2z_1,z_3,z_4)}{\partial z_1} \\
0 &= \frac{\partial\Phi(WL;z_1,z_2)}{\partial z_2} \\
0 &= \frac{\partial\Psi(WL;2z_{2\gamma-3},z_{2\gamma-1},z_{2\gamma})}{\partial z_{2\gamma -1}} + \frac{\partial\Psi(WL;2z_{2\gamma-1},z_{2\gamma+1},z_{2\gamma+2})}{\partial z_{2\gamma -1}}  \qquad \text{for $\gamma=2,\dots, p$} \\
0 &= \frac{\partial\Psi(WL;2z_{2\gamma-1}, z_{2\gamma+1},z_{2\gamma + 2})}{\partial z_{2\gamma+2}}  
\qquad\qquad\qquad\qquad\qquad\qquad\qquad\quad
 \text{for $\gamma=1,\dots, p$} \\
0 &= \frac{\partial\Psi(WL;2z_{2p-1},z_{2p+1},z_{2p+2})}{\partial z_{2p+1}} 
+ \frac{\partial\Psi(4_1;2z_{2p+1},z_{2p+3})}{\partial z_{2p+1}} \\
0 &= \frac{\partial\Psi(4_1;2z_{2p+1},z_{2p+3})}{\partial z_{2p+3}} 
\end{empheq}
Note that when $z_1=z_3=\frac{1}{2}, z_2=z_4=\frac{1}{4}$, we have
\begin{align*}
\frac{\partial\Phi(WL;z_1,z_2)}{\partial z_1} +  \frac{\partial\Psi(WL;2z_1,z_3,z_4)}{\partial z_1}
&= 16\pi i \lt(\frac{1}{2}\rt)- 2(2\pi i)\lt(\frac{1}{2}+\frac{1}{4}\rt) \notag\\
&\qquad -2\log(1-e^{-2\pi i(3/4)}) + 2\log(1-e^{-2\pi i(1/4)}) \\
&= 6\pi i \\
\frac{\partial\Phi(WL;z_1,z_2)}{\partial z_2} 
&= 0
\end{align*}
Similarly, when $z_{2\gamma-3}=z_{2\gamma-1}=z_{2\gamma+1}=\frac{1}{2}, z_{2\gamma}=z_{2\gamma+2}=\frac{1}{4}$, we have
\begin{alignat*}{2}
\frac{\partial\Psi(WL;2z_{2\gamma-3},z_{2\gamma-1},z_{2\gamma})}{\partial z_{2\gamma -1}} + \frac{\partial\Psi(WL;2z_{2\gamma-1},z_{2\gamma+1},z_{2\gamma+2})}{\partial z_{2\gamma -1}} 
&= 6\pi i &&\qquad\qquad\qquad \text{for $\gamma=2,\dots, p$} \\
\frac{\partial\Psi(WL;2z_{2\gamma-1}, z_{2\gamma+1},z_{2\gamma + 2})}{\partial z_{2\gamma+2}} &= 0  &&\qquad\qquad\qquad \text{for $\gamma=1,\dots, p$}
\end{alignat*}
Finally, when $z_{2p-1}=z_{2p+1}=\frac{1}{2}$, $z_{2p+2}=\frac{1}{4}$, $z_{2p+3}=\frac{5}{6}$, we have
\begin{align*}
\frac{\partial\Psi(WL;2z_{2p-1},z_{2p+1},z_{2p+2})}{\partial z_{2p+1}} 
+ \frac{\partial\Psi(4_1;2z_{2p+1},z_{2p+3})}{\partial z_{2p+1}}
&= \frac{\partial\Psi\lt(4_1;1,\frac{5}{6}\rt)}{\partial z_{2p+1}} 
= -2\pi i\\
\frac{\partial\Psi(4_1;2z_{2p+1},z_{2p+3})}{\partial z_{2p+3}} 
&= 0
\end{align*}
As a result,the point
$$ \lt(\frac{1}{2},\frac{1}{4},\dots, \frac{1}{2},\frac{1}{4},\frac{5}{6} \rt) $$
is a critical point of the Fourier coefficient
$$  \Phi\lt(W_{0,1,1,0}^{p+1}(4_1); z_1, z_2, \dots, z_{2p+2}, z_{2p+3}\rt) - \lt(\sum_{\gamma=1}^{p} 6\pi i z_{2\gamma-1}\rt) + 2\pi i z_{2p+1}$$
Moreover, the critical value is given by
\begin{align*}
&\Phi\lt(W_{0,1,1,0}^{p+1}(4_1); \frac{1}{2},\frac{1}{4},\dots,\frac{1}{2},\frac{1}{4}, \frac{5}{6}\rt) - \lt(\sum_{\gamma=1}^p 3\pi i \rt) + \pi i \\
&= (p+1)\Phi\lt(WL; \frac{1}{2},\frac{1}{4} \rt) + \Phi\lt(4_1; \frac{1}{2},\frac{5}{6} \rt) 
-(3p-1)\pi i\\
&= \frac{1}{2\pi}\lt[(p+1)\Vol(\SS^3\backslash WL) + \Vol(\SS^3\backslash 4_1) + i \frac{(p+1)\pi^2}{4} \rt] - (3p-1)\pi i\\
&= \frac{1}{2\pi}\lt[ v_3 || \SS^3 \backslash W_{0,1,1,0}^{p+1}(4_1) || + i \frac{(p+1)\pi^2}{4}\rt] - (3p-1)\pi i
\end{align*}

The following lemma shows that the critical point is indeed non-degenerate.
\begin{lemma}\label{nonsing1} We have $\det\Hess\lt(\Phi_{N,N}\lt(W_{0,1,1,0}^{p+1}(4_1); \frac{1}{2},\frac{1}{4},\dots,\frac{1}{2},\frac{1}{4},\frac{5}{6} \rt)\rt) \neq 0$.
\end{lemma}
\begin{proof}
By direct computation, we have
\begin{align}
\Hess\lt(\Phi\lt(WL; \frac{1}{2}, \frac{1}{4}\rt)\rt)
&= 2\pi i
\begin{pmatrix}
-\frac{1}{2} + i & i \\
i & 2i
\end{pmatrix}\\
\Hess\lt(\Psi\lt(WL; 1, \frac{1}{2}, \frac{1}{4}\rt)\rt)
&= 2\pi i
\begin{pmatrix}
8+2i & -1-i & -2-i \\
-1 -i & -\frac{1}{2} + i & i \\
-2-i & i & 2i
\end{pmatrix} \\
\Hess\lt(\Phi\lt(4_1; \frac{1}{2}, \frac{1}{4}\rt)\rt)
&= 2\pi i
\begin{pmatrix}
4\sqrt{3}i & 2 \\
2 & \sqrt{3}i
\end{pmatrix}
\end{align}
Next, we consider the real parts of the above matrices, denoted by $A_1,A_2$ and $A_3$ respectively:
\begin{align}
A_1
&=\Re\Hess\lt(\Phi\lt(WL; \frac{1}{2}, \frac{1}{4}\rt)\rt)
= -2\pi 
\begin{pmatrix}
1& 1 \\
1 & 2
\end{pmatrix}\\
A_2
&=\Re\Hess\lt(\Psi\lt(WL; 1, \frac{1}{2}, \frac{1}{4}\rt)\rt)
= -2\pi 
\begin{pmatrix}
-2 & -1 & -1 \\
-1 & 1 & 1 \\
-i & 1 & 2
\end{pmatrix} \\
A_3
&= \Re\Hess\lt(\Phi\lt(4_1; \frac{1}{2}, \frac{1}{4}\rt)\rt)
= -2\pi 
\begin{pmatrix}
4\sqrt{3} & 0 \\
0 & \sqrt{3}
\end{pmatrix}
\end{align}
One can verify that the three matrices above are negative definite (for example, by the Sylvester's criterion). We are going to show that the matrix
$$ A = \Re\Hess\lt(\Phi_{N,N}\lt(W_{0,1,1,0}^{p+1}(4_1); \frac{1}{2},\frac{1}{4},\dots,\frac{1}{2},\frac{1}{4},\frac{5}{6} \rt)\rt)$$
is also negative definite. To see this, note that for $\mathbf{z}=(z_1,\dots,z_{2p+3})\in\RR^{2p+3}$,
\begin{align}
\mathbf{z} A \mathbf{z}^T
&= (z_1, z_2) A_1 (z_1, z_2)^T 
+ \sum_{\gamma=1}^p (z_{2\gamma-1}, z_{2\gamma+1}, z_{2\gamma+2})A_2(z_{2\gamma-1}, z_{2\gamma+1}, z_{2\gamma+2})^T \notag\\
&\qquad + (z_{2p+1}, z_{2p+3})A_3(z_{2p+1}, z_{2p+3})^T \notag\\
&\leq 0
\end{align}
with equality holds if and only if $\mathbf{z}=0$. As a result, since the real part of the matrix 
$$\Hess\lt(\Phi_{N,N}\lt(W_{0,1,1,0}^{p+1}(4_1); \frac{1}{2},\frac{1}{4},\dots,\frac{1}{2},\frac{1}{4},\frac{5}{6} \rt)\rt)$$
is negative definite, by Lemma in \cite{L81}, it is non-singular.
\end{proof}

As a result, by using the same argument as before, the asymptotic expansion formula for \linebreak $J_{N} (W^{p+1}_{0,1,1,0}(4_1),e^{\frac{2\pi i}{N+\frac{1}{2}}})$ is given by
\begin{align}
&J_{N, N} (W_{0,1,1,0}^{p+1}(4_1),e^{\frac{2\pi i}{N+\frac{1}{2}}}) \notag\\
&\stackrel[N \to \infty]{\sim}{} - \frac{N+\frac{1}{2}}{\pi} 
\exp\lt(  -\varphi_r \lt(  \frac{\pi}{2N+1}\rt)\rt)^{p+1}
\sum_{D_\delta }e^{\frac{2\pi i}{N+\frac{1}{2}}\lt( -\frac{n_1}{2}+\frac{l_1}{2}- k +\sum_{\gamma=1}^p (3n_\gamma-(l_{\gamma+1}+n_{\gamma+1})) \rt)}
\notag \\
&\qquad \exp\lt( \lt(N+\frac{1}{2}\rt)\Phi_{N}\lt(W_{0,1,1,0}^{p+1}(4_1);\frac{n_1}{N+\frac{1}{2}}, \dots, \frac{n_{p+1}}{N+\frac{1}{2}}, \frac{l_1}{N+\frac{1}{2}}\dots, \frac{l_{p+1}}{N+\frac{1}{2}}, \frac{k}{N+\frac{1}{2}}\rt) \rt) \notag\\
&\stackrel[N \to \infty]{\sim}{} - \frac{N+\frac{1}{2}}{\pi} 
\exp\lt(  -\varphi_r \lt(  \frac{\pi}{2N+1}\rt)\rt)^{p+1}
\times e^{\pi i \lt(\frac{-23}{12}+\frac{3p}{2}\rt)} \lt(N+\frac{1}{2}\rt)^{\frac{2p+3}{2}}
\notag \\
&\qquad \frac{\pi^{(2p+3)/2}E(W_{0,1,1,0}^{p+1}(4_1); \frac{1}{2},\frac{1}{4},\dots,\frac{1}{2},\frac{1}{4},\frac{5}{6})}{\sqrt{-\det\Hess\lt(\Phi\lt(W_{0,1,1,0}^{p+1}(4_1); \frac{1}{2},\frac{1}{4},\dots,\frac{1}{2},\frac{1}{4},\frac{5}{6} \rt)\rt)} }\notag \\
&\qquad \times \exp\lt( \frac{N+\frac{1}{2}}{2\pi} 
\lt(v_3 || \SS^3 \backslash W_{0,1,1,0}^{p+1}(4_1) || + i \frac{(p+1)\pi^2}{4}\rt) - (3p-1)\pi i\rt)
\end{align}

\subsection{Generalized volume conjecture for $W_{0,1,1,0}^{p+1}(4_1)$}\label{GVCcp}

Next, we apply the `continuity argument' described in Section 2.3 of \cite{W19} to study the generalized volume conjecture for $W_{0,1,1,0}^{p+1}(4_1)$. Let $\ds s =\lim_{N\to \infty} \frac{M}{N+\frac{1}{2}} $. From direct computation, the potential function for colored Jones polynomials $J_{M}(W_{0,1,1,0}^{p+1}(4_1), e^{\frac{2\pi i}{N+\frac{1}{2}}})$ is given by
\begin{align}
&\Phi^{(s)}\lt(W_{0,1,1,0}^{p+1}(4_1); z_1, z_2, \dots, z_{2p+2}, z_{2p+3}\rt) \notag\\
&= \Phi^{(s)} (WL;z_1,z_2) + \sum_{\gamma=1}^p \Psi(WL;2z_{2\gamma-1}, z_{2\gamma+1},z_{2\gamma+2})  +  \Psi(4_1;z_{2p+1},z_{2p+3}) 
\end{align} 
where
\begin{align}
\Phi^{(s)} (WL;z_1,z_2) \notag
&=  \frac{1}{2\pi i} \lt[  - (2\pi i (s_2-1))(2\pi i z_1 + 2\pi i z_2)  \rt. \notag\\
&\qquad + \Li\lt(e^{2\pi i (s_2-1) - 2\pi i z_1 - 2\pi i z_2}\rt) - \Li\lt( e^{2\pi i (s_2-1) - 2\pi i z_2 }\rt) \notag\\
&\qquad \lt.+ \Li \lt(e^{2\pi i z_2}\rt) - \Li\lt(e^{2\pi iz_1+ 2\pi i z_2}\rt) + \Li\lt(e^{2\pi i z_1}\rt)  \rt] \\
\Phi (WL; 2z_{2\gamma - 1},z_{2\gamma + 1},z_{2\gamma + 2})
&=  8\pi i z_{2\gamma-1}^2  + \frac{1}{2\pi i} \lt[  - (2\pi i (2z_{2\gamma - 1}-1))(2\pi i z_{2\gamma + 1} + 2\pi i z_{2\gamma + 2})  \rt. \notag\\
&\qquad + \Li\lt(e^{2\pi i (2z_{2\gamma - 1}-1) - 2\pi i z_{2\gamma + 1} - 2\pi i z_{2\gamma + 2}}\rt) - \Li\lt( e^{2\pi i (2z_{2\gamma - 1}-1) - 2\pi i z_{2\gamma + 2} }\rt) \notag\\
&\qquad \lt.+ \Li \lt(e^{2\pi i z_{2\gamma + 2}}\rt) - \Li\lt(e^{2\pi iz_{2\gamma + 1}+ 2\pi i z_{2\gamma + 2}}\rt) + \Li\lt(e^{2\pi i z_{2\gamma + 1}}\rt)  \rt] \\
\Psi\lt(4_1; z_{2p+1},z_{2p+3} \rt) 
&= \frac{1}{2\pi i}\lt[ \Li(e^{-2\pi i z_{2p+3} + 4 \pi i z_{2p+1}}) - \Li(e^{2\pi i z_{2p+3} + 4 \pi i z_{2p+1}}) \rt] \notag\\
&\qquad - 2\pi i (1-2z_{2p+1}) z_{2p+3}
\end{align}

Note that the critical point equations of the potential function
$$\Phi^{(s)}\lt(W_{0,1,1,0}^{p+1}(4_1); z_1, z_2, \dots, z_{2p+2}, z_{2p+3}\rt)$$
are given by
\begin{empheq}[left = \empheqlbrace]{align}
0 &= \frac{\partial\Phi^{(s)}(WL;z_1,z_2)}{\partial z_1} +  \frac{\partial\Psi(WL;2z_1,z_3,z_4)}{\partial z_1} \\
0 &= \frac{\partial\Phi(WL;z_1,z_2)}{\partial z_2} \\
0 &= \frac{\partial\Psi(WL;2z_{2\gamma-3},z_{2\gamma-1},z_{2\gamma})}{\partial z_{2\gamma -1}} + \frac{\partial\Psi(WL;2z_{2\gamma-1},z_{2\gamma+1},z_{2\gamma+2})}{\partial z_{2\gamma -1}}  \qquad \text{for $\gamma=2,\dots, p$} \\
0 &= \frac{\partial\Psi(WL;2z_{2\gamma-1}, z_{2\gamma+1},z_{2\gamma + 2})}{\partial z_{2\gamma+2}}  
\qquad\qquad\qquad\qquad\qquad\qquad\qquad\quad
 \text{for $\gamma=1,\dots, p$} \\
0 &= \frac{\partial\Psi(WL;2z_{2p-1},z_{2p+1},z_{2p+2})}{\partial z_{2p+1}} 
+ \frac{\partial\Psi(4_1;2z_{2p+1},z_{2p+3})}{\partial z_{2p+1}} \\
0 &= \frac{\partial\Psi(4_1;2z_{2p+1},z_{2p+3})}{\partial z_{2p+3}} 
\end{empheq}

To solve this system of equations, we first study the equation
$$ \frac{\partial\Phi^{(s)}(WL;z_1,z_2)}{\partial z_1} = 0 $$
Note that the function $\Phi^{(s)}(WL;z_1,z_2)$ is the function $\Phi^{(1,s)}(z_1,z_2)$ discussed in Section 2.4 of \cite{W19}. Let $Z_1=e^{2\pi i z_1}, Z_2 = e^{2\pi i z_2}$ and $B=e^{2\pi i s}$. Recall from Section 2.4 in \cite{W19} that for the potential function $\Phi^{(1,s)}(z_1,z_2)$, when we put $z_1=\frac{1}{2}$, its critical point equations become
\begin{empheq}[left = \empheqlbrace]{align}
\frac{\lt(1 + \dfrac{B}{Z_2} \rt)\lt( 1 + Z_2\rt)}{2} &=   B \label{smeridian2} \\
\frac{\lt(1 + \dfrac{B}{Z_2} \rt)\lt( 1 + Z_2\rt)}{(1-\dfrac{B}{Z_2})(1- Z_2)} &= B \label{slong2}
\end{empheq}

Next, consider the equation 
\begin{align}
\lt(1-\dfrac{B}{Z_2}\rt)\lt(1- Z_2\rt) = 2 \label{qlemma}
\end{align}
which is equivalent to 
\begin{align}
Z_2^2 + (1 - B)Z_2 + B =0 \label{quadr}
\end{align}

Let $Z_2(s) = \frac{-(1-B) +\sqrt{(1-B)^2 - 4B}}{2}$. Note that $Z_2(1) = i$. Furthermore, from (\ref{qlemma}),
\begin{align}
\lt( \dfrac{B}{Z_2} + Z_2 \rt) = -1 + B,
\end{align}
which implies 
\begin{align}\label{qlemma2}
\lt(1 + \dfrac{B}{Z_2} \rt)\lt( 1 + Z_2\rt) 
=\lt(1 - \dfrac{B}{Z_2} \rt)\lt( 1 - Z_2\rt) + 2\lt( \dfrac{B}{Z_2} + Z_2 \rt)
= 2B
\end{align}

As a result, from (\ref{qlemma}) and (\ref{qlemma2}), $(Z_1(s),Z_2(s))=(-1,Z_2(s))$ is a holomorphic family of solution for (\ref{smeridian2}) and (\ref{slong2}) with $(Z_1(1),Z_2(1))=(-1, i)= (e^{2\pi i (\frac{1}{2})}, e^{2\pi i (\frac{1}{4})})$. Let $z_2(s)$ be the holomorhpic function such that $Z_2(s)=e^{2\pi i z_2(s)}$ with $z_2(1)=\frac{1}{4}$.

By straightforward computation, the point 
$$  (z_1, z_2, \dots, z_{2p+2}, z_{2p+3}) = \lt(\frac{1}{2}, z_2(s), \frac{1}{2}, \frac{1}{4},\dots, \frac{1}{2}, \frac{1}{4}, \frac{5}{6}  \rt)  $$
is a critical point for the Fourier coefficient 
$$\Phi^{(s)}\lt(W_{0,1,1,0}^{p+1}(4_1); z_1, z_2, \dots, z_{2p+2}, z_{2p+3}\rt)- \lt(\sum_{\gamma=1}^{p} 6\pi i z_{2\gamma-1}\rt) + 2\pi i z_{2p+1}$$
Furthermore, by Theorem 3 in \cite{W19}, we have
\begin{align*}
&\Re\Phi^{(s)}\lt(W_{0,1,1,0}^{p+1}(4_1), \frac{1}{2}, z_2(s), \frac{1}{2}, \frac{1}{4},\dots, \frac{1}{2}, \frac{1}{4}, \frac{5}{6}  \rt) \\
&= \frac{1}{2\pi} \lt[\Vol(\SS^3 \backslash WL; u_1 = 0, u_2=2\pi i(1-s)) + p\Vol(\SS^3 \backslash WL) + \Vol(\SS^3 \backslash 4_1) \rt]
\end{align*}
By the `continuity argument', Theorem~\ref{mainthm3} follow from the same arguments as before. In particular, it implies Theorem~\ref{mainthm5} for $W^{p+1}_{0,1,1,0}(4_1)$. Corollary~\ref{cablingtv2} for $\SS^3\backslash W_{0,1,1,0}^{p+1}(4_1)$ then follows directly from the method described in Section~\ref{strategy}.

\section{Iterated Whitehead double on Hopf link}\label{Walphabeta}
In this Section, we study the asymptotics of the $(M_1,M_2)$-th colored Jones polynomials of the link $W^\alpha_\beta$ at $(N+\frac{1}{2})$-th root of unity.

First of all, by using the method described in Section 2 of \cite{HZ07}, the formula of the colored Jones polynomials of the link $W^\alpha_\beta$ is given by

\begin{align}
&J_{M_1,M_2}(W^\alpha_\beta,t) \notag\\
&= 
\sum_{n_1=0}^{2M_1-1}\sum_{n_2=0}^{2n_1}\dots\sum_{n_{\alpha-1}=0}^{2n_{\alpha-2}}
\sum_{n_1'=0}^{2M_1-1}\sum_{n_2'=0}^{2n_1'}\dots\sum_{n_{\beta}'=0}^{2n_{\beta-1}'} C(n_1,t;M_1) \prod_{\gamma=2}^{\alpha-1} C(n_\gamma, t; 2n_{\gamma-1}+1) \notag\\
&\qquad \cdot
C(n_1',t;M_2) \prod_{\gamma=2}^{\beta} C(n_\gamma', t; 2n_{\gamma-1}+1)
\cdot J_{2n_{\beta}'+1,2{n_{\alpha-1}+1}}(WL,t) \notag\\
&=
\sum_{n_1=0}^{2M_1-1}\sum_{n_2=0}^{2n_1}\dots\sum_{n_{\alpha-1}=0}^{2n_{\alpha-2}}
\sum_{n_1'=0}^{2M_1-1}\sum_{n_2'=0}^{2n_1'}\dots\sum_{n_{\beta}'=0}^{2n_{\beta-1}'} 
\sum_{n_{\zeta}=0}^{2n_{\alpha-1}} \notag\\
&\qquad \cdot
C(n_1,t;M_1) \prod_{\gamma=2}^{\alpha-1} C(n_\gamma, t; 2n_{\gamma-1}+1)
C(n_1',t;M_2) \prod_{\gamma=2}^{\beta} C(n_\gamma', t; 2n_{\gamma-1}'+1)  \notag\\
&\qquad
\cdot \frac{t^{(2n_{\beta}+1)(2n_{\zeta}+1)/2}-t^{-(2n_{\beta}+1)(2n_\zeta+1)/2}}{t^{\frac{1}{2}}-t^{-\frac{1}{2}}} \cdot C(n_\zeta, t; 2n_{\alpha-1}+1)
\end{align}
where
$$C(n,t; N) = \sum_{l = 0}^{N-1-n} t^{-N(l+n)} \prod_{j=1}^{n} \frac{(1-t^{N-l-j})(1-t^{l+j})}{1-t^j}$$

Thus,
\begin{align}
&J_{M_1,M_2}(W^\alpha_\beta,t) = \notag\\
&\sum_{n_1=0}^{2M_1-1}\sum_{n_2=0}^{2n_1}\dots\sum_{n_{\alpha-1}=0}^{2n_{\alpha-2}}
\sum_{n_1'=0}^{2M_1-1}\sum_{n_2'=0}^{2n_1'}\dots\sum_{n_{\beta}'=0}^{2n_{\beta-1}'} 
\sum_{n_{\zeta}=0}^{2n_{\alpha-1}} \notag\\
& \sum_{l_1=0}^{M_1-1-n_1}\sum_{l_2=0}^{2n_1-n_2}\sum_{l_3=0}^{2n_2-n_3}\dots\sum_{l_{\alpha-1}=0}^{2n_{\alpha-2}-n_{\alpha-1}} 
\sum_{l_1'=0}^{M_2-1-n_1}\sum_{l_2'=0}^{2n_1'-n_2'}\sum_{l_3'=0}^{2n_2'-n_3'}\dots \sum_{l_{\beta}'=0}^{2n_{\beta-1}'-n_{\beta}'} 
\sum_{l_\zeta=0}^{2n_{\alpha-1} - n_\zeta}\notag\\
&\quad \cdot t^{\lt(M_1^2 - \frac{M_1}{2} - \frac{1}{2}\rt)+\sum_{\gamma=2}^{\alpha-1}(3n_{\gamma-1} - (l_{\gamma}+n_{\gamma}))
+\lt(M_2^2 - \frac{M_2}{2} - \frac{1}{2}\rt)+\sum_{\gamma=2}^{\beta}(3n'_{\gamma-1} - (l'_{\gamma}+n'_{\gamma})) - l_\zeta - n_\zeta} \notag\\
&\quad 
\lt( t^{-M_1(l_1+n_1)} \frac{(t)_{M_1-l_1-1}(t)_{l_1+n_1}}{(t)_{n_1}(t)_{M_1-l_1-n_1-1}(t)_{l_1}}\rt)
\prod_{\gamma=2}^{\alpha-1}
\lt(t^{(2n_{\gamma-1})^2} \cdot t^{-(2n_{\gamma-1})(l_{\gamma}+n_{\gamma}) }\frac{(t)_{2n_{\gamma-1}-l_{\gamma}}(t)_{l_{\gamma}+n_{\gamma}}}{(t)_{n_{\gamma}}(t)_{2n_{\gamma-1}-l_{\gamma}-n_{\gamma}}(t)_{l_{\gamma}}}\rt) 
\notag\\
&\quad 
\lt( t^{-M_2(l'_1+n'_1)} \frac{(t)_{M_2-l'_1-1}(t)_{l'_1+n'_1}}{(t)_{n'_1}(t)_{M_2-l'_1-n'_1-1}(t)_{l'_1}}\rt)
\prod_{\gamma=2}^{\beta}
\lt(t^{(2n'_{\gamma-1})^2} \cdot t^{-(2n'_{\gamma-1})(l'_{\gamma}+n'_{\gamma}) }\frac{(t)_{2n'_{\gamma-1}-l'_{\gamma}}(t)_{l'_{\gamma}+n'_{\gamma}}}{(t)_{n'_{\gamma}}(t)_{2n'_{\gamma-1}-l'_{\gamma}-n'_{\gamma}}(t)_{l'_{\gamma}}}\rt) \notag\\
&\quad \cdot \frac{t^{(2n_{\beta}+1)(2n_{\zeta}+1)/2}-t^{-(2n_{\beta}+1)(2n_\zeta+1)/2}}{t^{\frac{1}{2}}-t^{-\frac{1}{2}}}
\lt(t^{(2n_{\alpha-1})^2} \cdot t^{-(2n_{\alpha-1})(l_{\zeta}+n_{\zeta}) }\frac{(t)_{2n_{\alpha-1}-l_{\zeta}}(t)_{l_{\zeta}+n_{\zeta}}}{(t)_{n_{\zeta}}(t)_{2n_{\alpha-1}-l_{\zeta}-n_{\zeta}}(t)_{l_{\zeta}}}\rt)
\end{align}

By Lemma~\ref{uppbdW}, when $t=e^{\frac{2\pi i}{N+\frac{1}{2}}}$, we know that
\begin{align*}
\limsup_{N \to \infty} \frac{2\pi}{N+\frac{1}{2}} \log| J_{N, N} (W^\alpha_\beta, t) | \leq v_3 || \SS^3 \backslash W^\alpha_\beta|| = (\alpha+\beta) \Vol(\SS^3\backslash WL)
\end{align*}
Furthermore, we consider the norm of the term
$$ \frac{(t)_{M_1-l_1-1}(t)_{l_1+n_1}}{(t)_{n_1}(t)_{M_1-l_1-n_1-1}(t)_{l_1}}, $$
For $\ds s_1 = \lim_{N\to \infty}\frac{M_1}{N+\frac{1}{2}}$ sufficiently close to 1, in order to compute the asymptotic expansion formula, it suffices to consider those $n_1$ and $l_1$ with
$$ \frac{n_1}{N+\frac{1}{2}} \sim \frac{1}{2} \quad \text{and} \quad \frac{l_1}{N+\frac{1}{2}} \sim \frac{1}{4}$$
In this case, we have $2n_1 + 1 \sim 1$. By considering the term
$$ \frac{(t)_{2n_{1}-l_{2}}(t)_{l_{2}+n_{2}}}{(t)_{n_{2}}(t)_{2n_{1}-l_{2}-n_{2}}(t)_{l_{2}}} $$
to compute the asymptotic expansion formula, it suffices to consider those $n_2$ and $l_2$ with
$$ \frac{n_2}{N+\frac{1}{2}} \sim \frac{1}{2} \quad \text{and} \quad \frac{l_2}{N+\frac{1}{2}} \sim \frac{1}{4}$$
Inductively, for $\gamma=3,\dots, \alpha-1$, it suffices to consider
$$ \frac{n_\gamma}{N+\frac{1}{2}} \sim \frac{1}{2} \quad \text{and} \quad \frac{l_\gamma}{N+\frac{1}{2}} \sim \frac{1}{4}$$

Similarly, for $\gamma = 1,2,\dots, \beta$ and $\ds s_2 = \lim_{N\to \infty}\frac{M_2}{N+\frac{1}{2}}$ sufficiently close to 1, we only need to consider the terms with
$$ \frac{n_\gamma'}{N+\frac{1}{2}} \sim \frac{1}{2} \quad \text{and} \quad \frac{l_\gamma'}{N+\frac{1}{2}} \sim \frac{1}{4}$$

For the same reason, since $\frac{2n_{\alpha - 1}}{N+\frac{1}{2}}$ is closed to one, it suffices to consider those $n_\zeta$ and $l_\zeta$ with
$$ \frac{n_\zeta}{N+\frac{1}{2}} \sim \frac{1}{2} \quad \text{and} \quad \frac{l_\zeta}{N+\frac{1}{2}} \sim \frac{1}{4}$$

From now on, we restrict our attention to the sum of the terms satisfying the above conditions.

Define
\begin{align}
&\Psi_{N}(WL; z_{2\gamma-1}, z_{2\gamma+1},z_{2\gamma+2})\notag\\
&= \frac{1}{2\pi i} \lt[
 - (2\pi i (z_{2\gamma-1} - 1))(2\pi i z_{2\gamma+1} + 2\pi i z_{2\gamma+2}) +\varphi_r \lt( z_{2\gamma-1}\pi  - \pi z_{\gamma+1} - \pi z_{2\gamma+2} -\frac{\pi}{2N+1}\rt) \rt.\notag \\
&\qquad  - \varphi_r\lt(  z_{2\gamma-1}\pi - \pi z_{2\gamma+2}- \frac{\pi}{2N+1}\rt) + \varphi_r \lt(\pi z_{2\gamma+2} + \frac{\pi}{2N+1}\rt)   \notag\\
&\lt.\qquad - \varphi_r\lt( \pi z_{2\gamma+1} + \pi z_{2\gamma+2} + \frac{\pi}{2N+1}\rt) + \varphi_r\lt( \pi z_{2\gamma+1} + \frac{\pi}{2N+1}\rt) \rt] 
\end{align}

Define
\begin{align}
\Xi_N^{\pm} (WL; z'_{2\beta-1}, z_{2\alpha-3},z_{\zeta},z_{\zeta+1})
&= \frac{1}{2\pi i}\lt[ \pm \lt(2\pi i \lt(z'_{2\beta-1}\rt)\rt) \lt(2\pi i \lt(z_\zeta -\frac{1}{2} \rt)\rt) \rt. \notag \\
&\qquad  - (2\pi i (z_{2\alpha-3}))(2\pi i z_\zeta + 2\pi i z_{\zeta+1})  \notag \\
&\qquad  +  \varphi_r \lt( z_{2\alpha-3}  - \pi z_\zeta - \pi z_\zeta -\frac{\pi}{2N+1}\rt)   \notag\\
&\qquad  - \varphi_r\lt(  z_{2\alpha-3} - \pi z_{\zeta+1}- \frac{\pi}{2N+1}\rt) \notag \\
&\qquad  + \varphi_r \lt(\pi z_{\zeta+1} + \frac{\pi}{2N+1}\rt)  - \varphi_r\lt( \pi z_\zeta + \pi z_{\zeta+1} + \frac{\pi}{2N+1}\rt) \notag\\
&\lt.\qquad  + \varphi_r\lt( \pi z_\zeta + \frac{\pi}{2N+1}\rt) \rt] 
\end{align}
and
\begin{align}
&\Phi_{M_1,M_2}^{\pm} \lt(W^\alpha_\beta; z_1, z_2, \dots, z_{2\alpha-3},z_{2\alpha-2}, 
z_1', z_2', \dots, z_{2\beta-1}',z_{2\beta}', z_{\zeta},z_{\zeta+1} \rt) \notag\\
&=  \Psi_{N}\lt(WL; \frac{M_1}{N+\frac{1}{2}}; z_1, z_2\rt) + \sum_{\gamma=2}^{\alpha-1} \Psi_{N}(WL;2z_{2\gamma-3}, z_{2\gamma-1},z_{2\gamma}) 
\notag \\
&\qquad + \Psi_{N}\lt(WL; \frac{M_2}{N+\frac{1}{2}}; z_1', z_2'\rt) + \sum_{\gamma=2}^{\beta} \Psi_{N}(WL;2z_{2\gamma-3}', z_{2\gamma-1}',z_{2\gamma}') \notag\\
&\qquad\qquad + \Xi^{\pm}_N (WL; z'_{2\beta-1}, z_{2\alpha-3},z_{\zeta},z_{\zeta+1})
\end{align} 

Then we can write
\begin{align}
J_{M_1,M_2}(W^\alpha_\beta,e^{\frac{2\pi i}{N+\frac{1}{2}}})
\stackrel[N \to \infty]{\sim}{}  \frac{e^{\frac{2\pi i}{N+\frac{1}{2}} \lt(M_1^2 - \frac{M_1}{2} + M_2^2 - \frac{M_2}{2} - 1\rt)}}{2i\sin(\frac{\pi}{N+\frac{1}{2}})}
\exp\lt(  -\varphi_r \lt(  \frac{\pi}{2N+1}\rt)\rt)^{\alpha+\beta}
(I_+ + I_-),
\end{align}
where
\begin{align}
I_\pm 
= e^{\pm \frac{2\pi i}{N+\frac{1}{2}}} &\sum_{n_1=0}^{2M_1-1}\sum_{n_2=0}^{2n_1}\dots\sum_{n_{\alpha-1}=0}^{2n_{\alpha-2}}
\sum_{n_1'=0}^{2M_1-1}\sum_{n_2'=0}^{2n_1'}\dots\sum_{n_{\beta}'=0}^{2n_{\beta-1}'} 
\sum_{n_{\zeta}=0}^{2n_{\alpha-1}} \notag\\
& \sum_{l_1=0}^{M_1-1-n_1}\sum_{l_2=0}^{2n_1-n_2}\sum_{l_3=0}^{2n_2-n_3}\dots\sum_{l_{\alpha-1}=0}^{2n_{\alpha-2}-n_{\alpha-1}} 
\sum_{l_1'=0}^{M_2-1-n_1}\sum_{l_2'=0}^{2n_1'-n_2'}\sum_{l_3'=0}^{2n_2'-n_3'}\dots \sum_{l_{\beta}'=0}^{2n_{\beta-1}'-n_{\beta}'} 
\sum_{l_\zeta=0}^{2n_{\alpha-1} - n_\zeta}\notag\\
&\qquad e^{2\pi i \lt( 
\sum_{\gamma=2}^{\alpha-1} (3z_{2\gamma-3}+(z_{2\gamma}+z_{2\gamma-1}))
+
\sum_{\gamma=2}^{\beta} (3z'_{2\gamma-3}+(z'_{2\gamma}+z'_{2\gamma-1}))
+
(z_\zeta+z_{\zeta+1})
\pm 
(z'_{2\beta-1}+z_\zeta)\rt)}\notag\\
&\qquad \exp\lt(\frac{N+\frac{1}{2}}{2\pi i} \Phi_{M_1,M_2}^{\pm} \lt(W^\alpha_\beta; \frac{n_1}{N+\frac{1}{2}}, \frac{l_1}{N+\frac{1}{2}}, \dots, \frac{n_{\alpha-1}}{N+\frac{1}{2}}, \frac{l_{\alpha-1}}{N+\frac{1}{2}}, \rt.\rt. \notag\\
&\qquad \lt.\lt.
\frac{n_1'}{N+\frac{1}{2}}, \frac{l_1'}{N+\frac{1}{2}}, \dots, \frac{n_\beta'}{N+\frac{1}{2}}, \frac{l_\beta'}{N+\frac{1}{2}}, \frac{n_\zeta}{N+\frac{1}{2}}, \frac{l_\zeta}{N+\frac{1}{2}} \rt)  \rt)
\end{align}

Take $N\to \infty$, we have
\begin{align}
&\Phi^\pm_{(s_1,s_2)} (W^\alpha_\beta; z_1, z_2, \dots, z_{2\alpha-3},z_{2\alpha-2}, 
z_1', z_2', \dots, z_{2\beta-1}',z_{2\beta}', z_{\zeta},z_{\zeta+1})  \notag\\
&=  \Psi\lt(WL; \frac{M_1}{N+\frac{1}{2}}; z_1, z_2\rt) + \sum_{\gamma=2}^{\alpha-1} \Psi(WL;2z_{2\gamma-3}, z_{2\gamma-1},z_{2\gamma})  \notag \\
&\qquad + \Psi\lt(WL; \frac{M_2}{N+\frac{1}{2}}; z_1', z_2'\rt) + \sum_{\gamma=2}^{\beta} \Psi(WL;2z_{2\gamma-3}', z_{2\gamma-1}',z_{2\gamma}') \notag \\
&\qquad\qquad + \Xi^{\pm} (WL; 2z'_{2\beta-1}, 2z_{2\alpha-3},z_{\zeta},z_{\zeta+1})
\end{align}
with
\begin{align}
&\Psi (WL; z_{2\gamma - 1},z_{2\gamma + 1},z_{2\gamma + 2}) \notag\\
&=  8 \pi i (z_{2\gamma-1})^2 + \frac{1}{2\pi i} \lt[  - (2\pi i (z_{2\gamma - 1}-1))(2\pi i z_{2\gamma + 1} + 2\pi i z_{2\gamma + 2})  \rt. \notag\\
&\qquad + \Li\lt(e^{2\pi i (z_{2\gamma - 1}-1) - 2\pi i z_{2\gamma + 1} - 2\pi i z_{2\gamma + 2}}\rt) - \Li\lt( e^{2\pi i (z_{2\gamma - 1}-1) - 2\pi i z_{2\gamma + 2} }\rt) \notag\\
&\qquad \lt.+ \Li \lt(e^{2\pi i z_{2\gamma + 2}}\rt) - \Li\lt(e^{2\pi iz_{2\gamma + 1}+ 2\pi i z_{2\gamma + 2}}\rt) + \Li\lt(e^{2\pi i z_{2\gamma + 1}}\rt)  \rt] \\
&\Xi^\pm (WL;z'_{2\beta-1}, z_{2\alpha-3}, z_{\zeta},z_{\zeta+1}) \notag\\
&=  8\pi i (z'_{2\beta-1})^{2} + \frac{1}{2\pi i} \lt[ \pm \lt(2\pi i \lt( z'_{2\beta-1} - 1\rt)\rt) \lt(2\pi i \lt(z_\zeta -\frac{1}{2} \rt)\rt) \rt.\notag\\
&\qquad - (2\pi i (z_{2\alpha-3}-1))(2\pi i z_\zeta + 2\pi i z_{\zeta+1})     \notag \\
&\qquad + \Li\lt(e^{2\pi i (z_{2\alpha-3}-1) - 2\pi i z_\zeta - 2\pi i z_{\zeta+1}}\rt) - \Li\lt( e^{2\pi i (z_{2\alpha-3}-1) - 2\pi i z_{\zeta+1} }\rt) \notag\\
&\qquad \lt.+ \Li \lt(e^{2\pi i z_{\zeta+1}}\rt) - \Li\lt(e^{2\pi iz_\zeta+ 2\pi i z_{\zeta+1}}\rt) + \Li\lt(e^{2\pi i z_\zeta}\rt)\rt]
\end{align}

Note that the function 
$$\Psi(WL;2z_{2\gamma-1}, z_{2\gamma+1},z_{2\gamma+2})-8\pi i (z_{2\gamma-1})^2$$
is the potential function $\Phi^{(1,s_2)}(z_1,z_2)$ in \cite{W19} with $s_2 = 2z_{2\gamma-1}$, $z_1 =  z_{2\gamma+1}$ and $z_2 = z_{2\gamma+2}$. Besides, the function 
$$\Xi^{\pm}(WL; 2z'_{2\beta-1}, 2z_{2\alpha-3}, z_{\zeta},z_{\zeta+1})
-8\pi i (z_{2\beta-1}')^2$$
is the potential function $\Phi^{\pm(s_1,s_2)}(z_1,z_2)$ in \cite{W19} with $s_1=2z'_{2\beta-1}$, $s_2 = 2z_{2\alpha-3}$, $z_1 =  z_{\zeta}$ and $z_2 = z_{\zeta+1}$.

By the maximum point estimation, it remains to study the potential function 
$$\Phi^\pm \lt(W^\alpha_\beta; z_1, z_2, \dots, z_{2\alpha-3},z_{2\alpha-2}, 
z_1', z_2', \dots, z_{2\beta-1}',z_{2\beta}', z_{\zeta},z_{\zeta+1} \rt) $$
defined on the domain
\begin{align*}
D &= \lt\{(z_1,z_2,\dots,z_{2\alpha-3},z_{2\alpha-2},
z_1', z_2', \dots, z_{2\beta-1}',z_{2\beta}', z_{\zeta},z_{\zeta+1})\mid \rt.\\
&\qquad |z_{2\gamma-1}-\frac{1}{2}|, |z_{2\gamma} - \frac{1}{4}|<\eta \text{, where } \gamma=1,\dots, \alpha-1 \\
&\qquad |z_{2\gamma-1}'-\frac{1}{2}|, |z_{2\gamma}' - \frac{1}{4}|<\eta \text{, where } \gamma=1,\dots, \beta \\
&\qquad  |z_{\zeta}-\frac{1}{2}|, |z_{\zeta+1} - \frac{1}{4}|<\eta
\lt. \rt\}
\end{align*}
where $\eta>0$ is any sufficiently small positive constant.

By considering the Talyor series expansion, we have
\begin{align*}
&\exp\lt( \lt(N+\frac{1}{2}\rt) \Phi^\pm_{M_1,M_2}\lt(W^\alpha_\beta; z_1, z_2, \dots, z_{2\alpha-3},z_{2\alpha-2}, 
z_1', z_2', \dots, z_{2\beta-1}',z_{2\beta}', z_{\zeta},z_{\zeta+1} \rt)  \rt)\\
&= E(W^\alpha_\beta; z_1, z_2, \dots, z_{2\alpha-3},z_{2\alpha-2}, 
z_1', z_2', \dots, z_{2\beta-1}',z_{2\beta}', z_{\zeta},z_{\zeta+1} )\\
&\qquad \exp\lt( \lt(N+\frac{1}{2}\rt) \Phi^{\pm} (W^\alpha_\beta ;z_1, z_2, \dots, z_{2\alpha-3},z_{2\alpha-2}, 
z_1', z_2', \dots, z_{2\beta-1}',z_{2\beta}', z_{\zeta},z_{\zeta+1}) \rt)\lt(1+O\lt(\frac{1}{N+\frac{1}{2}}\rt)\rt),
\end{align*}
where
\begin{align*}
&E(W^\alpha_\beta; z_1, z_2, \dots, z_{2\alpha-3},z_{2\alpha-2}, 
z_1', z_2', \dots, z_{2\beta-1}',z_{2\beta}', z_{\zeta},z_{\zeta+1}  ) \\
&= \exp\lt(\log(1-e^{-2\pi i z_1 - 2\pi i z_2}) - \log(1-e^{-2\pi i z_2}) - \frac{1}{2}\log(1-e^{2\pi i z_2}) \rt.\\
&\qquad\qquad + \frac{1}{2}\log(1-e^{2\pi i z_1 + 2\pi i z_2}) - \frac{1}{2}\log(1-e^{2\pi i z_1}) \\
&\qquad + \sum_{\gamma=2}^{\alpha-1} \lt(
-\frac{1}{2}\log(1-e^{4\pi iz_{2\gamma - 3}-2\pi i z_{2\gamma-1} - 2\pi i z_{2\gamma}})
+ \frac{1}{2}\log(1-e^{4\pi i z_{2\gamma -3} - 2\pi i z_{2\gamma}})
\rt. \\
&\qquad\qquad\quad \lt. - \frac{1}{2}\log(1-e^{2\pi i z_{2\gamma}}) + \frac{1}{2}\log(1-e^{2\pi i z_{2\gamma-1} + 2\pi i z_{2\gamma}}) - \frac{1}{2}\log(1-e^{2\pi i z_{2\gamma-1}}) \rt) \\
&\quad\quad\quad + \log(1-e^{-2\pi i z'_1 - 2\pi i z'_2}) - \log(1-e^{-2\pi i z'_2}) - \frac{1}{2}\log(1-e^{2\pi i z'_2}) \\
&\qquad\qquad + \frac{1}{2}\log(1-e^{2\pi i z'_1 + 2\pi i z'_2}) - \frac{1}{2}\log(1-e^{2\pi i z'_1}) \\
&\qquad + \sum_{\gamma=2}^\beta \lt(
-\frac{1}{2}\log(1-e^{4\pi iz'_{2\gamma - 3}-2\pi i z'_{2\gamma-1} - 2\pi i z'_{2\gamma}})
+ \frac{1}{2}\log(1-e^{4\pi i z'_{2\gamma -3} - 2\pi i z'_{2\gamma}})
\rt. \\
&\qquad\qquad\quad \lt. - \frac{1}{2}\log(1-e^{2\pi i z'_{2\gamma}}) + \frac{1}{2}\log(1-e^{2\pi i z'_{2\gamma-1} + 2\pi i z'_{2\gamma}}) - \frac{1}{2}\log(1-e^{2\pi i z'_{2\gamma-1}})
\rt) \\
&\qquad + \lt(
-\frac{1}{2}\log(1-e^{4\pi iz_{2\alpha - 3}-2\pi i z_{\zeta} - 2\pi i z_{\zeta+1}})
+ \frac{1}{2}\log(1-e^{4\pi i z_{2\alpha -3} - 2\pi i z'_{\zeta+1}})
\rt. \\
&\qquad\qquad\quad \lt. - \frac{1}{2}\log(1-e^{2\pi i z_{\zeta+1}}) + \frac{1}{2}\log(1-e^{2\pi i z_{\zeta} + 2\pi i z'_{\zeta+1}}) - \frac{1}{2}\log(1-e^{2\pi i z_{\zeta}})
\rt) 
\end{align*}

Recall from Section \ref{GVCcp} that for $s_1,s_2$ sufficiently close to $1$, we have a holomorphic family of solution $(z_1(s), z_2(s))=(\frac{1}{2},z_2(s))$ for the critical point equations
$$ \frac{\partial}{\partial z_1} \Phi^{(s)}(WL,z_1,z_2) 
= \frac{\partial}{\partial z_2} \Phi^{(s)}(WL,z_1,z_2)
= 0 $$
In particular, when $z_1=\frac{1}{2}$ and $z_2=z_2\lt(\frac{M_1}{N+\frac{1}{2}}\rt)$, we have
$$\frac{\partial\Psi(WL;\frac{M_1}{N+\frac{1}{2}},z_1,z_2)}{\partial z_1} = 0 $$

Note that the critical point equations of the potential function
$$\Phi^\pm(W^\alpha_\beta; z_1, z_2, \dots, z_{2\alpha-3},z_{2\alpha-2}, 
z_1', z_2', \dots, z_{2\beta-1}',z_{2\beta}', z_{\zeta},z_{\zeta+1})$$
are given by
\begin{empheq}[left = \empheqlbrace]{align}
0 &= \frac{\partial\Psi(WL; \frac{M_1}{N+\frac{1}{2}}, z_1,z_2)}{\partial z_1} +  \frac{\partial\Psi(WL;2z_1,z_3,z_4)}{\partial z_1} \\
0 &= \frac{\partial\Psi(WL;\frac{M_1}{N+\frac{1}{2}}, z_1,z_2)}{\partial z_2} \\
0 &= \frac{\partial\Psi(WL;2z_{2\gamma-3},z_{2\gamma-1},z_{2\gamma})}{\partial z_{2\gamma -1}} + \frac{\partial\Psi(WL;2z_{2\gamma-1},z_{2\gamma+1},z_{2\gamma+2})}{\partial z_{2\gamma -1}}  \qquad \text{for $\gamma=2,\dots, \alpha-2$} \\
0 &= \frac{\partial\Psi(WL;2z_{2\alpha-5},z_{2\alpha-3},z_{2\alpha-2})}{\partial z_{2\alpha -3}} +  \frac{\partial\Xi^\pm(WL; 2z'_{2\beta-1}, 2z_{2\alpha-3}, z_{\zeta},z_{\zeta+1})}{\partial z_{2\alpha -3}} \\
0 &= \frac{\partial\Psi(WL;2z_{2\gamma-3}, z_{2\gamma-1},z_{2\gamma})}{\partial z_{2\gamma}}  
\qquad\qquad\qquad\qquad\qquad\qquad\qquad\quad
 \text{for $\gamma=2,\dots, \alpha-1$} \\
 0 &= \frac{\partial\Psi(WL; \frac{M_2}{N+\frac{1}{2}}, z_1',z_2')}{\partial z_1'} +  \frac{\partial\Psi(WL;2z_1',z_3',z_4')}{\partial z_1'} \\
0 &= \frac{\partial\Psi(WL;\frac{M_2}{N+\frac{1}{2}}, z_1',z_2')}{\partial z_2'} \\
0 &= \frac{\partial\Psi(WL;2z_{2\gamma-3}',z_{2\gamma-1}',z_{2\gamma}')}{\partial z_{2\gamma -1}'} + \frac{\partial\Psi(WL;2z_{2\gamma-1}',z_{2\gamma+1}',z_{2\gamma+2}')}{\partial z_{2\gamma -1}'}  \qquad \text{for $\gamma=2,\dots, \beta-1$} \\
0 &= \frac{\partial\Psi(WL;2z_{2\beta-3}',z_{2\beta-1}',z_{2\beta}')}{\partial z_{2\beta -1}'} +  \frac{\partial\Xi^\pm(WL; 2z_{2\beta-1}', 2z_{2\alpha-3},  z_{\zeta},z_{\zeta+1})}{\partial z_{2\beta -1}'} \\
0 &= \frac{\partial\Psi(WL;2z_{2\gamma-3}', z_{2\gamma-1}',z_{2\gamma}')}{\partial z_{2\gamma}'}  
\qquad\qquad\qquad\qquad\qquad\qquad\qquad\quad
 \text{for $\gamma=2,\dots, \beta-1$} \\
0 &= \frac{\partial\Xi^\pm(WL; 2z_{2\beta-1}', 2z_{2\alpha-3}, z_{\zeta},z_{\zeta+1})}{\partial z_{\zeta}} \\
0 &= \frac{\partial\Xi^\pm(WL; 2z_{2\beta-1}', 2z_{2\alpha-3}, z_{\zeta},z_{\zeta+1})}{\partial z_{\zeta+1}}
\end{empheq}

To solve the critical point equations for the potential function 
$$\Phi^\pm(W^\alpha_\beta; z_1, z_2, \dots, z_{2\alpha-3},z_{2\alpha-2}, 
z_1', z_2', \dots, z_{2\beta-1}',z_{2\beta}', z_{\zeta},z_{\zeta+1}),$$
note that if we put $z_1=z_3=\frac{1}{2}, z_2= z_2\lt(\frac{M_1}{N+\frac{1}{2}}\rt), z_4=\frac{1}{4}$, we get
\begin{align*}
\frac{\partial\Psi(WL;\frac{M_1}{N+\frac{1}{2}},z_1,z_2)}{\partial z_1} +  \frac{\partial\Psi(WL;2z_1,z_3,z_4)}{\partial z_1}
&= 16\pi i \lt(\frac{1}{2}\rt)- 2(2\pi i)(\frac{1}{2}+\frac{1}{4}) \notag\\
&\qquad -2\log(1-e^{-2\pi i(3/4)}) + 2\log(1-e^{-2\pi i(1/4)}) \\
&= 6\pi i \\\\
\frac{\partial\Psi(WL;\frac{M_1}{N+\frac{1}{2}},z_1,z_2)}{\partial z_2} 
&= 0
\end{align*}
Besides, when $z_{2\gamma-3}=z_{2\gamma-1}=z_{2\gamma+1}=\frac{1}{2}, z_{2\gamma}=z_{2\gamma+2}=\frac{1}{4}$, we have
\begin{alignat*}{2}
\frac{\partial\Psi(WL;2z_{2\gamma-3},z_{2\gamma-1},z_{2\gamma})}{\partial z_{2\gamma -1}} + \frac{\partial\Psi(WL;2z_{2\gamma-1},z_{2\gamma+1},z_{2\gamma+2})}{\partial z_{2\gamma -1}} 
&= 6\pi i &&\qquad\qquad \text{for $\gamma=2,\dots, \alpha-2$} \\
\frac{\partial\Psi(WL;2z_{2\gamma-3}, z_{2\gamma-1},z_{2\gamma})}{\partial z_{2\gamma}} &= 0  &&\qquad\qquad \text{for $\gamma=2,\dots, \alpha-1$} 
\end{alignat*}

By the same computation, when $z_1'=z_3'=\frac{1}{2}, z_2'= z_2\lt(\frac{M_2}{N+\frac{1}{2}}\rt), z_4'=\frac{1}{4}$, we have
\begin{align*}
\frac{\partial\Psi(WL; \frac{M_2}{N+\frac{1}{2}}, z_1', z_2')}{\partial z_1'} +  \frac{\partial\Psi(WL;2z_1',z_3',z_4')}{\partial z_1'}
&= 16\pi i \lt(\frac{1}{2}\rt)- 2(2\pi i)(\frac{1}{2}+\frac{1}{4}) \notag\\
&\qquad -2\log(1-e^{-2\pi i(3/4)}) + 2\log(1-e^{-2\pi i(1/4)}) \\
\frac{\partial\Phi(WL;z_1',z_2')}{\partial z_2'} 
&= 0
\end{align*}
Besides, when $z_{2\gamma-3}'=z_{2\gamma-1}'=z_{2\gamma+1}'=\frac{1}{2}, z_{2\gamma}'=z_{2\gamma+2}'=\frac{1}{4}$, we have
\begin{alignat*}{2}
\frac{\partial\Psi(WL; 2z_{2\gamma-3}',z_{2\gamma-1}',z_{2\gamma}')}{\partial z_{2\gamma -1}'} + \frac{\partial\Psi(WL;2z_{2\gamma-1}',z_{2\gamma+1}',z_{2\gamma+2}')}{\partial z_{2\gamma -1}'} 
&= 6\pi i &&\qquad\qquad\qquad \text{for $\gamma=2,\dots, \beta-1$} \\
\frac{\partial\Psi(WL;2z_{2\gamma-3}', z_{2\gamma-1}',z_{2\gamma}')}{\partial z_{2\gamma}'} &= 0  &&\qquad\qquad\qquad \text{for $\gamma=2,\dots, \beta$}
\end{alignat*}
Finally, when $z_{2\alpha-3}=z_{2\beta-1}'=\frac{1}{2}$, $z_{\zeta}=\frac{1}{2}$, $z_{\zeta+1}=\frac{1}{4}$, we have

\begin{align*}
\frac{\partial\Psi(WL;2z_{2\alpha-5},z_{2\alpha-3},z_{2\alpha-2})}{\partial z_{2\alpha -3}} +  \frac{\partial\Xi^\pm(WL; 2z_{2\alpha-3}, 2z_{2\beta-1}, z_{\zeta},z_{\zeta+1})}{\partial z_{2\alpha -3}} 
&=  6\pi i \\
\frac{\partial\Psi(WL;2z_{2\beta-3}',z_{2\beta-1}',z_{2\beta}')}{\partial z_{2\beta -1}'} +  \frac{\partial\Xi^\pm(WL; 2z_{2\alpha-3}', 2z_{2\beta-1}', z_{\zeta},z_{\zeta+1})}{\partial z_{2\beta -1}'} 
&= 8\pi i \\
 \frac{\partial\Xi^\pm(WL; 2z_{2\beta-1}', 2z_{2\alpha-3}, z_{\zeta},z_{\zeta+1})}{\partial z_{\zeta}}
= \frac{\partial\Xi^\pm(WL; 2z_{2\beta-1}', 2z_{2\alpha-3}, z_{\zeta},z_{\zeta+1})}{\partial z_{\zeta+1}}
&= 0
\end{align*}

As a result, the point $\mathbf{z_{M_1,M_2}}=(z_1,\dots,z_{2\alpha-2},z'_1,\dots,z'_{2\beta},z_\zeta,z_{\zeta+1})$ with
\begin{align}
z_1=\frac{1}{2}, \quad z_2= z_2\lt(\frac{M_1}{N+\frac{1}{2}}\rt), \quad&z_3 = z_5 = \dots = z_{2\alpha-3}=\frac{1}{2},\quad z_4=z_6=\dots=z_{2\alpha-2} = \frac{1}{4} \notag\\
z_1'=\frac{1}{2}, \quad z_2'= z_2\lt(\frac{M_2}{N+\frac{1}{2}}\rt), \quad&z_3' = z_5' = \dots = z_{2\alpha-3}'=\frac{1}{2},\quad z_4'=z_6'=\dots=z_{2\beta}' = \frac{1}{4} \notag\\
&z_\zeta = \frac{1}{2},\quad z_{\zeta+1} = \frac{1}{4}
\end{align}
is a critical point of the Fourier coefficient 
$$\Phi^\pm(W^\alpha_\beta; z_1, z_2, \dots, z_{2\alpha-3},z_{2\alpha-2}, 
z_1', z_2', \dots, z_{2\beta-1}',z_{2\beta}', z_{\zeta},z_{\zeta+1})
-6\pi i \sum_{\gamma=2}^{\alpha} z_{2\gamma-3} - 6\pi i \sum_{\gamma=2}^{\beta}
z'_{2\gamma-3} - 8\pi i z'_{2\beta-1}$$
Furthermore, we have
$$ \Phi^+(W^\alpha_\beta; \mathbf{z_{M_1,M_2}}) = \Phi^-(W^\alpha_\beta;\mathbf{z_{M_1,M_2}}) $$
with
\begin{align}
\Re\Phi^\pm(W^\alpha_\beta; {\mathbf{z}_{M_1,M_2}})
&= \Re \lt[\Psi_{N}\lt(WL; \frac{M_1}{N+\frac{1}{2}}; \frac{1}{2}, z\lt(\frac{M_1}{N+\frac{1}{2}}\rt)\rt) + \sum_{\gamma=1}^{\alpha-2} \Psi_{N}\lt(WL;1,\frac{1}{2},\frac{1}{4}\rt) \rt. \notag \\
&\qquad\quad + \Psi_{N}\lt(WL; \frac{M_2}{N+\frac{1}{2}}; \frac{1}{2}, z\lt(\frac{M_1}{N+\frac{1}{2}}\rt)\rt) + \sum_{\gamma=1}^{\beta-2} \Psi_{N}\lt(WL;1,\frac{1}{2},\frac{1}{4}\rt) \notag \\
&\qquad\qquad \lt. + \Xi^{\pm} \lt(WL; 1, 1, \frac{1}{2}, \frac{1}{4}\rt) \rt] \notag\\
&= \frac{1}{2\pi}\lt[ 
\Vol\lt(\SS^3\backslash WL; s_1=1,s_2 = \frac{M_1}{N+\frac{1}{2}} \rt)
+ 
\Vol\lt(\SS^3\backslash WL; s_1=1,s_2 = \frac{M_2}{N+\frac{1}{2}} \rt) \rt. \notag\\
&\qquad\quad \lt. + (\alpha+\beta-2) \Vol(\SS^3\backslash WL) \rt]
\end{align}
Next, the following lemma guarantees that the Hessian of the potential function at the critical point is non-zero.
\begin{lemma}\label{nonsing2} We have $\det(\Hess\Phi^\pm_{M_1,M_2}(W^\alpha_\beta;\mathbf{z_{M_1,M_2}})) \neq 0$ for $s_1, s_2$ sufficiently close to $1$.
\end{lemma}
\begin{proof}
The idea of the proof is the same as Lemma~\ref{nonsing1}. Consider the matrices
\begin{align}
B_1
&=\Re\Hess\lt(\Phi\lt(WL; \frac{1}{2}, \frac{1}{4}\rt)\rt)
= -2\pi 
\begin{pmatrix}
1& 1 \\
1 & 2
\end{pmatrix}\\
B_2
&=\Re\Hess\lt(\Psi\lt(WL; 1, \frac{1}{2}, \frac{1}{4}\rt)\rt)
= -2\pi 
\begin{pmatrix}
-2 & -1 & -1 \\
-1 & 1 & 1 \\
-1 & 1 & 2
\end{pmatrix} \\
B_3
&= \Re\Hess\lt(\Xi^\pm \lt(WL; 1, 1, \frac{1}{2}, \frac{1}{4}\rt)\rt)
= -2\pi 
\begin{pmatrix}
0 & 0 & 0 & 0 \\
0 & -2 & -1 & -1 \\
0 &-1 & 1 & 1 \\
0 & -1 & 1 & 2
\end{pmatrix}
\end{align}

Next, consider the function 
\begin{align}
F(\mathbf{z}) 
&= \lim_{N\to\infty} \Phi^\pm_{N,N}(W^\alpha_\beta; \mathbf{z}) \notag\\
&=  \Psi\lt(WL; 1; z_1, z_2\rt) + \sum_{\gamma=2}^{\alpha-1} \Psi(WL;2z_{2\gamma-3}, z_{2\gamma-1},z_{2\gamma})  \notag \\
&\qquad + \Psi\lt(WL; 1; z_1', z_2'\rt) + \sum_{\gamma=2}^{\beta} \Psi(WL;2z_{2\gamma-3}', z_{2\gamma-1}',z_{2\gamma}') \notag \\
&\qquad\qquad + \Xi^{\pm} (WL; 2z'_{2\beta-1}, 2z_{2\alpha-3},z_{\zeta},z_{\zeta+1})
\end{align}
Note that as $N\to\infty$, $\mathbf{z_{N,N}}$ tends to the point $\mathbf{z_\infty} = (z_1,\dots,z_{2\alpha-2},z'_1,\dots,z'_{2\beta},z_\zeta,z_{\zeta+1})$ with
\begin{align}
z_1=\frac{1}{2}, \quad z_2= \frac{1}{4}, \quad&z_3 = z_5 = \dots = z_{2\alpha-3}=\frac{1}{2},\quad z_4=z_6=\dots=z_{2\alpha-2} = \frac{1}{4} \notag\\
z_1'=\frac{1}{2}, \quad z_2'= \frac{1}{4}, \quad&z_3' = z_5' = \dots = z_{2\alpha-3}'=\frac{1}{2},\quad z_4'=z_6'=\dots=z_{2\beta}' = \frac{1}{4} \notag\\
&z_\zeta = \frac{1}{2},\quad z_{\zeta+1} = \frac{1}{4}
\end{align}
Let $B= \Re\Hess F(\mathbf{z_\infty})$. Note that for $\mathbf{z}=(z_1, z_2, \dots, z_{2\alpha-3},z_{2\alpha-2}, 
z_1', z_2', \dots, z_{2\beta-1}',z_{2\beta}', z_{\zeta},z_{\zeta+1})\in \RR^{2\alpha+2\beta}$,
\begin{align}
\mathbf{z} B \mathbf{z}^T
&= (z_1, z_2) B_1 (z_1, z_2)^T 
+ \sum_{\gamma=2}^{\alpha-1} (z_{2\gamma-1}, z_{2\gamma+1}, z_{2\gamma+2})B_2(z_{2\gamma-1}, z_{2\gamma+1}, z_{2\gamma+2})^T \notag\\
&\qquad (z'_1, z'_2) B_1 (z'_1, z'_2)^T 
+ \sum_{\gamma=2}^{\beta} (z'_{2\gamma-1}, z'_{2\gamma+1}, z'_{2\gamma+2})B_2(z_{2\gamma-1}, z_{2\gamma+1}, z_{2\gamma+2})^T \notag\\
&\qquad + (z'_{2\beta-1}, z_{\alpha-3}, z_{\zeta}, z_{\zeta+1})B_3(z'_{2\beta-1}, z_{2\alpha-3}, z_{\zeta}, z_{\zeta+1})^T \notag\\
&= (z_1, z_2) B_1 (z_1, z_2)^T 
+ \sum_{\gamma=2}^{\alpha-1} (z_{2\gamma-1}, z_{2\gamma+1}, z_{2\gamma+2})B_2(z_{2\gamma-1}, z_{2\gamma+1}, z_{2\gamma+2})^T \notag\\
&\qquad (z'_1, z'_2) B_1 (z'_1, z'_2)^T 
+ \sum_{\gamma=2}^{\beta} (z'_{2\gamma-1}, z'_{2\gamma+1}, z'_{2\gamma+2})B_2(z_{2\gamma-1}, z_{2\gamma+1}, z_{2\gamma+2})^T \notag\\
&\qquad + (z_{2\alpha-3}, z_{\zeta}, z_{\zeta+1})B_2(z_{2\alpha-3}, z_{\zeta}, z_{\zeta+1})^T \notag\\
&\leq 0
\end{align}
with equality holds if and only if $\mathbf{z}=0$. As a result, since the real part of the matrix 
$$\Hess\Phi^\pm_{M_1,M_2}(W^\alpha_\beta;\mathbf{z_{M_1,M_2}})$$
is negative definite for $s_1, s_2$ sufficiently close to $1$, by Lemma in \cite{L81}, it is non-singular.
\end{proof}
As a result, the asymptotic expansion formula for $I^\pm$ is given by
\begin{align}
I_\pm 
&\stackrel[N \to \infty]{\sim}{} 
\frac{ \pi^{\alpha+\beta}e^{\frac{\pi i}{2}(\alpha+\beta-2)}E(W^\alpha_\beta; \mathbf{z_{M_1,M_2}})}{\sqrt{\det(-\Hess\Phi^\pm_{M_1,M_2}(W^\alpha_\beta;\mathbf{z_{M_1,M_2}}))}}
\lt(N+\frac{1}{2}\rt)^{\alpha+\beta}
\exp\lt(\lt(N+\frac{1}{2}\rt)\Phi^\pm_{M_1,M_2}(W^\alpha_\beta;\mathbf{z_{M_1,M_2}})\rt)
\end{align} 
and
\begin{align}
&J_{M_1,M_2}(W^\alpha_\beta,e^{\frac{2\pi i}{N+\frac{1}{2}}}) \notag\\
&\stackrel[N \to \infty]{\sim}{}  
\frac{e^{\frac{2\pi i}{N+\frac{1}{2}} \lt(M_1^2 - \frac{M_1}{2} + M_2^2 - \frac{M_2}{2} - 1\rt)}}{2i\sin(\frac{\pi}{N+\frac{1}{2}})}
\exp\lt(  -\varphi_r \lt(  \frac{\pi}{2N+1}\rt)\rt)^{\alpha+\beta} 
(I_+ + I_-)
\notag\\
&\stackrel[N \to \infty]{\sim}{} 
\frac{e^{\frac{2\pi i}{N+\frac{1}{2}} \lt(M_1^2 - \frac{M_1}{2} + M_2^2 - \frac{M_2}{2} - 1\rt)}}{2i\sin(\frac{\pi}{N+\frac{1}{2}})} 
\exp\lt(  -\varphi_r \lt(  \frac{\pi}{2N+1}\rt)\rt)^{\alpha+\beta} 
\pi^{\alpha+\beta}e^{\frac{\pi i}{2}(\alpha+\beta-2)}E(W^\alpha_\beta; \mathbf{z_{M_1,M_2}}) \notag\\
&\qquad \qquad \times \lt[ 
\frac{1}{\sqrt{\det(-\Hess\Phi^+_{M_1,M_2}(W^\alpha_\beta;\mathbf{z_{M_1,M_2}}))}} 
+
\frac{1}{\sqrt{\det(-\Hess\Phi^-_{M_1,M_2}(W^\alpha_\beta;\mathbf{z_{M_1,M_2}}))}}  \rt] \notag \\
&\qquad \qquad \qquad \times \exp\lt(\lt(N+\frac{1}{2}\rt)\Phi^+_{M_1,M_2}(W^\alpha_\beta;\mathbf{z_{M_1,M_2}})\rt)
\end{align}

Since the equation
$ \frac{1}{\sqrt{a}} + \frac{1}{\sqrt{b}} = 0$ has no solution for all $a,b \in \CC$, we have
$$ \lt[ 
\frac{1}{\sqrt{\det(-\Hess\Phi^+_{M_1,M_2}(W^\alpha_\beta;\vec{z}_{M_1,M_2}))}} 
+
\frac{1}{\sqrt{\det(-\Hess\Phi^-_{M_1,M_2}(W^\alpha_\beta;\vec{z}_{M_1,M_2}))}}  \rt] \neq 0 $$

This completes the proof of Theorem~\ref{mainthm2}. In particular, it implies Theorem~\ref{mainthm5} for $W^\alpha_\beta$. Corollary~\ref{cablingtv2} for $\SS^3 \backslash W^\alpha_\beta$ then follows directly from the method described in Section~\ref{strategy}.

\appendix
\section{Colored Jones polynomials under the Hopf union}

Recall that in the $N$-th Temperley-Lieb algebra $TL_N$ over $\ZZ[t^{\pm \frac{1}{4}}]$, the Jones-Wenzl idempotent $f_N$ satisfies the properties that for any $b \in TL_N$,  $f_N \cdot b = C_b f_N$ for some element $C_b \in \ZZ[t^{\pm \frac{1}{4}}]$ (see for example Lemma 13.2 in \cite{L97}). Suppose $D_K$ is a $1$-tangle such that the closure of $D_K$ is a knot diagram of the knot $K$. Then we have

\begin{align*}
\includegraphics[width=.25\linewidth]{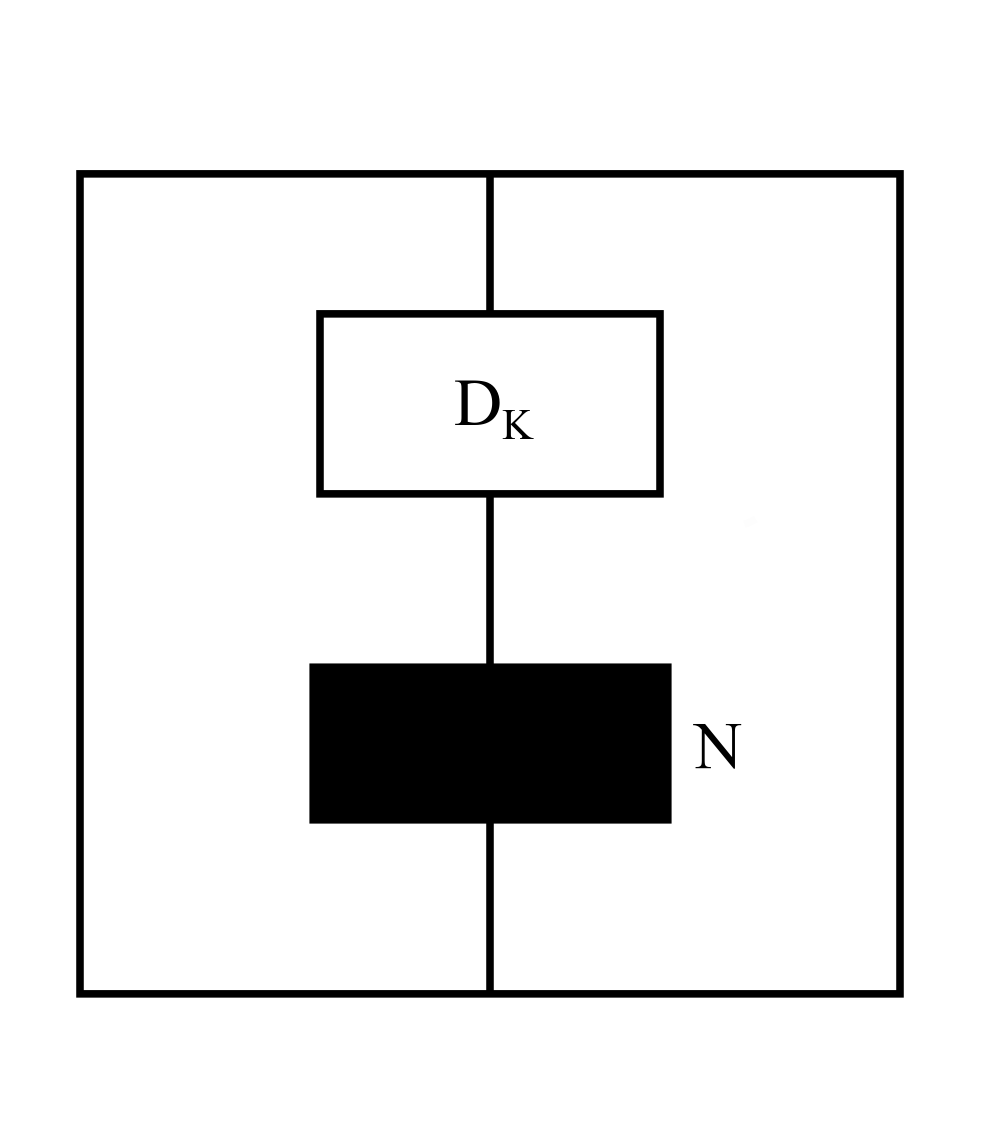} 
\raisebox{60pt}{= C} 
\includegraphics[width=.25\linewidth]{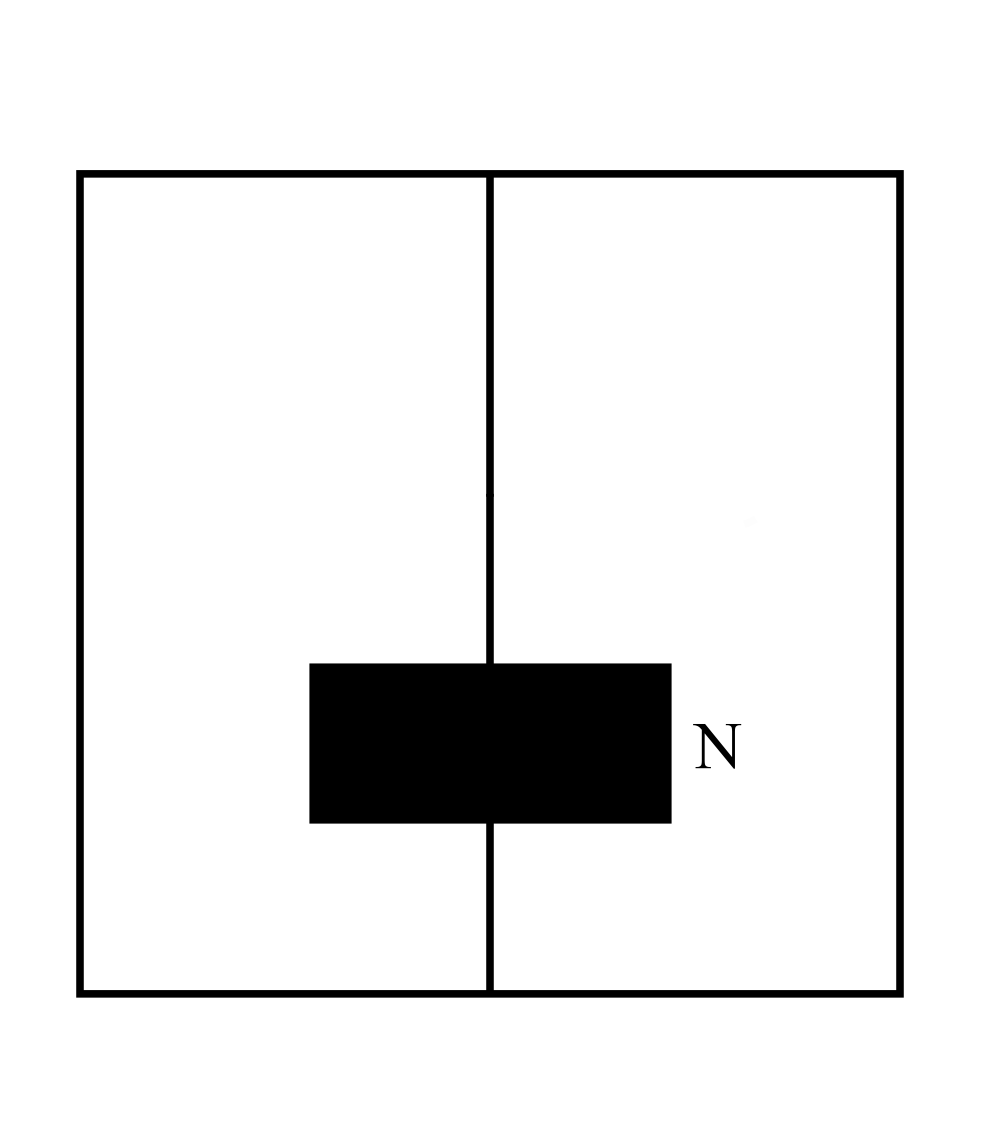} 
\end{align*}
where the black box is the $N$-th Jones Wenzl idempotent and $C$ is an element in $\ZZ[t^{\pm \frac{1}{4}}]$. By closing up the tangles on both sides, we get
$$ C = \frac{J_N(K, t)}{J_N(U,t)} = \frac{J_N(K, t)}{[N]} $$ 
where $U$ is the unknot and $[N]=(t^{\frac{N}{2}}-t^{-\frac{N}{2}})/(t^{\frac{1}{2}}-t^{-\frac{1}{2}})$. As a consequence,
\begin{align*}
|J_{M_1,M_2}( K_1 \# W^0_0 \# K_2, t) |
= \lt|J_{M_1,M_2}( W^0_0, t ) \frac{J_{M_1}(K_1,t)}{[M_1]}\frac{J_{M_2}(K_1,t)}{[M_2]} \rt|
\end{align*}
By Lemma 14.2 in \cite{L97}, we have
$$ |J_{M_1,M_2}( W^0_0, t )| = \lt| \frac{[M_1M_2]}{[M_1]} J_{M_1}(U,t) \rt| = | [M_1 M_2] | $$
This completes the proof of Lemma~\ref{CJHUlemma}.

\vspace{10pt}
\begin{minipage}[t]{0.5\textwidth}
   Ka Ho Wong\\
   Department of Mathematics,\\
   Texas A\&M University,\\
   College Station, \\
   Texas, United States\\
   daydreamkaho@math.tamu.edu\\
\end{minipage}


\begin{thebibliography}{100}

\bibitem{AH06}
  J. E. Andersen and S. K. Hansen,
  `Asymptotics of the quantum invariants for surgeries on the figure 8 knot',
  J. Knot Theory Ramifications 15 (2006) 479-548

\bibitem{BHMV95}
 C. Blanchet, N. Habegger, G. Masbaum, P. Vogel,
 `Topological Quantum Field Theories derived from the Kauffman bracket',
 Topology Volume 34, Issue 4, 1995, Pages 883-927

\bibitem{CY15}
  Q. Chen and T. Yang,
  `Volume conjectures for the Reshetikhin-Turaev and the Turaev-Viro invariants',
  Quantum Topology, 9 (2018) no. 3, 419-140. 

\bibitem{C07}
  Francesco Costantino,
  `6$j$-symbols, hyperbolic structures and the volume conjecture'
  Geom. Topol., Volume 11, Number 3 (2007), 1831-1854.
  
\bibitem{D18}
  R. Detcherry,
  `Growth of Turaev-Viro invariants and cabling',
   to appear in Journal of Knot Theory and Its Ramifications

\bibitem{DK19}
  R. Detcherry and E.Kalfagianni,
  `Gromov norm and Turaev-Viro invariants of 3-manifolds',
   to appear in  Ann. Sci. Ecole Norm. Sup
  
\bibitem{DKY17}
  R. Detcherry, E.Kalfagianni and T.Yang,
  `Turaev-Viro invariants, colored Jones polynomials and volume',
   Quantum Topol. 9 (2018) 775

\bibitem{F95}
  L. Faddeev, 
  `Discrete Heisenberg-Weyl group and modular group',
  Lett. Math. Phys. 34 (1995), no. 3, 249–254.

\bibitem{FKV01}
  L. Faddeev, R. Kashaev and A. Volkov, 
  `Strongly coupled quantum discrete Liouville theory, I',
  Algebraic approach and duality. Comm. Math. Phys. 219 (2001), no. 1, 199–219.

\bibitem{K20}
  Sanjay Kumar, 
  `Fundamental shadow links realized as links in $\SS^3$',
  https://arxiv.org/abs/2005.11447

\bibitem{G05}
  Sergei Gukov, 
  `Three-dimensional quantum gravity, Chern-Simons theory, and the A polynomial',
  Commun. Math. Phys. 255 (2005), 577-627.
  
\bibitem{GM08}
  Sergei Gukov, Hitoshi Murakami, 
  `SL(2, C) Chern-Simons theory and the asymptotic behavior of the colored Jones polynomial, 
   H. Lett Math Phys (2008) 86: 79. https://doi.org/10.1007/s11005-008-0282-3
   
\bibitem{K97}
  R. M. Kashaev,
  `The hyperbolic volume of knots from the quantum dilogarithm',
   Lett. Math. Phys. 39(1997) 269-275
    
\bibitem{L97}
    W. B. R. Lickorish, `An introduction to knot theory', Graduate Texts in Mathematics, vol. 175, Springer-Verlag, New York, 1997.

\bibitem{L81}
    D. London, `A note on matrices with positive definite real part', Proc. Amer. Math. Soc. 82 (1981), no. 3, 322-324.

\bibitem{LY18}
  F. Luo, T. Yang
  `Volume and rigidity of hyperbolic polyhedral 3-manifolds',
   Journal of Topology, 11(2018) 1-29
      
\bibitem{MY07}
  H. Murakami and Y. Yokota,
  `The colored Jones polynomials of the figure eight knot and its Dehn surgery spaces,
  J. reine angew. Math. 607 (2007) 47-68

\bibitem{HM11}
  H. Murakami,
  `An introduction to the volume conjecture',
  Interactions between hyperbolic geometry, quantum topology and number theory, Contemporary Mathematics 541 (American Mathematical Society,
Providence, RI, 2011) I-40.

\bibitem{HM13}
  H. Murakami,
  `The coloured Jones polynomial, the Chern-Simons invariant, and the Reidemeister torsion of the figure-eight knot', J. Topol. 6 (2013), no. 1, 193-216. MR 3029425

\bibitem{MM01}
  H. Murakami and J. Murakami,
  `The colored Jones polynomials and the simplicial volume of a knot',
  Acta Math. 186 (2001) 85-104.

\bibitem{MMOTY02}
 H. Murakami, J. Murakami, M. Okamoto, T. Takata, Y. Yokota, 
 'Kashaev’s Conjecture and the Chern-Simons Invariant of Knots and Links',
 Exp. Math. 11, 3 (2002) 427-435.
 
\bibitem{NZ85}
  W. Neumann and D. Zagier,
  `Volumes of hyperbolic three-manifolds'
  Topology, Volume 24, Issue 3, 1985, Pages 307-332
  
\bibitem{O52}
 Tomotada Ohtsuki,
 `On the asymptotic expansion of the Kashaev invariant of the $5_2$ knot',   Quantum Topology. 7. 669-735. 

\bibitem{T79}
 William P. Thurston,
 `The geometry and topology of three-manifolds',
  Princeton Univ. Math. Dept. Notes, 1979, Available at http://www.msri.org/communications/books/gt3m.

\bibitem{V08}
  Roland Van Der Veen,
  `Proof of the volume conjecture for Whitehead chains',
  ACTA MATHEMATICA VIETNAMICA, Volume 33, Number 3, 2008, pp. 421-431

\bibitem{V09}
  Roland Van Der Veen,
  `The volume conjecture for augmented knotted trivalent graph',
Algebr. Geom. Topol., Volume 9, Number 2 (2009), 691-722.

\bibitem{WA17}
  Ka Ho Wong and Thomas Kwok-Keung Au,
  `Asymptotic Behavior of the Colored Jones polynomials and Turaev-Viro Invariants of the figure eight knot', to appear in Algebraic \& Geometric Topology

\bibitem{W19}
  Ka Ho Wong, 
  `Asymptotics of some quantum invariants of the Whitehead chains', 
  arXiv:1912.10638

\bibitem{WY20}
  Ka Ho Wong and Tian Yang,
  `On the Volume Conjecture for hyperbolic Dehn-filled 3-manifolds along the figure-eight knot', arXiv:2003.10053

\bibitem{YY10}
  Mayuko Yamazaki \& Yoshiyuki Yokota. 
  `On the Limit of the Colored Jones Polynomial of a Non-simple Link'
  , Tokyo Journal of Mathematics. 33. 2010, 10.3836/tjm/1296483487. 

\bibitem{HZ07}
  Hao Zheng,
  `Proof of the Volume Conjecture for Whitehead Doubles of a Family of Torus Knots'
  Chinese Annals of Mathematics, Series B, 2007, 375-388


\end{thebibliography}
\end{document}